\documentclass[a4paper,reqno,12pt]{amsart}

\usepackage{cite}
\usepackage{amssymb,amsthm,graphicx,bm,bbm,mathrsfs,mathtools,exscale,amsmath,url,color,extarrows}
\usepackage{geometry,multicol,rotfloat,MnSymbol} 
\usepackage{dcolumn}
\usepackage{comment,verbatim}
\usepackage{bold-extra,longtable,mdwlist}
\usepackage[colorlinks,linkcolor=blue,urlcolor=black]{hyperref}
\usepackage[font=small, labelfont=bf, width=0.8\textwidth]{caption}
\usepackage{physics}
\usepackage{enumitem}
\usepackage{array,booktabs}
\usepackage{accents}
\usepackage{arydshln}
\usepackage{subcaption}
\usepackage[normalem]{ulem}
\usepackage{mathtools}
  \mathtoolsset{showonlyrefs}
\usepackage{epstopdf}
\usepackage{tikz}
\usepackage{pstricks} 
\usetikzlibrary{graphs,patterns,decorations.markings,arrows,matrix}
\allowdisplaybreaks
\newtheorem{thm}{Theorem}
\newtheorem{lm}[thm]{Lemma}
\newtheorem{cor}[thm]{Corollary}
\newtheorem{prop}[thm]{Proposition}

\theoremstyle{definition}
\newtheorem{remark}[thm]{Remark}
\newtheorem{defn}[thm]{Definition}

\numberwithin{thm}{section}

\newcommand{\Rmnum}[1]{uppercase\expandafter{\romannumeral #1}}

\newcommand{\vir}{\dagger}
\newcommand{\virst}{\!\!\text{\normalfont\dagger}}
\newcommand{\conv}{\star}
\newcommand{\fconv}{\diamond}
\newcommand{\contour}{\beta}

\newcommand*{\given}{\;%
	\ifnum\currentgrouptype=16\middle|\else%
	\iftoggle{WithinBracMacro}{\middle|}{|}%
	\fi%
	\;}%

\DeclareMathOperator{\pf}{Pf }

\def\ii{{\rm i}}
\def\dd{{\rm d}}

\let\oldast\ast
\renewcommand{\ast}{\!\oldast\!}

\numberwithin{equation}{section}

\DeclareMathOperator*{\res}{Res}

\DeclareMathOperator{\diag}{diag}

\DeclareMathOperator{\Pf}{Pf} 
\DeclareMathOperator{\id}{id} 

\def\be{\begin{equation}}
\def\ee{\end{equation}}

\title{Pfaffian Transition Probability and Correlation Kernel of TASEP in Half-space}
\author{Jan de Gier}
\address{School of Mathematics and Statistics, University of Melbourne, Victoria 3010, Australia}
\email{\href{mailto:jdgier@unimelb.edu.au}{jdgier@unimelb.edu.au}}
\author{William Mead}
\address{School of Mathematics and Statistics, University of Melbourne, Victoria 3010, Australia}
\email{\href{mailto:wmead@student.unimelb.edu.au}{wmead@student.unimelb.edu.au}}
\author{Daniel Remenik}
\address{
  Departamento de Ingenier\'ia Matem\'atica and Centro de Modelamiento Matem\'atico (IRL-CNRS 2807)\\
  Universidad de Chile\\
  Av. Beauchef 851, Torre Norte, Piso 5\\
  Santiago\\
  Chile} \email{dremenik@dim.uchile.cl}
\author{Michael Wheeler}
\address{School of Mathematics and Statistics, University of Melbourne, Victoria 3010, Australia}
\email{\href{mailto:wheelerm@unimelb.edu.au}{wheelerm@unimelb.edu.au}}

\begin{document}

\begin{abstract}
    We present the transition probability for the asymmetric simple exclusion process on the half-space for general initial conditions and particle insertion at the boundary. In the limit of total asymmetry, where particles only jump to the right, we show that the transition probability reduces to a single Pfaffian for a large class of initial conditions. We further derive a Fredholm Pfaffian formula for the multipoint distribution of the totally asymmetric model conditioned on the number of particles in the system using a connection with a Pfaffian point process supported on Gelfand--Tsetlin patterns.
\end{abstract}

\maketitle 


\setcounter{tocdepth}{1}
\tableofcontents

\section{Introduction}

\subsection{Background}
The asymmetric simple exclusion process (ASEP) is an integrable continuous-time Markov chain central to the KPZ universality class. It is a specialization of the stochastic six-vertex model which can be regarded as a discrete time stochastic process.

Literature on the ASEP is extensive. Its integrability is utilized in \cite{GwaSpohn1992,Kim1995} to study the spectral gap of the transition matrix as well as universal scaling behaviour by solving Bethe ansatz equations for a finite lattice with periodic boundary conditions, see also \cite{Golinelli_2004,Lee_2006,baik_fluctuations_2018}. Similar results for the more intricate case of a finite lattice with open boundaries are given in  
\cite{GierEssler2005,Gier_2006,Gier_2008,Simon_2009,PhysRevLett.107.010602, PhysRevLett.109.170601, Lazarescu_2014,Prolhac2016, Prolhac_2024}.

On the infinite line it is possible to make significant further progress because one does not need to solve Bethe equations. The transition probability for the totally asymmetric simple exclusion process (TASEP) is given in determinant form by \cite{schutz_exact_1997} using coordinate Bethe ansatz, and in terms of multiple contour integrals for the general ASEP by \cite{tracy_integral_2008}. The fact that the asymptotic behaviour of related stochastic particle models can be described using random matrix distribution functions was established in \cite{johansson_shape_2000}.

The determinant form in \cite{schutz_exact_1997} is not immediately amenable for asymptotic analysis, but can be rewritten using a determinantal point process supported on Gelfand--Tsetlin patterns  \cite{sasamoto_spatial_2005,borodin_fluctuation_2007}. The correlation kernel for this point process is then expressed in terms of the solution of a biorthogonalization [problem] that can be solved for general boundary conditions, leading to the KPZ fixed point \cite{matetski_kpz_2021}. In a further development, a multi-time,  formula has been solved in \cite{JohanssonRahman,Liu-multi-point} for some special initial conditions.

For the ASEP on the infinite half-line, the transition probability is given in the case of reflective boundary conditions by \cite{tracy_bose_2013} and for general open boundaries for the discrete stochastic six-vertex model and continuous-time ASEP in half-space in \cite{garbali_symmetric_2025} in the case of empty initial conditions.

Half-space and other generalizations of \cite{johansson_shape_2000} have been studied in \cite{baik_algebraic_2001,baik_asymptotics_2001,sasamoto_fluctuations_2004,baik_pfaffian_2018}, while half-space Airy line ensembles have recently been analyzed in \cite{dimitrov_half-space_2025}. A structural analysis of Macdonald processes is developed in \cite{barraquand_half-space_2020}, including integral formulas for moments and Laplace transforms of Macdonald measures and various specializations. Relationships between several half-space models and related distributions using shift invariance are established in \cite{he_shift_2025}. Formulas and KPZ scaling analyses for the expectation of the height function and integrated current were obtained for special choices of open boundary conditions for both ASEP and stochastic six-vertex model in \cite{barraquand_stochastic_2018} (off-diagonal) and \cite{barraquand_markov_2024} (Liggett's condition).  

A KPZ scaling analysis for the boundary height function for the case of general boundary conditions is given in \cite{he_boundary_2024}. A  probability distribution of the totally asymmetric simple exclusion process (TASEP) in half-space for generic initial data, with its KPZ scaling analysis and subsequent description of the half-space KPZ fixed point, is obtained in \cite{zhang_tasep_2024}.

In this paper we provide the half-space ASEP transition probability for arbitrary initial conditions and without absorption. We then turn our attention to the TASEP for which the transition probability takes a Pfaffian form. This Pfaffian can be related to a Pfaffian point process supported on Gelfand--Tsetlin patterns \cite{borodin_eynardmehta_2005,betea_free_2018}, which allows us to establish the  joint distribution conditioned on particle number. The relationship with the unconditional  distribution of \cite{zhang_tasep_2024} is intriguing and at present unclear.

\subsection{Main results}
In Section~\ref{sec:ASEP} we define the asymmetric simple exclusion process (ASEP) in half-space with open boundary conditions and finitely many particles at positions $1\leq x_N < \cdots<x_1$ obeying the following dynamics. A particle at site $x\geq1$ jumps to $x+1$ at rate $1$ and to $x-1$ (assuming $x\geq2$) at rate $q\geq0$, with jumps being permitted only if the target site is empty. Additionally, at rate $\alpha\geq0$ new particles enter the system at site $1$ whenever this site is empty and exit the system through site one at rate $\gamma\geq0$.

Our first main result is an explicit multiple integral formula for the transition probability of the half-space ASEP with general initial conditions and $\gamma=0$. 

\begin{thm}[Theorem~\ref{thm:ASEP transition prob} in the main text]
\label{th:transprob_intro}
	Let $N\geq M\geq 0$ be fixed integers and let $x=(x_1,\dots,x_N)$ and $y=(y_1,\dots,y_M)$ be ordered coordinates of particles in the half-space ASEP. The transition probability between $y$ and $x$ for the half-space ASEP under the specialization $\gamma=0$ is given by
	\begin{multline}
		\mathbb{P}_t(y\to x) = V_{N-M} \alpha^{N-M} \mathrm{e}^{-\alpha t} \oint_{\mathcal{C}_1} \frac{\dd w_1}{2\pi\ii} \cdots \oint_{\mathcal{C}_N} \frac{\dd w_N}{2\pi\ii} \prod_{1\leq i<j\leq N}\left[\frac{w_j-w_i}{qw_j-w_i}\frac{1-qw_iw_j}{1- w_iw_j}\right] \\
		\times\prod_{i=1}^N \left[\frac{1-q w_i^2}{w_i(q+\alpha-1-\alpha w_i)(1-qw_i)} \left(\frac{1-w_i}{1-qw_i}\right)^{x_i-1} \exp(\frac{(1-q)^2 w_i t}{(1-w_i)(1-qw_i)})\right] \\
		\times \sum_{\sigma\in\mathcal{B}_N} \sigma\left(\prod_{1\leq i<j\leq N}
		\left[\frac{w_i-qw_j}{w_i-w_j}\frac{1-w_iw_j}{1- qw_iw_j}\right] 	\prod_{i=1}^M \left[ \frac{1-q-\alpha+\alpha w_i}{1-qw_i^2}\frac{(1-q)w_i}{1-w_i}\left(\frac{1-qw_i}{1-w_i}\right)^{y_i-1}\right]
		\right),
	\end{multline}
	where the contours surround singularities at $w_i=1$ and are defined according to Definition \ref{defn:ASEP nested contours}, and where $\mathcal{B}_N$ is the group of signed permutations of size $\{1,\dotsc,N\}$ (see Section \ref{sec:int formula} and \eqref{eq:sigma action} for more details). The overall normalization is given by the following $q$-dependent factor
    \begin{equation}
        V_k = \dfrac{q^{\binom{k}{2}}(1-q)^k}{\prod_{i=1}^k\left[\left(1-q^i\right)\left(1+q^{i-1}\right)\right]}.
    \end{equation}
\end{thm}
The case of empty initial conditions, presented in \cite{garbali_symmetric_2025}, is recovered from \eqref{eq:ASEP transition prob} as the special case $M=0$ or $y=\emptyset$. In particular, this special case has a completely factorized integrand. 

\begin{remark}
As discussed further in Section~\ref{sec:ASEP}, Theorem~\ref{th:transprob_intro} was originally derived within the context of the stochastic six vertex model on half space as defined in \cite{garbali_symmetric_2025}, using a related Cauchy identity. In this paper we provide an alternative direct proof, presented in Sections~\ref{sec:ASEP} and \ref{se:ASEPproof}. 
\end{remark}

We turn now to the totally asymmetric (TASEP) limit $\gamma=q=0$, for which the transition probability takes a Pfaffian form akin to Sch\"utz's determinantal formula \cite{schutz_exact_1997} for the full-space case. This is our second main result and we state here the simplest version, for empty initial conditions and even number of particles.

\begin{thm}[See Theorems~\ref{thm: transition prob pfaffian kernel} and \ref{thm: transition prob pfaffian kernel initial cond} in the main text]
	Let $N$ be a fixed even integer and let $x=(x_1,\dots,x_N)$ be the ordered coordinates of half-space TASEP particles. Then the half-line open TASEP has transition probability from an empty state to $x$ is given by the Pfaffian
	\begin{equation}
		\mathbb{P}_t(\emptyset\to x) = (-1)^{\binom{N}{2}} \mathrm{e}^{-\alpha t} \Pf \left[\mathsf{Q}\right],
	\end{equation}
    where $\mathsf{Q}$ is a skew-symmetric $N\times N$-dimensional matrix whose entries are given by
    \begin{align*}
        [\mathsf{Q}]_{i,j} := & Q_{i,j}(x_{N-i+1},x_{N-j+1}), 
    \end{align*}
    for all $1\leq i,j\leq N$ and
	\begin{equation}
        \label{eq:Qintro}
		Q_{i,j}(x,y) = \alpha^2 \oint_\contour \frac{\dd u}{2\pi\ii} \oint_\contour \frac{\dd w}{2\pi\ii} \frac{u-w}{1-u-w} \frac{w^{i-x}\mathrm{e}^{t(w-1)}}{(w-\alpha)(w-1)^i} \frac{u^{j-y}\mathrm{e}^{t(u-1)}}{(u-\alpha)(u-1)^j},
	\end{equation}
    where the contour $\beta$ surrounds the points $1,0,\alpha$ and $1-\alpha$.
 \end{thm}   
The extension to odd $N$ is straightforward and also provided in Theorem~\ref{thm: transition prob pfaffian kernel}. More generally, a Pfaffian expression can be obtained for arbitrary initial conditions, as long as the leftmost initial particle is not too close to site $1$. For fixed integers $N\geq M\geq 0$, let $x$ be the ordered coordinates $1\leq x_N<\ldots<x_1$ of TASEP particles at time $t$ and let $y$ denote the ordered coordinates $1\leq y_M<\ldots<y_1$ of particles at time $t=0$, then the transition probability $\mathbb{P}_t(y\to x)$ is given in Theorem~\ref{thm: transition prob pfaffian kernel initial cond} as a Pfaffian for the case\footnote{The restriction to $y_M > N-M+1$ is of a technical nature. For a physical interpretation see Remark \ref{rmk:initial condition restriction}.} $y_M > N-M+1$. 

Following ideas of \cite{sasamoto_spatial_2005, borodin_fluctuation_2007} for the full-line case, and adjusting these to the Pfaffian case using the framework developed in \cite{borodin_eynardmehta_2005} and \cite{rains_correlation_2000}, the above Pfaffian expression for the transition probability can be written in terms of a Pfaffian point processes supported on Gelfand--Tsetlin patterns, see Propositions \ref{prop:GT pattern empty} and \ref{prop:GT pattern initial} in the main text. 

Using the formulation as a Pfaffian point process, correlation kernels can be computed via a result from \cite{borodin_eynardmehta_2005} given in Proposition \ref{prop:cond pf L is PPP}. This leads us to our third main result: a conditional  distribution of half-space TASEP, which we present here for the simplest case of empty initial conditions. A Fredholm Pfaffian expression for the conditional  distribution for more general initial conditions is given in Theorem~\ref{thm:cond prob Fredholm pf initial} in Section~\ref{se:GeneralIC}.

\begin{thm}[Theorem~\ref{thm:cond prob Fredholm pf empty} in the main text]
\label{thm:cond prob Fredholm pf empty intro}
    Let $N>0$ be a fixed even integer. Let $\Psi(x,y)=Q_{1,1}(x,y)$ with $Q$ defined in \eqref{eq:Qintro}, and let $\Phi_0(x),\dots,\Phi_{N-1}(x)$ be a family of polynomials where $\Phi_k$ is of degree $k$. Assume that these polynomials satisfy the following skew-biorthogonality relations:
    \begin{align}
        \begin{split}
            \sum_{x,y=1}^\infty \Phi_{2i}(x) \Psi(x,y) \Phi_{2j}(y) & = 0, \\
            \sum_{x,y=1}^\infty \Phi_{2i+1}(x) \Psi(x,y) \Phi_{2j+1}(y) & = 0, \\
            \sum_{x,y=1}^\infty \Phi_{2i}(x) \Psi(x,y) \Phi_{2j+1}(y) & = -\delta_{i,j},
        \end{split}
    \end{align}
    for all $0\leq i,j< N/2$. Furthermore, for a fixed integer $m\geq 0$, let $\{p_1,\dots,p_m\}\subseteq\{1,\dots,N\}$ be indices denoting particle labels such that $p_1<p_2<\cdots<p_m$. Then the half-line TASEP with empty initial conditions has conditional joint distribution given by a Fredholm Pfaffian of the form
    \begin{equation}
        \mathbb{P}\left[\bigcap_{k=1}^m \left\{X_{t}(p_k) > a_k\right\} \given \abs{X_t} = N\right] = \pf\left(J - \bar\chi_a K\bar\chi_a\right)_{\ell^2(\{p_1,\dots,p_m\}\times \mathbb{N})},
    \end{equation}
    where $\bar\chi_a(p,x) = \mathbbm{1}_{x\leq a_p}$ and $a_1,\dotsc,a_k\geq0$. The correlation kernel $K$ is given explicitly in terms of the $\Phi_{i}(x)$ polynomials in Section~\ref{se:EmptyIC}.
\end{thm}  

This result, and its extension to more general initial data, provides a Fredholm Pfaffian for the TASEP joint distribution conditioned on there being $N$ particles in the system at time $t$. Our approach is constructive starting from the ASEP transition probability, and the conditioning arises naturally in our setting from the Pfaffian point process formalism. Via a completely different approach, which essentially amounts to guessing a formula and then checking that it satisfies the Kolmogorov equations of the system, X. Zhang \cite{zhang_tasep_2024} has recently obtained a Fredholm Pfaffian expression for the \textsl{unconditional} joint distribution of the TASEP height function, and is able to extract the KPZ scaling limit from it. It is of significant interest to understand how this and other approaches are related.

The skew-biorthogonalization problem in Theorem~\ref{thm:cond prob Fredholm pf empty intro} seems to be more challenging than the full-line case, even for empty initial data. We suspect that one reason for this is related to the fact that the joint distribution is conditioned on the number of particles in the system at time $t$.

\subsection{Acknowledgments}
We warmly thank Guillaume Barraquand, Nick Beaton, Peter J. Forrester, Alexandr Garbali, Cengiz Gazi, Jimmy He, Axel Saenz and Tomohiro Sasamoto for helpful discussions. JdG and MW acknowledge support from the Australian Research Council. WM was supported by an Australian Government Research Training Program Scholarship.
DR was supported by Centro de Modelamiento Matem\'{a}tico Basal Funds FB210005 from ANID-Chile and by Fondecyt Grant 1241974. Part of this research was performed while the authors were visiting the Institute for Pure and Applied Mathematics (IPAM), which is supported by the National Science Foundation (Grant No. DMS-1925919).

\section{Half-space ASEP}
\label{sec:ASEP}

\subsection{Model outline}

We consider the \emph{asymmetric simple exclusion process} (ASEP) in half-space $\mathbb{N}\coloneqq\mathbb{Z}_{>0}$ with open boundary conditions, and with finitely many particles.
This is a continuous time Markov process with state space given by 
\[\mathbb{W}= \left\{S\subset \mathbb{N}: S \text{ is finite}\right\}\]
(note that $\mathbb{W}$ is countable).
We index these half-space configurations by $x=(x_1,\dots,x_N)\in\mathbb{W}$ satisfying $x_1> \cdots>x_N\geq 1$, and we think of the state $X_t = (X_t(1),\dots,X_t(N))$ of the process at time $t$ as the positions (in decreasing order) of $N$ particles present at that time $t$.

The dynamics of the process are as follows: a particle at site $x\geq1$ jumps to $x+1$ at rate $1$ and to $x-1$ (assuming $x\geq2$) at rate $q\geq0$, with jumps being permitted only if the target site is empty and, additionally, at rate $\alpha\geq0$ new particles enter the system at site $1$ whenever this site is empty and exit the system through site one at rate $\gamma\geq0$ (see Figure \ref{fig:ASEP dynamics figure}).
More precisely, the evolution of the system can be separated into \emph{bulk dynamics}, with transition rates
\begin{align*}
	x \mapsto (x_1,\dots,x_i+1,\dots,x_N) & \quad \text{ at rate $1$} \quad \text{if } x_{i-1}>x_i +1, \\
	x \mapsto (x_1,\dots,x_i-1,\dots,x_N) & \quad \text{ at rate $q$} \quad \text{if } x_{i+1}<x_i -1\quad(x_i\geq2),
\end{align*}
and \emph{boundary dynamics}, with transition rates
\begin{align*}
(x_1,\dots,x_N) \mapsto (x_1,\dots,x_N,1) & \quad \text{ at rate $\alpha$} \quad \text{if } N=0 \text{ or } x_{N}>1, \\
(x_1,\dots,x_N) \mapsto (x_1,\dots,x_{N-1}) & \quad \text{ at rate $\gamma$} \quad \text{if } x_{N}=1.
\end{align*} 
In this text we consider ASEP with generic open boundary conditions at the origin.
Note that the number of particles $\abs{X_t}$ is a random variable which evolves with $t$.

\begin{figure}
	\begin{center}
		\begin{tikzpicture}[scale=0.8]
			\foreach \x in {1,...,8}
			{\draw[line width=2pt] (\x-1,2) -- (\x-0.5,2);}

			\draw[black,fill=red] (2.25,2.4) circle (0.3);
			\draw[black,fill=red] (4.25,2.4) circle (0.3);
			\draw[black,fill=red] (5.25,2.4) circle (0.3);
			\draw[gray,fill=pink] (-0.65,2.7) circle (0.3);
			\draw[line width=1pt,->] (-0.35,3) arc (100:0:0.5);
			\node[above] at (0.15,2.8) {$\alpha$};
			\draw[line width=1pt,->] (2.05,2.7) arc (45:165:0.4);
			\node[above] at (1.75,2.8) {$q$};
			\draw[line width=1pt,->] (2.45,2.7) arc (135:15:0.4);
			\node[above] at (2.75,2.8) {$1$};
			\draw[line width=1pt,->] (4.05,2.7) arc (45:165:0.4);
			\node[above] at (3.75,2.8) {$q$};
			\draw[line width=1pt,->] (5.45,2.7) arc (135:15:0.4);
			\node[above] at (5.75,2.8) {$1$};

			\foreach \x in {1,...,8}
			{\draw[line width=2pt] (\x-1,0) -- (\x-0.5,0);
				\node[below] at (\x-0.75,0) {$\x$};}
			\draw[black,fill=red] (0.25,0.4) circle (0.3);
			\draw[black,fill=red] (2.25,0.4) circle (0.3);
			\draw[black,fill=red] (3.25,0.4) circle (0.3);
			\draw[black,fill=red] (4.25,0.4) circle (0.3);
			\draw[black,fill=red] (6.25,0.4) circle (0.3);
			\draw[line width=1pt,->] (0.05,0.7) arc (45:165:0.4);
			\node[above] at (-0.25,0.8) {$\gamma$};
			\draw[line width=1pt,->] (0.45,0.7) arc (135:15:0.4);
			\node[above] at (0.75,0.8) {$1$};
			\draw[line width=1pt,->] (2.05,0.7) arc (45:165:0.4);
			\node[above] at (1.75,0.8) {$q$};
			\draw[line width=1pt,->] (4.45,0.7) arc (135:15:0.4);
			\node[above] at (4.75,0.8) {$1$};
			\draw[line width=1pt,->] (6.05,0.7) arc (45:165:0.4);
			\node[above] at (5.75,0.8) {$q$};
			\draw[line width=1pt,->] (6.45,0.7) arc (135:15:0.4);
			\node[above] at (6.75,0.8) {$1$};
		\end{tikzpicture}
	\end{center}
	\caption{Dynamics of the ASEP in half-space $\mathbb{N}$ with generic open boundaries.}
    \label{fig:ASEP dynamics figure}
\end{figure}
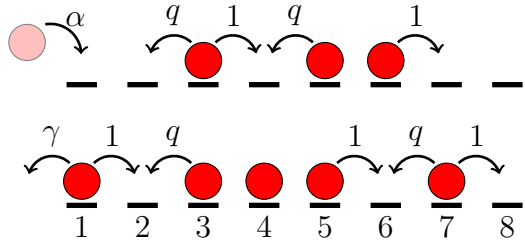
The ASEP Markov generator $\mathscr{L}$ acts on functions $f:\mathbb{W}\longrightarrow\mathbb{C}$ as
\begin{multline}
    \label{eq:ASEP generator}
    \mathscr{L}f(x) = \alpha(1-\eta_1)\left(f(x\cup\{1\})- f(x)\right) + \gamma\hspace{0.05em}\eta_1\left(f(x\setminus\{1\})-f(x)\right) \\
    +\sum_{s=1}^\infty 
     \eta_s(1-\eta_{s+1})(f(x^{s,s+1})-f(x))+\sum_{s=2}^\infty q\hspace{0.05em}\eta_s(1-\eta_{s-1})(f(x^{s,s-1})-f(x)),
\end{multline}
where $x^{s,s\pm1}\in\mathbb{W}$ is the configuration obtained by the interchange of occupations at sites $s$ and $s\pm1$, and where we have used the notation $\eta_s$ as the occupation of $x$ at site $s\in\mathbb{N}$:
\[\eta_s = \begin{cases}
    1 & \text{if }x_j=s\text{ for some } j, \\
    0 & \text{otherwise}.
\end{cases}\]
Let us define an inner product with an orthogonal basis indexed by half-space configurations $\mathbb{W}$ by $\bra{y}\ket{x} = \delta_{y,x}$.
With this inner product, the \emph{transition matrix elements} of the Markov generator $\mathscr{L}$ are expressed as $\bra{y}\mathscr{L}\ket{x}$. That is, when $y\neq x$ the matrix element $\bra{y}\mathscr{L}\ket{x}$ is equal to the possible single-particle transition rates to go from state $y$ to state $x$ (if there is no possible single particle transition from $y$ to $x$ then the matrix element is zero), while the diagonal terms are equally defined as $\bra{y}\mathscr{L}\ket{y} = -\sum_{x'\neq y} \bra{y}\mathscr{L}\ket{x'}$.

With this notation, the action of $\mathscr{L}$ on $f:\mathbb{W}\longrightarrow\mathbb{C}$ can be expressed as
\[\mathscr{L}f(y)=\sum_{z}\bra{x}\mathscr{L}\ket{z}f(z).\]

\subsubsection{ASEP transition probability}

In this section we present the transition probability for the half-space ASEP with generic initial conditions. We denote the transition probability by 
\begin{equation}
    \label{eq:trans prob notation defn}
    \mathbb{P}_t(y\to x) := \mathbb{P}\left(X_t = x|X_0=y\right).
\end{equation}
The transition probabilities $\mathbb{P}_t(y\to x)$ can be formulated as the solution of the \emph{master equation}, or \emph{Kolmogorov forward equation}
\begin{equation}
    \label{eq:ASEP evolution equation matrix forward}
    \frac{\dd}{\dd t} \mathbb{P}_t(y\to x) = \sum_{x'} \mathbb{P}_t(y\to x')\bra{x'}\mathscr{L}\ket{x},
\end{equation}
subject to the initial condition 
\[\mathbb{P}_0(y\to x) = \delta_{y,x}.\]
Explicit integral formula expressions for transition probabilities have proven to be fruitful for obtaining exact asymptotic results for the ASEP under various boundary and initial conditions conditions. Notably for the full-space $\mathbb{Z}$ (i.e. no boundary), solutions of the master equation were obtained for TASEP ($q=0$) \cite{schutz_exact_1997} and for generic bulk rates \cite{tracy_integral_2008} using coordinate Bethe ansatz techniques. In half-space, explicit expressions were similarly obtained for closed boundary conditions $(\alpha=\gamma=0)$ \cite{tracy_bose_2013}. For upper-triangular boundaries with empty initial conditions ($\gamma=0$ and $y=\emptyset$) expressions were obtained using a reduction of the stochastic six-vertex model \cite{garbali_symmetric_2025}.

The transition probability \eqref{eq:trans prob notation defn} is equivalently formulated as the unique\footnote{Since our state space $\mathbb{W}$ is countable, the solution of the backward Kolmogorov equation is unique whenever the process is not explosive (see e.g. Theorem 2.33 in \cite{liggett_continuous_2010}); non-explosivity holds in our case because the sequence of jump times is clearly dominated stochastically from below by a sequence of independent exponential random variables $\tau_k$ with parameters $|X_0|+k$, which satisfy $\sum_{k\geq1}\tau_k=\infty$ almost surely.} solution of the \emph{Kolmogorov backward equation}
\begin{equation}
    \label{eq:ASEP evolution equation matrix backward}
    \frac{\dd}{\dd t} \mathbb{P}_t(y\to x) = \sum_{y'} \bra{y}\mathscr{L}\ket{y'} \mathbb{P}_t(y'\to x),
\end{equation}
with the same initial condition.
Here we will employ \eqref{eq:ASEP evolution equation matrix backward} rather than \eqref{eq:ASEP evolution equation matrix forward} to derive a formula for $\mathbb{P}_t(y\to x)$.

\subsection{Integral formula for the transition probability}\label{sec:int formula}
For a fixed integer $N>0$ we denote the group of \emph{signed permutations} of $\{1,\dots,N\}$ by $\mathcal{B}_N=S_N\ltimes\{\pm1\}^N$. This group is of order $2^N N!$ and is sometimes known as the \emph{hyperoctahedral group}. The group $\mathcal{B}_N$ acts on functions of some alphabet $(w_1,\dots,w_N)$ of length $N$ by permuting the variables as
\begin{equation}\label{eq:sigma action}
    \sigma(f(w_1,\dots,w_N)) = f\left(w_{\sigma(1)},\dots,w_{\sigma(N)}\right),
\end{equation}
where the negative indices are defined by $w_{-k} := 1/(qw_k)$ for any $1\leq k\leq N$. Wherever it is convenient, we will use either $\sigma_i$ or $\sigma(i)$ to denote the image a signed permutation. The symmetric group of permutations, $S_N$, is contained within $\mathcal{B}_N$ as a subgroup.

In order to state the transition probability as a complex contour integral we must explicitly outline the domain of integration.
\begin{defn}
    \label{defn:ASEP nested contours}
    For a fixed integer $N$, let $\mathcal{C}_1,\dots,\mathcal{C}_N$ be a collection of positively oriented closed complex integration contours satisfying:
    \begin{itemize}
        \item For all $1\leq i\leq N$, each contour $\mathcal{C}_i$ encloses the point 1 and does not enclose the points $0,q^{-1}$ and $(1-q-\alpha)/\alpha$.
        \item For all $1\leq i<j\leq N$, each $\mathcal{C}_i$ is completely contained within the interior of each $\mathcal{C}_j$. Additionally, $q\mathcal{C}_j$ lies completely outside the interior of $\mathcal{C}_i$, where $q\mathcal{C}_j$ denotes the image of $\mathcal{C}_j$ under multiplication by $q$. This also implies that no part of $\mathcal{C}_j$ is contained within the interior of $q^{-1}\mathcal{C}_i$.
        \item For all $1\leq k,\ell\leq N$, the contours $\mathcal{C}_k$ and $q^{-1}\mathcal{C}_\ell^{-1}$ do not intersect and have completely disjoint interiors. The contour $\mathcal{C}_\ell^{-1}$ is the contour which traverses the points which are the reciprocal of those on $\mathcal{C}_\ell$, and whose orientation is fixed so that $q^{-1}\mathcal{C}_\ell^{-1}$ encloses the point $q^{-1}$.
    \end{itemize}
\end{defn}
\begin{figure}
    \centering
    \[
    \begin{tikzpicture}[scale=1.2]
        \draw[lightgray,line width=1pt] (-1,0) -- (8.5,0);
        \draw[lightgray,line width=1pt] (0,-2.2) -- (0,2.2);
        \node[below] at (-0.5,0) {$\frac{1-q-\alpha}{
        \alpha}$};
        \draw[black,fill=black] (-0.5,0) circle (0.07);
        \node[below] at (8,0) {$q^{-1}$};
        \draw[black,fill=black] (8,0) circle (0.07);
        \node[below] at (5,0) {$1$};
        \draw[black,fill=black] (5,0) circle (0.07);
        \node[below] at (1.5,0) {$q$};
        \draw[black,fill=black] (1.5,0) circle (0.07);
        \draw[line width=1.5pt,->] (5,-0.67) arc (-90:325:0.67);
        \draw[line width=1.5pt,->] (4.8,-1.2) arc (-90:305:1.2);
        \draw[line width=1.5pt,->] (4.5,-1.77) arc (-90:295:1.77);
        \draw[line width=0.8pt,dashed] (1.5,-0.45) arc (-90:270:0.45);
        \draw[line width=0.8pt,dashed] (1.35,-0.85) arc (-90:270:0.85);
        \draw[line width=0.8pt,dotted] (7.85,1.82) arc (120:240:2.13);
        \node[above] at (4.2,0.2) {$\mathcal{C}_1$};
        \node[above] at (3.8,0.8) {$\mathcal{C}_2$};
        \node[above] at (3.2,1.3) {$\mathcal{C}_3$};
        \node at (-0.5,1.8) {$\mathbb{C}$};
    \end{tikzpicture}
    \]
    \caption{Examples of nested complex contours $\mathcal{C}_1,\mathcal{C}_2,\mathcal{C}_3$ satisfying the restrictions of Definition \ref{defn:ASEP nested contours} for some $0<q<1$. The contours $q\mathcal{C}_1,q\mathcal{C}_2$ depicted with dashed lines while part of the contour $q^{-1}\mathcal{C}^{-1}_3$ is depicted with a dotted line. This example depicts parameters satisfying $\alpha>1-q$, however, this is not a requirement.
    }
    \label{fig:ASEP nested contours}
\end{figure}
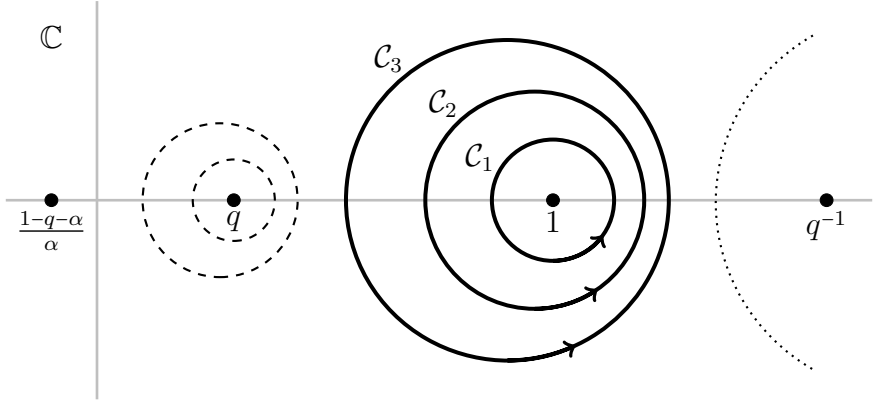
A depiction of nested contours which satisfy the the restrictions of Definition \ref{defn:ASEP nested contours} is given in Figure \ref{fig:ASEP nested contours}.

We now present the main result of this section: an explicit multiple integral formula for the transition probability of the half-space ASEP.
\begin{thm}
	\label{thm:ASEP transition prob}
	Let $N\geq M\geq 0$ be fixed integers and let $x=(x_1,\dots,x_N)$ and $y=(y_1,\dots,y_M)$ be ordered coordinates of particles in the half-space ASEP. The transition probability between $y$ and $x$ for the half-space ASEP under the specialization $\gamma=0$ is given by
	\begin{multline}
		\label{eq:ASEP transition prob}
		\mathbb{P}_t(y\to x) = V_{N-M} \alpha^{N-M} \mathrm{e}^{-\alpha t} \oint_{\mathcal{C}_1} \frac{\dd w_1}{2\pi\ii} \cdots \oint_{\mathcal{C}_N} \frac{\dd w_N}{2\pi\ii} \prod_{1\leq i<j\leq N}\left[\frac{w_j-w_i}{qw_j-w_i}\frac{1-qw_iw_j}{1- w_iw_j}\right] \\
		\times\prod_{i=1}^N \left[\frac{1-q w_i^2}{w_i(q+\alpha-1-\alpha w_i)(1-qw_i)} \left(\frac{1-w_i}{1-qw_i}\right)^{x_i-1} \exp(\frac{(1-q)^2 w_i t}{(1-w_i)(1-qw_i)})\right] \\
		\times \sum_{\sigma\in\mathcal{B}_N} \sigma\left(\prod_{1\leq i<j\leq N}
		\left[\frac{w_i-qw_j}{w_i-w_j}\frac{1-w_iw_j}{1- qw_iw_j}\right] 	\prod_{i=1}^M \left[ \frac{1-q-\alpha+\alpha w_i}{1-qw_i^2}\frac{(1-q)w_i}{1-w_i}\left(\frac{1-qw_i}{1-w_i}\right)^{y_i-1}\right]
		\right),
	\end{multline}
	where the contours surround singularities at $w_i=1$ and are defined according to Definition \ref{defn:ASEP nested contours}. The overall normalization is given by the following $q$-dependent factor
    \begin{equation}
        \label{eq:V-normalization}
        V_k = \dfrac{q^{\binom{k}{2}}(1-q)^k}{\prod_{i=1}^k\left[\left(1-q^i\right)\left(1+q^{i-1}\right)\right]}.
    \end{equation}
\end{thm}
The proof of this result is based on studying the following function, which encapsulates many of the essential algebraic properties of the half-space ASEP:
\begin{equation}
	\label{eq:F initial coord function}
	\mathcal{F}_y(w_1,\dots,w_N) := V_{N-M} \alpha^{N-M}\sum_{\sigma\in\mathcal{B}_N} \sigma\left(\prod_{1\leq i<j\leq N}
	\left[\frac{w_i-qw_j}{w_i-w_j}\frac{1-w_iw_j}{1- qw_iw_j}\right] 	\prod_{i=1}^M \varphi_{y_i}(w_i)\right),
\end{equation}
where
\[\varphi_y(w) := \frac{1-q-\alpha+\alpha w}{1-qw^2}\frac{(1-q)w}{1-w}\left(\frac{1-qw}{1-w}\right)^{y-1}.\]
For consistency with what follows, we define $\mathcal{F}_y(w)=0$ whenever $M>N$. 
Our formula \eqref{eq:ASEP transition prob} for the transition probability may then be expressed in terms of $\mathcal{F}_y$:
\begin{multline}
	\label{eq:ASEP transition prob with F}
	\mathbb{P}_t(y\to x) = \mathrm{e}^{-\alpha t} \oint_{\mathcal{C}_1} \frac{\dd w_1}{2\pi\ii} \cdots \oint_{\mathcal{C}_N} \frac{\dd w_N}{2\pi\ii} \prod_{1\leq i<j\leq N}\left[\frac{w_j-w_i}{qw_j-w_i}\frac{1-qw_iw_j}{1- w_iw_j}\right] \\
	\times\prod_{i=1}^N \left[\frac{1-q w_i^2}{w_i(q+\alpha-1-\alpha w_i)(1-qw_i)} \left(\frac{1-w_i}{1-qw_i}\right)^{x_i-1} \exp(\frac{(1-q)^2 w_i t}{(1-w_i)(1-qw_i)})\right] \mathcal{F}_y(w_1,\dots,w_N).
\end{multline}
The case of empty initial conditions, presented in \cite{garbali_symmetric_2025}, is recovered from \eqref{eq:ASEP transition prob with F} as the special case $M=0$ or $y=\emptyset$, following from the fact that 
\[\mathcal{F}_{\emptyset}(w_1,\dots,w_N)=\alpha^N,\]
which is proven in Section \ref{sec:F-preliminary results}.
Note, in particular, that this special case has a completely factorized integrand.

The next two results encapsulate the main identities needed in the proof of Theorem  \ref{thm:ASEP transition prob}, 
which follow from properties of the function \eqref{eq:F initial coord function}.
Their proofs are very technical, and we defer them to Section \ref{se:ASEPproof}.

The first of the two results shows that the function $\mathcal{F}_y(w_1,\dotsc,w_N)$ from \eqref{eq:F initial coord function} is an eigenvector of the ASEP Markov generator, and is equivalent to diagonalizing the generator using coordinate Bethe ansatz on the Kolmogorov backward equation \eqref{eq:ASEP evolution equation matrix backward}. 
The proof of this result appears in Section \ref{sec:F-eigenvector proof}.

\begin{lm}
	\label{lm: F-eigenvector property}
	For fixed integers $0\leq M\leq N$, let $y=(y_1,\dots,y_M)$ be a half-space particle configuration and let $(w_1,\dots,w_N)\in\mathbb{C}^N$ be a fixed alphabet. Then the function \eqref{eq:F initial coord function} has the following eigenvector relation with the half-space ASEP Markov generator:
	\begin{equation}
		\label{eq: F-eigenvector property}
		\sum_{y'} \bra{y}\mathscr{L}\ket{y'} \mathcal{F}_{y'}(w_1,\dots,w_N) = \left(-\alpha + \sum_{i=1}^N\frac{(1-q)^2 w_i}{(1-w_i)(1-qw_i)}\right)\mathcal{F}_y(w_1,\dots,w_N).
	\end{equation}
\end{lm}
The second intermediate result needed in the proof of Theorem \ref{thm:ASEP transition prob} corresponds the initial condition of \eqref{eq:ASEP transition prob}. This initial condition is equivalent to a homogeneous version of the spatial orthogonality property proposed as Conjecture 5.5 in \cite{garbali_symmetric_2025} at the level of the six-vertex model. We now present it for the half-space ASEP. The proof of this result appears in Section \ref{sec:initial cond proof}.

\begin{lm}
	For fixed integers $N\geq M\geq 0$ let $x=(x_1,\dots,x_N)$ and $y=(y_1,\dots,y_M)$ be half-space configurations. The following orthogonality holds:
	\label{lm: orthogonality}
	\begin{multline}
		\label{eq:orthog t=0}
		\oint_{\mathcal{C}_1} \frac{\dd w_1}{2\pi\ii} \cdots \oint_{\mathcal{C}_N} \frac{\dd w_N}{2\pi\ii} \prod_{1\leq i<j\leq N}\left[\frac{w_j-w_i}{qw_j-w_i}\frac{1-qw_iw_j}{1- w_iw_j}\right] \\
		\times \prod_{i=1}^N\left[\frac{1-q w_i^2}{w_i(q+\alpha-1-\alpha w_i)(1-qw_i)} \left(\frac{1-w_i}{1-qw_i}\right)^{x_i-1}\right] \mathcal{F}_y(w_1,\dots,w_N) = \delta_{x,y}.
	\end{multline}
\end{lm}
We can now use these two results to prove Theorem \ref{thm:ASEP transition prob}.
\begin{proof}[Proof of Theorem \ref{thm:ASEP transition prob}]
    Consider the multiple contour integral expression on the right hand side of \eqref{eq:ASEP transition prob with F}, which we want to show is equal to the half-space ASEP transition probability. This integral expression has the following factor in the integrand:
    \begin{equation}
        \label{eq:ASEP transition proof 1}
        g_t(y)=\mathrm{e}^{-\alpha t}\prod_{i=1}^N\exp(\frac{(1-q)^2w_it}{(1-w_i)(1-qw_i)}) \mathcal{F}_y(w_1,\dots,w_N).
    \end{equation}
    Expanding the exponential and using Lemma \ref{lm: F-eigenvector property} shows that
    \begin{equation}
        \label{eq:ASEP transition proof 2}
        g_t(y)=\sum_{m=0}^\infty \frac{t^m}{m!}\sum_{y^{(1)},\dots,y^{(m)}}\bra{y}\mathscr{L}\ket{y^{(1)}}\bra{y^{(1)}}\mathscr{L}\ket{y^{(2)}}\cdots\bra{{y^{(m-1)}}}\mathscr{L}\ket{y^{(m)}}\mathcal{F}_{y^{(m)}}(w_1,\dots,w_N)
    \end{equation}
    whence it follows that $g_t$ satisfies
    \[\frac{\dd}{\dd t}g_t(y)=\sum_{y'}\bra{y}\mathscr{L}\ket{y'}g_t(y').\]
    Letting $p_t(y,x)$ denote the right hand side of \eqref{eq:ASEP transition prob with F}, this shows that
    \[\frac{\dd}{\dd t}p_t(y,x)=\sum_{y'}\bra{y}\mathscr{L}\ket{y'}p_t(y',x).\]
    Hence $p_t(y,x)$ solves the Kolmogorov backward equation \eqref{eq:ASEP evolution equation matrix backward}.
    Lemma \ref{lm: orthogonality} yields the initial condition $p_0(y,x)=\delta_{x,y}$, finishing the proof.
\end{proof}

\begin{remark}
    The proof of Theorem \ref{thm:ASEP transition prob} is similar in character to the  proofs of integral formulas of symmetric functions indexed by skew partitions defined via vertex models. These integral formulas are typically derived using a Cauchy identity and orthogonality of appropriate families of symmetric functions.  For higher-spin vertex models in full-space, see for example \cite{borodin_higher_2018,borodin_coloured_2022,aggarwal_colored_2023}. In fact, the integral formula for the transition probability \eqref{eq:ASEP transition prob with F} may be derived by the taking the ASEP limit of the partition function $G_{\nu/\mu}$ from \cite{garbali_symmetric_2025}, where it was shown to reduce to a solution of Kolmogorov forward equation of the half-space open ASEP \eqref{eq:ASEP evolution equation matrix forward}. The proof of Theorem \ref{thm:ASEP transition prob} is equivalent to the aforementioned vertex model approach adapted to the half-space function $G_{\nu/\mu}$, whereby Lemma \ref{lm: F-eigenvector property} may be derived from its Cauchy identity and Lemma \ref{lm: orthogonality} is equivalent to the orthogonality of the dual family $F_\nu$.
\end{remark}

\section{Half-space TASEP Pfaffian formulas}

\subsection{Sch\"{u}tz-type Pfaffian transition probability}
\label{subsec:schutz pf}

This section will be devoted to explicit Pfaffian expressions for the transition probabilities of half-space TASEP, that is, the totally asymmetric version of half-space ASEP where particles only jump to the right.
The TASEP transition probability is obtained by taking the limit $q\to0$ within the integral formula of Theorem \ref{thm:ASEP transition prob}. We will present the general expression for this result under both empty and general deterministic initial conditions.
\begin{defn}
    \label{defn:TASEP countours}
    Let us denote by $\beta$ a positively oriented closed complex integration contour which is sufficiently large so that it surrounds the points $0,1,\alpha$ and $1-\alpha$.
\end{defn}
\subsubsection{Empty initial conditions}
\begin{thm}
	\label{thm: transition prob pfaffian kernel}
	Let $N$ be a fixed integer and let $x=(x_1,\dots,x_N)$ be the ordered coordinates of half-space TASEP particles. Then the half-line open TASEP has transition probability from an empty state to $x$ given by the following: when $N$ is even
	\begin{equation}
		\label{eq:half-line TASEP transition prob}
		\mathbb{P}_t(\emptyset\to x) = (-1)^{\binom{N}{2}} \mathrm{e}^{-\alpha t} \Pf \left[\mathsf{Q}\right],
	\end{equation}
	and when $N$ is odd the underlying matrix is augmented with an addition row and column
	\begin{equation}
		\mathbb{P}_t(\emptyset\to x) = (-1)^{\binom{N}{2}} \mathrm{e}^{-\alpha t} \pf\left[\begin{array}{cc}
			\mathsf{Q} & \mathsf{p} \\[5pt]
			-\mathsf{p}^T & 0
		\end{array}\right].
	\end{equation}
    Here, $\mathsf{Q}$ is a skew-symmetric $N\times N$-dimensional matrix and $\mathsf{p}$ is a column vector of length $N$ whose entries are given by
    \begin{align*}
        [\mathsf{Q}]_{i,j} := & Q_{i,j}(x_{N-i+1},x_{N-j+1}), \\ 
        [\mathsf{p}]_i := & p_{i}(x_{N-i+1}),
    \end{align*}
    for all $1\leq i,j\leq N$. We define the integral kernels
	\begin{equation}
		\label{eq:Q-kernel}
		Q_{i,j}(x,y) = \alpha^2 \oint_\contour \frac{\dd u}{2\pi\ii} \oint_\contour \frac{\dd w}{2\pi\ii} \frac{u-w}{1-u-w} \frac{w^{i-x}\mathrm{e}^{t(w-1)}}{(w-\alpha)(w-1)^i} \frac{u^{j-y}\mathrm{e}^{t(u-1)}}{(u-\alpha)(u-1)^j},
	\end{equation}
	and 
	\begin{equation}
		\label{eq:p-kernel}
		p_i(x) = \oint_\contour \frac{\dd w}{2\pi\ii} \frac{w^{i-x}\mathrm{e}^{t(w-1)}}{(w-\alpha)(w-1)^i},
	\end{equation}
	whose contours surround poles at $w,v=1,0,\alpha,1-\alpha$ whilst omitting all other singularities of the integrand.
\end{thm}

Theorem \ref{thm: transition prob pfaffian kernel} is a special case of Theorem \ref{thm: transition prob pfaffian kernel initial cond} which appears later in the text. As such, we only present the proof of the later result. The skew-symmetric nature of the Pfaffian kernel is assured by the following property of \eqref{eq:Q-kernel}
\begin{equation}
	\label{eq:Q-kernel interchange symmetry}
	Q_{k,\ell}(x,y) = - Q_{\ell,k} (y,x).
\end{equation}
In particular, this implies that the function $Q_{k,k}$ is anti-symmetric.
\begin{remark}
	We note here the resemblance of the transition probability \eqref{eq:half-line TASEP transition prob} with its Pfaffian kernel to Sch\"{u}tz's solution of the full-line TASEP \cite{schutz_exact_1997} presented with a determinantal kernel.
\end{remark}
\begin{remark}
	The $w$ and $u$ contours of \eqref{eq:Q-kernel} can be deformed to coincide with each other despite the apparent singularity due to the factor of $1-u-w$ appearing in the denominator of the integrand. Without loss of generality, assume we perform the $w$ integration first. Then the pole at $w=1-u$ can be understood to be a removable singularity since the remaining $u $-integrand has vanishing residue at infinity.
\end{remark}
\begin{remark}
	Our overall probability measure for the transition probability of Theorem \ref{thm: transition prob pfaffian kernel} is over all numbers of particles in the system. That is, for all $t\geq 0$ we have the normalization
	\begin{equation}
		\sum_{N=0}^\infty \sum_{1\leq x_N<\cdots < x_1} \mathbb{P}_t(\emptyset\to (x_1,\dots,x_N)) = 1.
	\end{equation}
\end{remark}

\subsubsection{General initial conditions}

\begin{thm}
	\label{thm: transition prob pfaffian kernel initial cond}
	For fixed integers $N\geq M\geq 0$, let $x=(x_1,\dots,x_N)$ be the ordered coordinates of TASEP particles at time $t$ and let $y=(y_1,\dots,y_M)$ denote the ordered initial coordinates of the particles. In the case $y_M > N-M+1$, the half-line open TASEP has transition probability from an empty state to $x$ given by the following: when $N+M$ is even
	\begin{equation}
		\label{eq:half-line TASEP transition prob initial cond even}
		\mathbb{P}_t(y\to x) = (-1)^{\binom{N}{2}}\mathrm{e}^{-\alpha t} \pf\left[\begin{array}{cc}
			\mathsf{Q} & \mathsf{U} \\[5pt]
			-\mathsf{U}^T & 0
		\end{array}\right],
	\end{equation}
	when $N+M$ is odd
	\begin{equation}
		\label{eq:half-line TASEP transition prob initial cond odd}
		\mathbb{P}_t(y\to x) = (-1)^{\binom{N}{2}}\mathrm{e}^{-\alpha t} 
		\pf\left[\begin{array}{ccc}
			\mathsf{Q} & \mathsf{p} & \mathsf{U}\\[5pt]
			-\mathsf{p}^T & 0 & 0 \\[5pt]
		 	-\mathsf{U}^T & 0 & 0
		\end{array}\right].
	\end{equation}
    Here, $\mathsf{Q}$ is a skew-symmetric $N\times N$-dimensional matrix and $\mathsf{p}$ is a column vector of length $N$ which are the same as those in Theorem \ref{thm: transition prob pfaffian kernel}, while $\mathsf{U}$ is a $N\times M$-dimensional matrix whose entries are given by
    \begin{align*}
        [\mathsf{Q}]_{i,j} := & Q_{i,j}(x_{N-i+1},x_{N-j+1}), \\ 
        [\mathsf{p}]_i := & p_{i}(x_{N-i+1}), \\
        [\mathsf{U}]_{i,k} := & U_{i-k}(x_{N-i+1}-y_{M-k+1}),
    \end{align*}
    for all $1\leq i,j\leq N$ and $1\leq k\leq M$. The functions $Q_{i,j}$ and $p_i$ are defined by \eqref{eq:Q-kernel} and \eqref{eq:p-kernel} respectively while we define the integral kernel
	\begin{align}
		\label{eq:R-kernel}
		U_{k}(z) := \oint_\gamma \frac{\dd w}{2\pi\ii} \frac{(w-1)^{N-M-k}\mathrm{e}^{t(w-1)}}{w^{z-k+N-M+1}},
	\end{align}
    where the contour $\gamma$ encloses the (possible) pole at the origin.
\end{thm}

\begin{remark}
    \label{rmk:initial condition restriction}
    The restriction $y_M>N-M+1$ on the initial conditions is technical in nature, and arises in the proof of the limit stated in Lemma \ref{lm:F inital coord q to 0} below. While this limit could, in principle, be computed without the restriction, the relatively simple Pfaffian structure obtained in the lemma appears not to persist beyond this case, so we do not attempt a generalization here.
    On the other hand, and despite arising in a purely technical manner, the restriction to $y_M>N-M+1$ admits a clear physical interpretation: it is precisely the assumption ensuring that  
    \begin{equation}
        \label{eq:initial condition restriction property}
        \mathbb{P}\left(\abs{X_t} = N\given X_0=y\right) = \mathbb{P}\left(\abs{X_t} = N-M\given X_0=\emptyset\right).
    \end{equation}
    In other words, given $M$ particles initially in the system at positions $y$, the probability that exactly $N-M$ additional particles enter the system is independent of the specific configuration $y$ provided that it satisfies the constraint $y_M>N-M+1$. 
\end{remark}

Before proving Theorem \ref{thm: transition prob pfaffian kernel initial cond}, let us show how it recovers Sch\"{u}tz's classical determinantal formula for the full-space TASEP.

\begin{cor}[Sch\"{u}tz's formula]
	\label{cor:schutz TASEP det}
	For fixed integer $N$, let $x=(x_1,\dots,x_N)$ and $y=(y_1,\dots,y_N)$ be fixed configurations of the full-space TASEP. The transition probability for the full-space TASEP is given by the following determinantal expression
	\begin{equation}
		\label{eq:schutz TASEP det}
		\lim_{\alpha\to0} \mathbb{P}_t(y\to x) = \det\left[U_{i-j}(x_{N-i+1}-y_{N-j+1})\right]_{1\leq i,j\leq N}.
	\end{equation}
	The expression \eqref{eq:schutz TASEP det} exactly coincides with Sch\"{u}tz's determinantal expression for the full-space TASEP transition probability from \cite{schutz_exact_1997}.
\end{cor}
\begin{proof}
    Since $N=M$, we consider the even version of the transition probability \eqref{eq:half-line TASEP transition prob initial cond even} where $x,y$ are half-space coordinates satisfying the conditions of Theorem \ref{thm: transition prob pfaffian kernel initial cond}. Using Proposition \ref{prop:block pf eval} observe that this Pfaffian reduces to a determinant, so that the transition probability evaluates as
    \[\mathbb{P}_t(y\to x) =  \mathrm{e}^{-\alpha t} \det[U_{i-j}(x_{N-i+1}-y_{N-j+1})]_{1\leq i,j\leq N}.\]
    In particular, we note this expression is completely independent of the kernel $Q_{k,\ell}$ so that the determinant is independent of $\alpha$. The result follows by setting $\alpha=0$ and extending the coordinates $x,y$ to become coordinates of the full-space TASEP by simultaneously shifting $x\mapsto x-(c,\dots,c)$ and $y\mapsto y-(c,\dots,c)$ for any chosen constant integer $c\geq0$.
\end{proof}

The key to the proof of Theorem \ref{thm: transition prob pfaffian kernel initial cond} is the following result:

\begin{lm}
    \label{lm:F inital coord q to 0}
	Let $(w_1,\dots,w_N)$ be an alphabet of length $N$ and let $y=(y_1,\dots,y_M)$ be a half-space ASEP particle configuration of length $M\leq N$ satisfying $y_M>N-M+1$. Then the function \eqref{eq:F initial coord function} has the following Pfaffian limit: when $N+M$ is even
	\begin{equation}
        \label{eq:F inital coord q to 0 even}
		\lim_{q\to0} \mathcal{F}_{y}(w_1,\dots,w_N) = (-1)^{\binom{M}{2}} \alpha^{N-M}\prod_{1\leq i<j\leq N} \frac{1-w_i w_j}{w_i-w_j} \pf \left[\begin{array}{cc}
		S(w_i,w_j)_{1\leq i,j\leq N} & g_{k,y_k}(w_i)_{1 \leq i \leq N\atop 1\leq k \leq M} \\[8pt]
		-g_{k,y_k}(w_j)_{1\leq k \leq M \atop 1 \leq j \leq N} & 0
		\end{array}\right],
	\end{equation}
	and when $N+M$ is odd
	\begin{multline}
        \label{eq:F inital coord q to 0 odd}
		\lim_{q\to0} \mathcal{F}_{y}(w_1,\dots,w_N) = (-1)^{\binom{M}{2}} \alpha^{N-M}\\
		\times \prod_{1\leq i<j\leq N}\frac{1-w_i w_j}{w_i-w_j} \pf \left[\begin{array}{ccc}
			S(w_i,w_j)_{1\leq i,j\leq N} &  \vec{\bm1}_N & g_{k,y_k}(w_i)_{1 \leq i \leq N\atop 1\leq k \leq M} \\[8pt]
			-\vec{\bm1}_N^T &  0 &  0 \\[8pt]
			-g_{k,y_k}(w_j)_{1\leq k \leq M \atop 1 \leq j \leq N} & 0 & 0
		\end{array}\right].
	\end{multline}
	Here the skew-symmetric function $S$ is the kernel of the Stembridge Pfaffian (see Lemma \ref{lm: stembridge pfaffian})
	\begin{equation}
		S(v,w) = \frac{v-w}{1-vw}
	\end{equation}
	while the function $g_{k,y}$ for $y>N-i+1$ is defined by
	\begin{equation}
		g_{k,y}(w) =  \frac{(1-\alpha+\alpha w)w^{N-k+1}}{(1-w)^y}.
	\end{equation}
\end{lm}
\begin{proof}
	Recall the symmetrization formula \eqref{eq:F initial coord function}. Let us define the function $\widetilde{\mathcal{F}}_y$ as the symmetrization formula \eqref{eq:F initial coord function} where the overall constant $V_{N-M}$ is replaced by 
    \[\begin{cases}
        \frac12 q^{\binom{N-M}{2}} & \text{ if }N>M \\
        1 & \text{ otherwise}
    \end{cases},\]
    so that $\lim_{q\to0}\widetilde{\mathcal{F}}_y(w)=\lim_{q\to0}\mathcal{F}_y(w)$ for all appropriate alphabets $w=(w_1,\dots,w_N)$. We will now consider the $q\to0$ limit on the function $\widetilde{\mathcal{F}}_y$ in the case $N>M$ only. The $N=M$ case follows by a very similar argument. For $N>M$, the symmetrization formula becomes:
    \begin{equation}
        \label{eq:Fqto0 symm}
        \widetilde{\mathcal{F}}_y(w_1,\dots,w_N) := \frac12 q^{\binom{N-M}{2}}\alpha^{N-M}\sum_{\sigma\in\mathcal{B}_N} \sigma\left(\prod_{1\leq i<j\leq N}
	\left[\frac{w_i-qw_j}{w_i-w_j}\frac{1-w_iw_j}{1- qw_iw_j}\right] 	\prod_{i=1}^M \varphi_{y_i}(w_i)\right).
    \end{equation}
    We first separate this symmetrization over signed permutations $\mathcal{B}_N$ into a sum over the symmetric group $S_N$ and a sum over signs. Observe that the factor
    \[
    \frac{w_i-qw_j}{w_i-w_j}\frac{1-w_iw_j}{1- qw_iw_j}
    \]
    in \eqref{eq:Fqto0 symm} is invariant under $w_j\to 1/qw_j$.
    By separating the product over $i$ into $1\le i \le M$ and  $M+1\le i \le N$ we may re-write \eqref{eq:Fqto0 symm} as
    \begin{multline}
    \widetilde{\mathcal{F}}_{y}(w_1,\dots,w_N) = \frac12 q^{\binom{N-M}{2}}\alpha^{N-M}\! \sum_{\{\epsilon_i=\pm 1\}}\sum_{\sigma\in S_N} \left(\prod_{i=M+1}^{N-1} \prod_{j=i+1}^N
    \left[ \frac{w_{\epsilon_i \sigma_i}-qw_{\sigma_j}}{w_{\epsilon_i \sigma_i}-w_{\sigma_j}}\frac{1-w_{\epsilon_i \sigma_i}w_{\sigma_j}}{1- qw_{\epsilon_i \sigma_i}w_{\sigma_j}} \right] \right. \times \\
    \left. \times \prod_{i=1}^M \left[ \prod_{j=i+1}^N
    \left[\frac{w_{\epsilon_i \sigma_i}-qw_{\sigma_j}}{w_{\epsilon_i \sigma_i}-w_{\sigma_j}}\frac{1-w_{\epsilon_i \sigma_i}w_{\sigma_j}}{1- qw_{\epsilon_i \sigma_i}w_{\sigma_j}}  \right] \varphi_{y_i}(w_{\epsilon_i \sigma_i})\right]\right).
    \label{eq:Fsymm_b}
    \end{multline}
    The sum over the signs $\{\epsilon_i=\pm 1\}$ in \eqref{eq:Fsymm_b} can now be performed independently for each $i$. Using $w_{-i} = 1/qw_i$ for $i=1,\ldots,N$, and absorbing the pre-factor $q^{\binom{N-M}{2}}$ we thus get
    \begin{multline}
    \widetilde{\mathcal{F}}_{y}(w_1,\dots,w_N) =  \alpha^{N-M}\sum_{\sigma\in S_N} \sigma \left\{ \prod_{i=M+1}^{N-1} \left( q^{N-i} \prod_{j=i+1}^N \left[\frac{w_i-q w_j}{w_i-w_j}\frac{1-w_iw_j}{1-q w_i w_j}\right] + \prod_{j=i+1}^N \left[\frac{qw_i-w_j}{w_i-w_j}\frac{1-q^2w_iw_j}{1- qw_iw_j}\right] \right) \times \right. \\
    \times 
    \left. \prod_{i=1}^M \left( \prod_{j=i+1}^N \left[\frac{w_i-q w_j}{w_i-w_j}\frac{1-w_iw_j}{1-q w_iw_j}\right] \varphi_{y_i}(w_{\sigma_i}) + q^{y_i-N+i-1}  \prod_{j=i+1}^N \left[\frac{qw_i-w_j}{w_i-w_j}\frac{1-q^2w_iw_j}{1-q w_iw_j} \right] \widetilde{\varphi}_{y_i}(w_{\sigma_i})
    \right)\right\}.
    \label{eq:Fsymm_c}
    \end{multline}
    where
    \begin{equation}
    \widetilde{\varphi}_y(w) := q^{-y+1} \varphi_y\left(\frac{1}{q w}\right) = \frac{qw-a}{1-qw^2} \frac{(1-q)w}{1-qw} \left(\frac{1-w}{1-qw}\right)^{y-1}.
    \end{equation}
    The assumption $y_M>N-M+1$ implies more generally that $y_i > N-i+1$ for each $i=1,\dotsc,N$ (because $y_1>y_2>\dotsm>y_M$), and under this condition the limit $q\to 0$ can be taken term-wise in \eqref{eq:Fsymm_c}\footnote{The $q\to0$ limit still exists when this restriction is not satisfied, however, the evaluation is more complicated.}.    
    By using the symmetrization to simplify, we arrive at the following expression:    
    \begin{equation}
        \label{eq:Fqto0 F-symmetric group}
        \lim_{q\to0}\widetilde{\mathcal{F}}_y(w_1,\dots,w_N) = \alpha^{N-M}\sum_{\sigma \in S_N}\sigma\left(\prod_{1\leq i<j\leq N}\frac{1}{w_i-w_j}\prod_{i=1}^N w_i^{N-i} \prod_{i=1}^M\left[\frac{(1-\alpha+\alpha w_i)w_i}{(1-w_i)^{y_i}}\prod_{j=i+1}^N(1-w_iw_j)\right]\right).
    \end{equation}
    Overall factors can be removed from the symmetrization over permutations \eqref{eq:Fqto0 F-symmetric group}, which yields:
    \begin{multline}
        \label{eq:Fqto0 F-symmetric group 2}
        \lim_{q\to0}\widetilde{\mathcal{F}}_y(w_1,\dots,w_N) = \alpha^{N-M}\prod_{1\leq i<j\leq N}\frac{1-w_iw_j}{w_i-w_j}\\
        \times\sum_{\sigma \in S_N}(-1)^{\abs{\sigma}}\sigma\left(\prod_{M+1\leq i<j\leq N}\frac{1}{1-w_iw_j}\prod_{i=M+1}^N w_i^{N-i} \prod_{i=1}^M\frac{(1-\alpha+\alpha w_i)w_i^{N-i+1}}{(1-w_i)^{y_i}}\right).
    \end{multline}
    From here, we may decompose the sum over the symmetric group using the following identity:
    \begin{multline}
        \label{eq:Fqto0 sym-id}
        \sum_{\sigma\in S_N}(-1)^{\abs{\sigma}} f_1(w_{\sigma_1},\dots,w_{\sigma_M})f_2(w_{\sigma_{M+1}},\dots,w_{\sigma_N}) \\ 
        = (-1)^{\binom{N+1}{2}+\binom{M+1}{2}}\sum_{T\subseteq \{1,\dots,N\}\atop \abs{T}=N-M} (-1)^{\sum_iT_i}\sum_{\rho\in S_{M}}(-1)^{\abs{\rho}}f_1\left(w_{T^\mathsf{C}_{\rho_1}},\dots,w_{T^\mathsf{C}_{\rho_M}}\right)\sum_{\sigma\in S_{N-M}}(-1)^{\abs{\sigma}}f_2\left(w_{T_{\sigma_1}},\dots,w_{T_{\sigma_{N-M}}}\right),
    \end{multline}
    for arbitrary appropriate functions $f_1,f_2$, where $T^\mathsf{C}$ denotes the complement of $T\subseteq \{1,\dots,N\}$. And so, the symmetrization \eqref{eq:Fqto0 F-symmetric group 2} can be simplified using identity \eqref{eq:Fqto0 sym-id} to yield:
    \begin{multline}
        \label{eq:Fqto0 F-symmetric group 3}
        \lim_{q\to0}\widetilde{\mathcal{F}}_y(w_1,\dots,w_N) = (-1)^{\binom{N+1}{2}+\binom{M+1}{2}}\alpha^{N-M}\prod_{1\leq i<j\leq N}\frac{1-w_iw_j}{w_i-w_j}\\
        \times \sum_{T\subseteq \{1,\dots,N\}\atop \abs{T}=N-M} (-1)^{\sum_iT_i} \left[\left(\sum_{\sigma\in S_{N-M}}(-1)^{\abs{\sigma}} \prod_{1\leq i<j\leq N-M}\frac{1}{1-w_{T_{\sigma_i}}w_{T_{\sigma_j}}}\prod_{i=1}^{N-M} w_{T_{\sigma_i}}^{N-M-i}\right) \right. \\
        \left.\times \left(\sum_{\rho\in S_M}(-1)^{\abs{\rho}} \prod_{i=1}^M\frac{\left(1-\alpha+\alpha w_{T^\mathsf{C}_{\rho_i}}\right)w_{T^\mathsf{C}_{\rho_i}}^{N-i+1}}{\left(1-w_{T^\mathsf{C}_{\rho_i}}\right)^{y_i}}\right)
        \right].
    \end{multline}
    Now, for each fixed $T\subseteq\{1,\dots,N\}$, the sum over $\sigma$ in \eqref{eq:Fqto0 F-symmetric group 3} can be identified as
    \begin{multline*}
        \sum_{\sigma\in S_{N-M}}(-1)^{\abs{\sigma}} \prod_{1\leq i<j\leq N-M}\frac{1}{1-w_{T_{\sigma_i}}w_{T_{\sigma_j}}}\prod_{i=1}^{N-M} w_{T_{\sigma_i}}^{N-M-i} \\ = \prod_{1\leq i<j\leq N-M}\frac{w_{T_i}-w_{T_j}}{1-w_{T_{i}}w_{T_{j}}} 
         = \begin{cases}
             \pf\left[S\left(w_{T_i},w_{T_j}\right)\right]_{1\leq i,j\leq N-M} & \text{if $N-M$ is even}, \\ \quad \\
             \pf\left[\begin{array}{cc}
                S\left(w_{T_i},w_{T_j}\right)_{1\leq i,j\leq N-M} & \vec{\bm1}_{N-M} \\[5pt]
                 \vec{\bm1}_{N-M}^T & 0
             \end{array}\right] & \text{if $N-M$ is odd},
         \end{cases}
    \end{multline*}
    where we have used the evaluation of the Vandermonde determinant and the Stembridge Pfaffian of Lemma \ref{lm: stembridge pfaffian}. 
    The sum over $\rho$ in \eqref{eq:Fqto0 F-symmetric group 3} is then also identified with the following determinant
    \[\sum_{\rho\in S_M}(-1)^{\abs{\rho}} \prod_{i=1}^M\frac{\left(1-\alpha+\alpha w_{T^\mathsf{C}_{\rho_i}}\right)w_{T^\mathsf{C}_{\rho_i}}^{N-i+1}}{\left(1-w_{T^\mathsf{C}_{\rho_i}}\right)^{y_i}} = \det\left[g_{j,y_j}\left(w_{T^\mathsf{C}_i}\right)\right]_{1\leq i,j\leq M}.\]
    Using both of these identifications, both the odd and even cases of the right hand side of \eqref{eq:Fqto0 F-symmetric group 3} may be simplified using Proposition \ref{prop:block pf eval}. The result for both even \eqref{eq:F inital coord q to 0 even} and odd \eqref{eq:F inital coord q to 0 odd} $N+M$ follow once the following identity is applied
    \[(-1)^{\binom{N+1}{2}+\binom{M+1}{2}}(-1)^{\binom{N}{2}} = (-1)^{\binom{M}{2} - (N-M)}.\qedhere\]
\end{proof}
    
\begin{proof}[Proof of Theorem \ref{thm: transition prob pfaffian kernel initial cond}]
	We proceed by considering the $q=\gamma=0$ case of the half-space ASEP transition probability which is obtained as the $q\to0$ limit of the expression \eqref{eq:ASEP transition prob with F}.  
    The evaluation of the $q\to0$ limit on $\mathcal{F}_y(w)$ follows from Lemma \ref{lm:F inital coord q to 0}. For simplicity, we will only present the case where $N+M$ is even, where the odd case follows similarly. In the case where $N+M$ is even, the this limit is given by 
    \begin{multline}
		\label{eq:Schutz Pfaffian IC proof 0}
		\lim_{q\to0} \mathbb{P}_t(y\to  x) = \alpha^N\mathrm{e}^{-\alpha t}\oint_{\mathcal{C}_1}\frac{\dd w_1}{2\pi\ii} \cdots \oint_{\mathcal{C}_N}\frac{\dd w_N}{2\pi\ii} 
		\prod_{i=1}^N \left[\frac{(1-w_i)^{x_i-1}}{w_i^{N-i+1}}\frac{\exp(\frac{w_i t}{1-w_i})}{\alpha-1-\alpha w_i} \right]\\ 
        \times \pf \left[\begin{array}{cc}
		S(w_i,w_j)_{1\leq i,j\leq N} & g_{k,y_k}(w_i)_{1 \leq i \leq N\atop 1\leq k \leq M} \\[8pt]
		-g_{k,y_k}(w_j)_{1\leq k \leq M \atop 1 \leq j \leq N} & 0
		\end{array}\right],
	\end{multline}
    where all contours $\mathcal{C}_i$ satisfy the constraints of Definition \ref{defn:ASEP nested contours}. In particular, they all surround the point $w_i=1$. It is important to note that, precisely when $q=0$, we are able to deform all contours to be the same without crossing any singularities of the integrand. That is, for all $i$ we deform $\mathcal{C}_i\mapsto \mathcal{D}$, for some appropriate contour $\mathcal{D}$ which surrounds the point 1. Since all integrations are over the same contour, we may symmetrize the integrand of \eqref{eq:Schutz Pfaffian IC proof 0} to obtain
    \begin{multline}
		\label{eq:Schutz Pfaffian IC proof 1}
		\lim_{q\to0} \mathbb{P}_t(y\to x) = \frac{\alpha^{N-M}\mathrm{e}^{-\alpha t}}{N!}(-1)^{\binom{M}{2}}\oint_{\mathcal{D}}\frac{\dd w_1}{2\pi\ii} \cdots \oint_{\mathcal{D}}\frac{\dd w_N}{2\pi\ii} \prod_{i=1}^N \frac{\exp(\frac{w_i t}{1-w_i})}{\alpha-1-\alpha w_i} \\
		\times\det\left[\frac{(1-w_i)^{x_j-1}}{w_i^{N-j+1}}\right]_{1\leq i,j \leq N} \pf \left[\begin{array}{cc}
		S(w_i,w_j)_{1\leq i,j\leq N} & g_{k,y_k}(w_i)_{1 \leq i \leq N\atop 1\leq k \leq M} \\[8pt]
		-g_{k,y_k}(w_j)_{1\leq k \leq M \atop 1 \leq j \leq N} & 0
		\end{array}\right],
	\end{multline}
    while the odd case follows similarly with an enlarged Pfaffian as in Lemma \ref{lm:F inital coord q to 0}. We then make the change of integration variable $w_i\mapsto 1-w_i^{-1}$ and suitably deform the contours. 
    
    We note here that whenever $y=\emptyset$ (or $M=0$), for both odd and even $N$, the Pfaffian in \eqref{eq:Schutz Pfaffian IC proof 1} is reduced to $\prod_{1\leq i<j\leq N} \frac{w_i-w_j}{1- w_iw_j}$, which is equal to Stembridge's Pfaffian from Lemma \ref{lm: stembridge pfaffian}. Theorem \ref{thm: transition prob pfaffian kernel} follows by appropriate manipulation after applying de Bruijn's integration formula from Lemma \ref{lm:de bruijn formula}.
    
    Whenever $M>0$, regardless of whether $N+M$ is even or odd, we apply generalized version of de Bruijn's integration formula of Lemma \ref{lm:gen de bruijn formula} to \eqref{eq:Schutz Pfaffian IC proof 1}. When $N+M$ is even, this yields the following single Pfaffian expression
    \begin{equation}
        \label{eq:Schutz Pfaffian IC proof 2}
        \lim_{q\to0} \mathbb{P}_t(y\to x) = \mathrm{e}^{-\alpha t} \pf\left[\begin{array}{cc}
			Q_{N-i+1,N-j+1}(x_{i},x_{j})_{1\leq i,j\leq N} & U_{N-i+1-k}(x_{i}-y_{M-k+1})_{1\leq i\leq N\atop 1\leq k\leq M} \\[8pt]
			-U_{N-j+1-k}(x_{j}-y_{M-k+1})_{1\leq k\leq M\atop 1\leq j\leq N} & 0
		\end{array}\right],
    \end{equation}
    where the factor of $\alpha^{N-M}$ has been absorbed into the Pfaffian by simultaneously multiplying into the first $N$ rows and columns and dividing the last $M$ rows and columns by a factor of $\alpha$. The Pfaffian expression \eqref{eq:Schutz Pfaffian IC proof 2} can be brought into the form \eqref{eq:half-line TASEP transition prob initial cond even} by simultaneously reversing the order of the first $N$ rows and columns which generates an overall factor of $(-1)^{\binom{N}{2}}$. The case where $N+M$ is odd case follows analogously to yield \eqref{eq:half-line TASEP transition prob initial cond odd}.
\end{proof}

\subsection{Kernel recurrence relations}
The Pfaffian kernel functions of Theorems \ref{thm: transition prob pfaffian kernel} and \ref{thm: transition prob pfaffian kernel initial cond} satisfy important recurrence relations. Analogous relations have been used to construct a determinantal point process from Sch\"{u}tz's transition probability for the full-space TASEP in \cite{sasamoto_spatial_2005,borodin_fluctuation_2007} and will be used to define a Pfaffian point process for the half-space TASEP in Section \ref{sec:Pfaffian pt proc}.
\begin{lm}
	\label{lm:Q-kernel recursion}
	The functions defined by \eqref{eq:Q-kernel},\eqref{eq:p-kernel} and \eqref{eq:R-kernel} satisfy the following recurrence relations
	\begin{align}
        \begin{split}
    		Q_{i+1,j}(x_{i},x_{j}) & = \sum_{y_i\geq x_i} Q_{i,j}(y_i,x_j), \\
    		Q_{i,j+1}(x_{i},x_{j}) & = \sum_{y_j\geq x_j} Q_{i,j}(x_i,y_j), \\
    		U_{k+1}(z) & = \sum_{y\geq z} U_{k}(y), \\
    		p_{k+1}(z) & = \sum_{y\geq z} p_k(y).
        \end{split}
	\end{align}
\end{lm}
\begin{proof}
	The argument follows a straightforward geometric series where uniform convergence is assured by deforming the contour radii of \eqref{eq:Q-kernel}, \eqref{eq:p-kernel} and \eqref{eq:R-kernel} to be sufficiently large.
\end{proof}
The recurrence relations also have a useful finite-summation version.
\begin{lm}
    \label{lm:Q-kernel recursion finite}
    The functions defined by \eqref{eq:Q-kernel},\eqref{eq:p-kernel} and \eqref{eq:R-kernel} satisfy the following recurrence relations
	\begin{align}
        \begin{split}
    		Q_{i+1,j}(s_{i},x_{j})-Q_{i+1,j}(x_{i},x_{j}) & = \sum_{y_i=s_i}^{x_i-1} Q_{i,j}(y_i,x_j), \\
    		Q_{i,j+1}(x_{i},s_{j})-Q_{i,j+1}(x_{i},x_{j}) & = \sum_{y_j= s_j}^{x_j-1} Q_{i,j}(x_i,y_j), \\
    		U_{k+1}(s)-U_{k+1}(z) & = \sum_{y= s}^{z-1} R_{k}(y), \\
    		p_{k+1}(s)-p_{k+1}(z) & = \sum_{y=s}^{z-1} p_k(y).
        \end{split}
	\end{align}
\end{lm}
\subsection{Joint distribution}

From Theorem \ref{thm: transition prob pfaffian kernel initial cond}, we can derive directly by summation a similar formula for the joint distribution of the particle currents in the half-space TASEP.

\begin{prop}
    \label{prop:joint dist pfaff initial}
	Let $N\geq M\geq $ be fixed integers and consider integers $s_1>s_2>\dotsm>s_N\geq1$ and $y_1>y_2>\dotsm>y_M$. Then the joint distribution of particle positions in the half-space TASEP, given initial condition $y=(y_1,\dots,y_M)$ satisfying $y_M>N-M+1$, is given by the following Pfaffian formulas: when $N+M$ is even
    \begin{equation}
        \label{eq:joint dist pfaff initial even}
		\mathbb{P}\left(\bigcap_{k=1}^N \left\{X_t(k)\geq s_{k}\right\}\text{ and }\abs{X_t}=N\given X_0=y\right)
		= (-1)^{\binom{N}{2}}\mathrm{e}^{-\alpha t}\pf\left[\begin{array}{cc}
			\widetilde{\mathsf{Q}} & \widetilde{\mathsf{U}} \\[5pt]
			-\widetilde{\mathsf{U}}^T & 0
		\end{array}\right],
    \end{equation}
	and when $N+M$ is odd
    \begin{equation}
        \label{eq:joint dist pfaff initial odd}
		\mathbb{P}\left(\bigcap_{k=1}^N \left\{X_t(k)\geq s_{k}\right\}\text{ and }\abs{X_t}=N\given X_0=y\right)=(-1)^{\binom{N}{2}}\mathrm{e}^{-\alpha t} 
		\pf\left[\begin{array}{ccc}
			\widetilde{\mathsf{Q}} & \widetilde{\mathsf{p}} & \widetilde{\mathsf{U}}\\[5pt]
			-\widetilde{\mathsf{p}}^T & 0 & 0 \\[5pt]
		 	-\widetilde{\mathsf{U}}^T & 0 & 0
		\end{array}\right].
    \end{equation}
    Here, $\widetilde{\mathsf{Q}},\widetilde{\mathsf{U}}$ are $N\times N$ and $N\times M$-dimensional matrices respectively and $\widetilde{\mathsf{p}}$ and a vector of length $N$ whose entries are given by
    \begin{align}
        \begin{split}
            \label{eq:joint dist pfaff initial entries}
            \left[\widetilde{\mathsf{Q}}\right]_{i,j} := & Q_{i+1,j+1}(s_{N-i+1},s_{N-j+1}), \\ 
            \left[\widetilde{\mathsf{p}}\right]_i := & p_{i+1}(s_{N-i+1}), \\
            \left[\widetilde{\mathsf{U}}\right]_{i,k} := & U_{i-k+1}(s_{N-i+1}-y_{M-k+1}),
        \end{split}
    \end{align}
    for all $1\leq i,j\leq N$ and $1\leq k \leq M$ where the functions are given by \eqref{eq:Q-kernel}, \eqref{eq:p-kernel} and \eqref{eq:R-kernel} respectively.
\end{prop}

A formula of this type can be derived also for TASEP on the full-line, either by arguing similarly directly on \eqref{eq:schutz TASEP det}, or by specializing the half-space formula to full-space as done in the proof of Corollary \ref{cor:schutz TASEP det}.
The disadvantage of such a formula is that it appears to be ill-suited for asymptotic analysis (although for periodic TASEP the idea has been used fruitfully, see e.g. \cite{baik_fluctuations_2018}).

In our setting, we are interested in the proposition mostly due to the following corollary, which computes the probability of observing exactly $N$ particles in the half-space TASEP at time $t$ (this is equivalent to observing the integrated particle current at the boundary).
This yields the normalization constant which will later appear in the conditional joint distribution in Section \ref{sec:conditional}, and the explicit formula will be useful in the derivation there.

\begin{cor}
    \label{cor:boundary current schutz pf}
    Let $N\geq M\geq 0$ be a fixed integers and let $y=(y_1,\dots,y_M)$ be a fixed half-space configuration satisfying $y_M>N-M+1$. Then the probability of the half-TASEP having exactly $N$ particles, i.e. 
    \[\mathbb{P}\left(\abs{X_t} = N \given X_0=y\right),\]
    is given by the Pfaffian on the right hand sides of \eqref{eq:joint dist pfaff initial even} when $N+M$ is even and \eqref{eq:joint dist pfaff initial odd} when $N+M$ is odd, where the matrix entries are given by \eqref{eq:joint dist pfaff initial entries} with $s_1=\cdots=s_N=1$.
\end{cor}

\begin{proof}[Proof of Theorem \ref{prop:joint dist pfaff initial}]
	The joint distribution may be calculated from the transition probability as
    \begin{equation}
        \label{eq:joint dist pf proof 1}
        \mathbb{P}\left(\bigcap_{k=1}^N \left\{X_t(k)\geq s_{k}\right\}\text{ and }\abs{X_t}=N\given X_0=y\right) =  \sum_{x_1=s_1}^\infty \sum_{x_2=s_2}^{x_1-1}\cdots \sum_{x_N=s_N}^{x_{N-1}-1}\mathbb{P}_t(y\to x),
    \end{equation}
    where the transition probabilities are given by the Pfaffian expressions of Theorem \ref{thm: transition prob pfaffian kernel} when $y=\emptyset$ and Theorem \ref{thm: transition prob pfaffian kernel initial cond} when $y$ is non-empty. 
    Note that all these sums are over non-empty ranges thanks to our assumption on the $s_i$'s. 
    Let us consider the particular case of the Sch\"{u}tz-type Pfaffian transition probability where $N+M$ is odd \eqref{eq:half-line TASEP transition prob initial cond odd} since the even case is simpler. We evaluate the sum over $x_{N}$ in \eqref{eq:joint dist pf proof 1} first, where we note that the $x_N$-dependence only appears in the first row and column of the underlying matrix \eqref{eq:half-line TASEP transition prob initial cond odd}. The first row can be written as
    \begin{equation}
        \label{eq:joint dist pf proof 2}
        \left[\begin{array}{cccccccc}
        0 & Q_{1,2}(x_{N},x_{N-1}) & \cdots & Q_{1,N}(x_{N},x_1) & p_1(x_N) & U_0(x_{N}-y_{M}) & \cdots & U_{1-M}(x_{N}-y_1)  
    \end{array}\right],
    \end{equation}
    where the sum over $x_N$ can be evaluated on the individual rows and columns of the Pfaffian simultaneously using the multi-linearity of the Pfaffian. And so, the sum over $x_N$ is evaluated readily using the recurrence relation of Lemma \ref{lm:Q-kernel recursion finite}. This leads the the first row becoming
    \begin{multline}
        \label{eq:joint dist pf proof 3}
        \left[\begin{array}{cccccccc}
        0 & Q_{2,2}(s_{N},x_{N-1}) & \cdots & Q_{2,N}(s_{N},x_1) & p_2(s_N) & U_1(s_{N}-y_{M}) & \cdots & U_{2-M}(s_{N}-y_1)  
    \end{array}\right] \\ 
    - \left[\begin{array}{cccccccc}
        0 & Q_{2,2}(x_{N-1},x_{N-1}) & \cdots & Q_{2,N}(x_{N-1},x_1) & p_2(x_{N-1}) & U_1(x_{N-1}-y_{M}) & \cdots & U_{2-M}(x_{N-1}-y_1)  
    \end{array}\right],
    \end{multline}
    while the first column is modified simultaneously in the analogous way. Observe that, beyond the first entry, the second term in \eqref{eq:joint dist pf proof 3} is the equal to the second-row of the subsequent matrix after the $x_N$-summation has taken place. So, we apply the simultaneous row and column operations, $r_1\mapsto r_1+r_2$ and $c_1\mapsto c_1+c_2$, which do not affect the value of the Pfaffian. These operations leave the first row of the resulting matrix as the first term in \eqref{eq:joint dist pf proof 3}.

    We will now perform the subsequent summations over $x_{N-1},\dots,x_2$ along the same lines. That is, for each $1<k<N$ we evaluate the sum over $x_{N-k+1}$ in \eqref{eq:joint dist pf proof 1} by applying the recurrences of Lemma \ref{lm:Q-kernel recursion finite} simultaneously to the $k$-th row and column of the matrix. This is then followed by the simultaneous row and column operations $r_{k}\mapsto r_{k}+r_{k+1}$ and $c_{k}\mapsto c_{k}+c_{k+1}$. 
    
    Finally, after performing this procedure for all in order for all $k$, we may perform the summation over $x_1$ in \eqref{eq:joint dist pf proof 1} using Lemma \ref{lm:Q-kernel recursion}. This process results in the joint distribution given as the Pfaffian over the matrices \eqref{eq:joint dist pfaff initial odd} or \eqref{eq:joint dist pfaff initial even} depending on whether $N+M$ is odd or even respectively.
\end{proof}

\begin{proof}[Proof of Corollary \ref{cor:boundary current schutz pf}]
   Setting $s_i=N+1-i$ for $1\leq i \leq N$ in Theorem \ref{prop:joint dist pfaff initial} we get immediately that $\mathbb{P}\left(\abs{X_t} = N \given X_0=y\right)$ is given by the claimed formulas with that choice of $s_i$'s.
   Now from Lemma \ref{lm:Q-kernel recursion finite} we obtain the recursions $Q_{i,j}(i,j)=Q_{i,j}(i-1,j)-Q_{i-1,j}(i-1,j)$, $Q_{i,j}(i,j)=Q_{i,j}(i,j-1)-Q_{i,j-1}(i,j-1)$, $U_{i}(s)=U_{i}(s-1)-U_{i-1}(s-1)$ and $p_{i}(s)=p_{i}(s-1)-p_{i-1}(s-1)$, so applying repeatedly row and column operations to the Pfaffians in the formulas, in a similar way as done in the previous proof, we can successively lower the values of the $s_i$'s to obtain the same formula evaluated at $s_1=\dotsm=s_N=1$; we omit the details.
\end{proof}

\section{Pfaffian point processes}
\label{sec:Pfaffian pt proc}

\subsection{Overview}
Here we provide an overview of the essential definitions and properties of Pfaffian point processes.
\begin{defn}[Pfaffian point process]
	Let $\mathcal{Z}$ be a set of points and let $\mu:2^\mathcal{Z}\to\mathbb{C}$ be a measure on the powerset $2^\mathcal{Z}$. This measure is called a \emph{Pfaffian point process} if there exists a matrix valued kernel $K:\mathcal{Z}\times \mathcal{Z}\to\mathbb{C}^{2\times 2}$ which is skew-symmetric, i.e. $K_{ij}(x,y)=-K_{ji}(y,x)$, given by
	\[K(x,y) = \left[\begin{array}{cc}
		K_{11}(x,y) & K_{12}(x,y) \\
		K_{21}(x,y) & K_{22}(x,y)
		\end{array}\right],\]
    such that correlation functions $\rho(z_1,\dotsc,z_m)\coloneqq\mu(\{A\in2^\mathcal{Z}\!:z_1,\dotsc,z_m\in A\})$ have the form  
	\begin{equation}
		\label{eq:PPP general kernel}
		\rho(z_1,\dots,z_m) = \pf\left[\begin{array}{cc}
			K_{11}(z_i,z_j) & K_{12}(z_i,z_j) \\
			K_{21}(z_i,z_j) & K_{22}(z_i,z_j)
		\end{array}\right]_{1\leq i,j\leq m}.
	\end{equation}
	for all finite subsets $\{z_1,\dots,z_m\}\subseteq\mathcal{Z}$.
    The kernel $K$ is known as the \emph{correlation kernel} of the process.
\end{defn}

If $\mu$ is a Pfaffian point process on $\mathcal{Z}$ and $\mathcal{B}\subseteq\mathcal{Z}$ then the gap probability corresponding to seeing no points in $\mathcal{B}$ is given as a Fredholm Pfaffian:
\begin{equation}\label{eq:pfaffian gap probab}
	\sum_{X\subseteq\mathcal{Z}\setminus\mathcal{B}} \mu(X) = \pf(J-K)_{\ell^2(\mathcal{B})}.
\end{equation}
This follows e.g. from Theorem 8.2 in \cite{rains_correlation_2000}.
The Fredholm Pfaffian of a $2\times2$ matrix kernel $K$ acting on a space $L^2(\mathcal{X},\lambda)$ can be defined through its series expansion
\[\pf(J+K)_{L^2(\mathcal{X})}=1+\sum_{n\geq1}\frac1{n!}\int_{\mathcal{X}^n}\dd\lambda^{\otimes n}(x_1,\dotsc,x_n)\,\pf\big[K(x_i,x_j)\big]_{i,j=1}^n;\]
see Section 8 of \cite{rains_correlation_2000} or Appendix B of \cite{ortmann_pfaffian_2017} for more details on and properties of Fredholm Pfaffians.
The kernel $J$ here is defined as
	\[J(x,y) = \left[\begin{array}{cc}
		0 & 1\\
		-1 & 0
	\end{array}\right]\delta_{x,y},\]
for $x,y\in\mathcal{Z}$. 

\begin{defn}[Pfaffian $L$-ensemble]
	Let $\mathcal{Z}$ be a set of points, and let $L:\mathcal{Z}\times\mathcal{Z}\to \mathbb{C}^{2\times2}$ be a skew-symmetric $2\times2$-matrix valued function. A measure $\mu$ on $2^\mathcal{Z}$ is a \emph{Pfaffian $L$-ensemble} on $\mathbb{Z}$ if its correlation function $\rho$ is given by
	\begin{equation}
		\label{eq:pf L-ensemble}
		\rho(X) = \frac{\Pf L_X}{\Pf(J + L)},\quad X \subset \mathcal{Z}.
	\end{equation}
	Here $L_X$ denotes the $(2|X|)\times(2|X|)$ matrix $[L(x,x')]_{x,x'\in X}$, while $\Pf(J + L)$ denotes the Fredholm Pfaffian of $L$ on $\ell^2(\mathcal{Z})$ (which coincides with the Pfaffian of the matrix $J + L$ when $\mathcal{Z}$ is finite).
\end{defn}

\begin{defn}[Conditional Pfaffian $L$-ensemble]
	\label{defn:cond pf L-ensemble}
	Let $\mathcal{Z}$ be a set of points with subset $\mathcal{D}\subseteq\mathcal{Z}$ and complement $\mathcal{D}^\mathrm{C} = \mathcal{Z}\setminus\mathcal{D}$, and let $L:\mathcal{Z}\times\mathcal{Z}\to \mathbb{C}^{2\times2}$ be a skew-symmetric $2\times2$-matrix valued function. A measure $\mu$ on the powerset of the subset $2^\mathcal{D}$ is a \emph{conditional Pfaffian $L$-ensemble} on $\mathcal{D}$ within $\mathcal{Z}$ if its correlation function $\rho$ is given by
	\begin{equation}
		\label{eq:conditonal pf L-ensemble}
		\rho(Y) = \frac{\pf L_{Y\cup \mathcal{D}^\mathrm{C}}}{\pf\left(J_\mathcal{D}+L\right)},
	\end{equation}
	where $J_\mathcal{D}$ denotes the kernel $J$ restricted to $\mathcal{D}$, i.e.
	\[J_\mathcal{D}(x,y)=J(x,y)\mathbbm{1}_{x,y\in\mathcal{D}}.\]
\end{defn}

Note that the conditional Pfaffian $L$-ensemble \eqref{eq:conditonal pf L-ensemble} reduces to the regular Pfaffian $L$-ensemble \eqref{eq:pf L-ensemble} when the chosen subset is the whole set of points $\mathcal{D}=\mathcal{Z}$.

The following proposition is due to Borodin and Rains, Proposition 1.7 in \cite{borodin_eynardmehta_2005}.

\begin{prop}
	\label{prop:cond pf L is PPP}
        Let $\mathcal{Z}$ be a finite set of points with subset $\mathcal{D}\subseteq\mathcal{Z}$ and let $\mu$ be a conditional Pfaffian $L$-ensemble on $\mathcal{D}$ within $\mathbb{Z}$ as in Definition \ref{defn:cond pf L-ensemble}. Then $\mu$ is a Pfaffian point process on $\mathcal{D}$ whose correlation kernel is given by
	\begin{equation}
		K = J_\mathcal{D} + \left(J_\mathcal{D}+L\right)^{-1}\Big|_{\mathcal{D}\times\mathcal{D}},
	\end{equation}
	where the second term denotes the inverse of the matrix $J_\mathcal{D}+L$ indexed by $\mathcal{Z}\times\mathcal{Z}$ restricted to $\mathcal{D}\times\mathcal{D}$.
\end{prop}

\subsection{Half-space TASEP measure}
We will now proceed to develop the Sch\"{u}tz-type Pfaffian formulas for half-space TASEP from Section \ref{subsec:schutz pf} into a Pfaffian point process. Theorems \ref{thm: transition prob pfaffian kernel} and \ref{thm: transition prob pfaffian kernel initial cond} give the transition probability for both odd and even values of $N+M$. However, for simplicity, we will restrict our analysis to the even case henceforth.
\begin{figure}
    \centering
    \[\begin{tikzpicture}[scale=1.2]
        \draw[line width=1.2pt] (0,0) -- (3,3);
        \draw[line width=1.2pt] (-1,1) -- (1,3);
        \draw[line width=1.2pt] (-2,2) -- (-1,3);
        \draw[line width=1.2pt] (0,0) -- (-3,3);
        \draw[line width=1.2pt] (1,1) -- (-1,3);
        \draw[line width=1.2pt] (2,2) -- (1,3);
        \draw[black,fill=black] (0,0) circle (0.1);
        \node[below left] at (0,0) {$x_1 = z_1^1$};
        \draw[black,fill=black] (-1,1) circle (0.1);
        \node[below left] at (-1,1) {$x_2 = z_1^2$};
        \draw[black,fill=black] (-2,2) circle (0.1);
        \node[below left] at (-2,2) {$x_3 = z_1^3$};
        \draw[black,fill=black] (-3,3) circle (0.1);
        \node[below left] at (-3,3) {$x_4 = z_1^4$};
        \draw[black,fill=black] (1,1) circle (0.1);
        \node[right] at (1,1) {$z_2^2$};
        \draw[black,fill=black] (0,2) circle (0.1);
        \node[right] at (0,2) {$z_2^3$};
        \draw[black,fill=black] (-1,3) circle (0.1);
        \node[right] at (-1,3) {$z_2^4$};
        \draw[black,fill=black] (2,2) circle (0.1);
        \node[right] at (2,2) {$z_3^3$};
        \draw[black,fill=black] (1,3) circle (0.1);
        \node[right] at (1,3) {$z_3^4$};
        \draw[black,fill=black] (3,3) circle (0.1);
        \node[right] at (3,3) {$z_4^4$};
        \node[below,rotate=-45] at (-0.5,0.5) {$<$};
        \node[below,rotate=-45] at (-1.5,1.5) {$<$};
        \node[below,rotate=-45] at (-2.5,2.5) {$<$};
        \node[above,rotate=-45] at (0.5,1.5) {$<$};
        \node[above,rotate=-45] at (-0.5,2.5) {$<$};
        \node[above,rotate=-45] at (1.5,2.5) {$<$};
        \node[below,rotate=45] at (0.5,0.5) {$\leq$};
        \node[below,rotate=45] at (1.5,1.5) {$\leq$};
        \node[below,rotate=45] at (2.5,2.5) {$\leq$};
        \node[below,rotate=45] at (-0.5,1.5) {$\leq$};
        \node[below,rotate=45] at (0.5,2.5) {$\leq$};
        \node[below,rotate=45] at (-1.5,2.5) {$\leq$};
    \end{tikzpicture}\]
    \caption{The Gelfand--Tsetlin pattern $\mathsf{GT}_4(x)$ with left-edges fixed by $x=(x_1,x_2,x_3,x_4)$.}
    \label{fig:GT-pattern}
\end{figure}
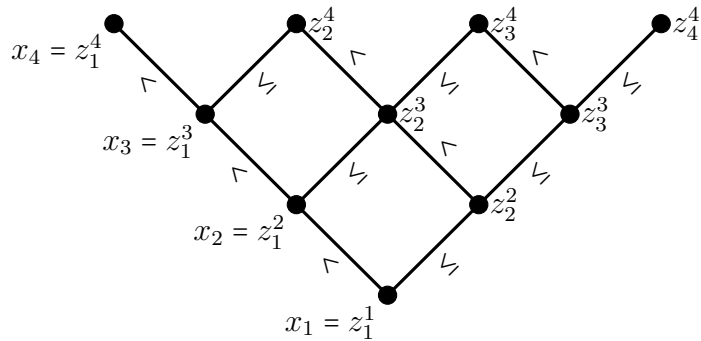

Let us denote the set of all \emph{integer valued triangular arrays} of size $N$ as
\begin{equation}
	\label{eq:triangular arrays lambda}
	\Lambda_N := \left\{\bm{z}=\left(z_i^k\right)\in \mathbb{Z}^{N(N+1)/2}\right\},
\end{equation}
where an individual entry on the triangular array is denoted by the coordinate $z_i^k\in\mathbb{Z}$ for $1\leq i\leq k\leq N$.
We must also define a more restricted subset of $\Lambda_N$: the \emph{Gelfand--Tsetlin patterns} (sometimes spelled Gelfand--Zetlin or Gelfand--Cetlin) which we define as
\begin{equation}
	\mathsf{GT}_N := \left\{\bm{z} = \left(z_i^k\right)\in \mathbb{Z}^{N(N+1)/2}: z_i^{k+1} < z_i^k \leq z_{i+1}^{k+1},\,1\leq i \leq k < N\right\}.
\end{equation}
We also let, for $x=(x_1,\dots,x_N)$ with $x_1>x_2>\dotsm>x_N$,
\begin{equation}
	\label{eq:triangular arrays and GT lambda x}
	\Lambda_N(x) := \{\bm{z}\in\Lambda_N\!:z_1^k=x_k\},\qquad \mathsf{GT}_N(x) := \{\bm{z}\in\mathsf{GT}_N\!:z_1^k=x_k\},
\end{equation}
which are the spaces of triangular arrays and Gelfand--Tsetlin patterns whose left-edges are fixed by the coordinates of $x$. An example of the Gelfand--Tsetlin pattern $\mathsf{GT}_N(x)$ appears in Figure \ref{fig:GT-pattern}.

To simplify the exposition, and in particular the presentation of the proofs, we treat separately the cases of empty and general initial conditions throughout most of the remainder of the paper.
\begin{prop}[Triangular array marginal for empty initial conditions]
	\label{prop:GT pattern empty}
	Let $N$ be an even integer and let $x=(x_1,\dots,x_N)$ be the ordered coordinates of TASEP particles. Then the half-line open TASEP transition probability from an empty state to $x$ can be written as a sum over Pfaffian kernels indexed by Gelfand--Tsetlin patterns with endpoints fixed by $x=(x_1,\dots,x_N)$:
	\begin{equation}
		\label{tasepmeasure}
		\mathbb{P}_t(\emptyset\to x) = (-1)^{\binom{N}{2}}\mathrm{e}^{-\alpha t} \sum_{z\in\mathsf{GT}_N(x)} \Pf\left[\Psi(z_i^N,z_j^N)\right]_{1\leq i,j\leq N},
	\end{equation}
	where we have defined
	\begin{equation}\label{eq:Psi-kernel}
		\Psi(x,y) := Q_{1,1}(x,y)
		= \alpha^2\oint_\contour \frac{\dd u}{2\pi\ii} \oint_\contour \frac{\dd w}{2\pi\ii} \frac{u-w}{1-u-w} \frac{w^{1-x}\mathrm{e}^{t(w-1)}}{(w-\alpha)(w-1)} \frac{u^{1-y}\mathrm{e}^{t(u-1)}}{(u-\alpha)(u-1)},
	\end{equation}
    where contour $\beta$, as in Definition \ref{defn:TASEP countours}, surrounds poles at $u,w=1,0,\alpha,1-\alpha$ and omits all other singularities of the integrand. The kernel $\Psi(x,y)$ is skew-symmetric under the interchange of $x$ and $y$ when $x,y\geq1$.
\end{prop}

The proof of this proposition will use the following result, which appears in \cite{borodin_fluctuation_2007}.

\begin{lm}
	\label{lm:GT sum restriction}
	Let $f$ be an anti-symmetric function in $N$ variables and let $x=(x_1,\dots,x_N)$ where $x_1> \cdots > x_N$ are integers. We have the following identity of sums over triangular arrays:
	\begin{equation}
		\sum_{\bm{z}\in\Lambda_N(x)\atop z_{i+1}^{k+1}\geq z_i^k} f\left(z_1^N,\dots,z_N^N\right) = \sum_{\bm{z} \in \mathsf{GT}_N (x)} f\left(z_1^N,\dots,z_N^N\right).
	\end{equation}
\end{lm}

\begin{proof}[Proof of Proposition \ref{prop:GT pattern empty}]
	The result will follow from an expansion of the Pfaffian from \eqref{eq:half-line TASEP transition prob} where we identify the coordinates $(x_1,\dots,x_N) = \left(z_1^1,\dots,z_1^N\right)$:
	\begin{multline}
		\label{eq:GT expansion proof 1}
		\pf \left[Q_{i,j}(z_1^{N-i+1},z_1^{N-j+1})\right] \\ = \pf \left(
		\begin{array}{ccccc}
			0 & Q_{1,2}\left(z_1^N,z_1^{N-1}\right) & \cdots & Q_{1,N-1}\left(z_1^N,z_1^{2}\right) & Q_{1,N}\left(z_1^N,z_1^{1}\right) \\
			Q_{2,1} (z_1^{N-1},z_1^{N}) & 0 & \cdots & Q_{2,N-1}\left(z_1^{N-1},z_1^{2}\right) & Q_{2,N}\left(z_1^{N-1},z_1^{1}\right) \\
			\vdots & & \ddots & & \vdots \\
			Q_{N-1,1}\left(z_1^{2},z_1^N\right) & Q_{N-1,2} \left(z_1^{2},z_1^{N-1}\right) & \cdots & 0 & Q_{N-1,N} \left(z_1^2,z_1^1\right) \\
			Q_{N,1} \left(z_1^1,z_1^N\right) & Q_{N,2} \left(z_1^1,z_1^{N-1}\right) & \cdots & Q_{N,N-1} \left(z_1^1,z_1^2\right) & 0
		\end{array}
		\right).
	\end{multline}
	Recall the property whereby Pfaffians are linear in expanding along rows and columns simultaneously. We may use the recurrence relations of Lemma \ref{lm:Q-kernel recursion} in the last row and column of \eqref{eq:GT expansion proof 1}. Expansion along this row and column yields
	\begin{multline*}
		\pf \left[Q_{i,j}(z_1^{N-i+1},z_1^{N-j+1})\right] \\ = \sum_{z_2^2\geq z_1^1} \pf \left(
		\begin{array}{ccccc}
			0 & Q_{1,2}\left(z_1^N,z_1^{N-1}\right) & \cdots & Q_{1,N-1}\left(z_1^N,z_1^{2}\right) & Q_{1,N-1}\left(z_1^N,z_2^{2}\right) \\
			Q_{2,1} (z_1^{N-1},z_1^{N}) & 0 & \cdots & Q_{2,N-1}\left(z_1^{N-1},z_1^{2}\right) & Q_{2,N-1}\left(z_1^{N-1},z_2^{2}\right) \\
			\vdots & & \ddots & & \vdots \\
			Q_{N-1,1}\left(z_1^{2},z_1^N\right) & Q_{N-1,2} \left(z_1^{2},z_1^{N-1}\right) & \cdots & 0 & Q_{N-1,N-1} \left(z_1^2,z_2^2\right) \\
			Q_{N-1,1} \left(z_2^2,z_1^N\right) & Q_{N-1,2} \left(z_2^2,z_1^{N-1}\right) & \cdots & Q_{N-1,N-1} \left(z_2^2,z_1^2\right) & 0
		\end{array}
		\right).
	\end{multline*}
	This step can be repeated so that it occurs $N-1$ times in total:
	\begin{multline*}
		\pf \left[Q_{i,j}(z_1^{N-i+1},z_1^{N-j+1})\right] \\ = \sum_{z_N^N\geq \cdots \geq z_2^2\geq z_1^1} \pf \left(
		\begin{array}{ccccc}
			0 & Q_{1,2}\left(z_1^N,z_1^{N-1}\right) & \cdots & Q_{1,N-1}\left(z_1^N,z_1^{2}\right) & Q_{1,1}\left(z_1^N,z_N^{N}\right) \\
			Q_{2,1} (z_1^{N-1},z_1^{N}) & 0 & \cdots & Q_{2,N-1}\left(z_1^{N-1},z_1^{2}\right) & Q_{2,1}\left(z_1^{N-1},z_N^{N}\right) \\
			\vdots & & \ddots & & \vdots \\
			Q_{N-1,1}\left(z_1^{2},z_1^N\right) & Q_{N-1,2} \left(z_1^{2},z_1^{N-1}\right) & \cdots & 0 & Q_{N-1,1} \left(z_1^2,z_N^N\right) \\
			Q_{1,1} \left(z_N^N,z_1^N\right) & Q_{1,2} \left(z_N^N,z_1^{N-1}\right) & \cdots & Q_{1,N-1} \left(z_N^N,z_1^2\right) & 0
		\end{array}
		\right).
	\end{multline*}
	A similar procedure can then be applied to each row and column so that the $j$-th row has the recurrence relations of Lemma \ref{lm:Q-kernel recursion} applied $j-1$ times. This yields 
	\begin{multline*}
		\pf \left[Q_{i,j}(z_1^{N-i+1},z_1^{N-j+1})\right] = \sum_{z_N^N\geq \cdots \geq z_2^2\geq z_1^1} \sum_{z_{N-1}^N\geq \cdots \geq z_2^3\geq z_1^2} \cdots \sum_{z_2^N\geq z_1^{N-1}} \\ \times \pf \left(
		\begin{array}{ccccc}
			0 & Q_{1,1}\left(z_1^N,z_2^{N}\right) & \cdots & Q_{1,1}\left(z_1^N,z_{N-1}^{N}\right) & Q_{1,1}\left(z_1^N,z_N^{N}\right) \\
			Q_{1,1} (z_2^{N},z_1^{N}) & 0 & \cdots & Q_{1,1}\left(z_2^{N},z_{N-1}^{N}\right) & Q_{1,1}\left(z_2^{N},z_n^{N}\right) \\
			\vdots & & \ddots & & \vdots \\
			Q_{1,1}\left(z_{N-1}^{N},z_1^N\right) & Q_{1,1} \left(z_{N-1}^{N},z_2^{N}\right) & \cdots & 0 & Q_{1,1} \left(z_{N-1}^N,z_N^N\right) \\
			Q_{1,1} \left(z_N^N,z_1^N\right) & Q_{1,1} \left(z_N^N,z_2^{N}\right) & \cdots & Q_{1,1} \left(z_N^N,z_{N-1}^N\right) & 0
		\end{array}
		\right),
	\end{multline*}
	which can be identified as a sum over integer valued triangular arrays with endpoints fixed by the coordinates $x_1 > \cdots > x_N$:
	\begin{equation}
		\label{eq:GT expansion proof 2}
		\pf \left[Q_{i,j}(z_1^{N-i+1},z_1^{N-j+1})\right]_{1\leq i.,j\leq N} = \sum_{\bm{z}\in\Lambda_N\atop z_{i+1}^{k+1}\geq z_i^k, z_1^j = x_j} \pf \left[\Psi(z_i^N,z_j^N)\right]_{1\leq i,j\leq N}.
	\end{equation}
	Recall the property whereby a Pfaffian is anti-symmetric under the simultaneous interchange of two rows and columns. Therefore the function defined by 
	\[\pf \left[\Psi(z_i^N,z_j^N)\right]_{1\leq i,j\leq N} \]
	is an anti-symmetric function under the permutation of the alphabet $z_1^N,\dots,z_N^N$. Then, using Lemma \ref{lm:GT sum restriction}, the sum on the right hand side of \eqref{eq:GT expansion proof 2} can be restricted to be of Gelfand--Tsetlin patterns which completes the argument.
\end{proof}

\begin{prop}[Triangular array marginal for general initial conditions]
	\label{prop:GT pattern initial}
	Let $0\leq M\leq N$ be integers such that $N+M$ is even and let $x=(x_1,\dots,x_N)$ and $y=(y_1,\dots,y_M)$ be ordered coordinates of TASEP particles satisfying $y_M>N-M+1$. Then the half-line open TASEP transition probability from $y$ to $x$ can be written as a sum over Pfaffian kernels indexed by Gelfand--Tsetlin patterns with endpoints fixed by $x=(x_1,\dots,x_N)$:
    \begin{equation}\label{eq:TASEP measure initial}
		\mathbb{P}_t(y\to x) = (-1)^{\binom{N}{2}}\mathrm{e}^{-\alpha t} \sum_{z\in\mathsf{GT}_N(x)} \pf\left[\begin{array}{cc}
			\Psi(z_i^N,z_j^N)_{1\leq i,j\leq N} & \Xi_{N-k}(z_i^N)_{1\leq i\leq N\atop 1\leq k\leq M} \\[8pt]
			-\Xi_{N-k}(z_j^N)_{1\leq k\leq M\atop 1\leq j\leq N} & 0
		\end{array}\right],
	\end{equation}
    where we define for each $1\leq k\leq M$ the function\footnote{The choice of indexing $(\Xi_{N-k})_{k=1,\dotsc,M}$ is made to better match notation in the literature, see Remark \ref{rem:Xi notation} below.}:
	\begin{equation}
        \label{eq:Xi func defn GT pattern}
		\Xi_{N-k}(z) = (-1)^k R_{k-M}(z-y_{k}) =  (-1)^k \oint_\gamma \frac{\dd w}{2\pi\ii} \frac{(w-1)^{N-k}\mathrm{e}^{t(w-1)}}{w^{z-y_{k}+N-k+1}},
	\end{equation}
    where the contour $\gamma$ encloses the origin.
\end{prop}

\begin{proof}
    The proof follows a similar argument to the proof of Proposition \ref{prop:GT pattern empty} while using the recurrence relations of Lemma \ref{lm:Q-kernel recursion} for the kernel $U_k$. This leads to a sum over Gelfand--Tsetlin patterns whose terms are given by
    \begin{equation}
    \label{eq:GT pattern initial proof 1}
    (-1)^{\binom{N}{2}}\mathrm{e}^{-\alpha t} \left[\begin{array}{cc}
			\Psi(z_i^N,z_j^N)_{1\leq i,j\leq N} & U_{1-k}(z_i^N-y_{M-k+1})_{1\leq i\leq N\atop 1\leq k\leq M} \\[8pt]
			-U_{1-k}(z_j^N-y_{M-k+1})_{1\leq k\leq M\atop 1\leq j\leq N} & 0
		\end{array}\right].
    \end{equation}
    Now, we may reverse the direction of last $M$ rows and columns so that an element of the upper-right block of \eqref{eq:GT pattern initial proof 1} has elements
    \[U_{k-M}\left(z_i^N-y_k\right),\]
    whose rows are represented by $1\leq i\leq N$ and columns by $1\leq k\leq M$. The process of performing these swaps contributes an overall factor of $(-1)^{\binom{M}{2}}$ to the Pfaffian expression. Following this, we multiply each of the last $M$ rows and columns by $(-1)^k$ for each $1\leq k\leq M$. This multiplication then exactly cancels the overall sign from which the result follows. 
\end{proof}

The goal of the rest of this section is to show that the Pfaffians appearing as summands in Propositions \ref{prop:GT pattern empty} and \ref{prop:GT pattern initial} constitute Pfaffian point processes supported on Gelfand--Tsetlin patterns, whose correlation kernels can be computed thanks to Proposition \ref{prop:cond pf L is PPP}. We now proceed in deriving them, following the analogous derivation for TASEP on the full line \cite{borodin_fluctuation_2007}, adjusting it to the Pfaffian case using the framework developed in \cite{borodin_eynardmehta_2005} and \cite{rains_correlation_2000}.

We are thinking of the summand $\Pf\left[\Psi(z_i^N,z_j^N)\right]_{1\leq i,j\leq N}$ in \eqref{tasepmeasure} and the analogous one in \eqref{eq:TASEP measure initial} as measures on Gelfand--Tsetlin patterns $\bm{z}$.
It will actually be more convenient to extend them to measures over the whole set of integer valued triangular arrays \eqref{eq:triangular arrays lambda}, by simply assigning zero weight to configurations outside $\mathsf{GT}_N$.
To this end we employ an identity due to \cite{borodin_fluctuation_2007}, for any triangular array $\bm{z}\in\Lambda_N$ satisfying $z_1^{k+1}<z_i^k$ for $k=1,\dotsc,N-1$, one has
\begin{equation}
	\label{eq:GT det identity}
	\mathbbm{1}_{\bm{z}\in \mathsf{GT}_N} = \prod_{k=2}^{N}\det[\mathbbm{1}_{z_i^{k-1}>z_j^{k}}]_{1\leq i,j\leq k},
\end{equation}
where $z_2^1,\dots,z_{N}^{N-1}$ play the role of \emph{virtual coordinates} (added due to technical considerations) which should be thought of as taking the value $\infty$, so that $\mathbbm{1}_{z_{k}^{k-1}>y}=1$ for any $y\in\mathbb{Z}$.
We will also add an additional virtual coordinate $z_1^0$ which does not appear in \eqref{eq:GT det identity} so that there are $N$ in total; doing that, the product on the right hand side can be extended up to $k=1$ straightforwardly.
To easily distinguish these virtual coordinates from the coordinates on the triangular array, we will henceforth denote them by $\vir_1,\dotsc,\vir_N$ (so that $\vir_k$ plays the role of $z_k^{k-1}$).

Note that the last row in each determinant in \eqref{eq:GT det identity} has all $1$'s. Subtracting it from each of the other rows flips the inequalities in each indicator function and (after accounting for the minus signs which result from this) leads to the following: letting\footnote{
Note that the only difference between the different $\phi_k$'s is that they are evaluated at different virtual variables. We could well have introduced a single function $\phi$, but keeping the dependence on $k$ makes some of the coming computations a bit more transparent.
\label{ft:phi i dependence}}
\begin{equation}
    \label{eq:phi with v}
	\phi_{k}(x,y) = \begin{dcases*}
		\mathbbm{1}_{x\leq y} & if $x\in\mathbb{Z}$,\\
		1 & if $x=\vir_k$,
	\end{dcases*}
\end{equation}
for $k=1,\dotsc,N$ and for $x\in\mathbb{Z}\cup\{\vir_k\}$ and $y\in\mathbb{Z}$ ($\vir_k$ should now be thought of as taking the value $0$), and assuming $z_1^{k+1}<z_1^k$ for $k=1,\dotsc,N-1$, we have
\begin{equation}
	\label{eq:GT det identity def}
	\mathbbm{1}_{\bm{z}\in \mathsf{GT}_N} = (-1)^{\binom{N}{2}} \prod_{k=1}^{N}\det[\phi_k(z_i^{k-1},z_j^{k})]_{1\leq i,j\leq k},
\end{equation}
with $z^{k-1}_{k}=\vir_k$ for $k=1,\dotsc,N$.
This version of the identity will turn out to be more convenient for us than \eqref{eq:GT det identity} (see Remark \ref{rem:phi flip}).

\begin{defn}
	\label{defn:W-measure}
	Given integers $N\geq1$ and $M\geq0$ such that $N+M$ is even, and given an ordered set of coordinates $y=(y_1,\dots,y_M)$ satisfying $y_M>N-M+1$ ($y=\emptyset$ if $M=0$), we define a measure on $\Lambda_N$ corresponding to half-space TASEP with initial condition $y$ as follows:
    \begin{equation}
        \label{eq:W-measure initial}
        \mathcal{W}_{N,M}(\bm{z}| y) = \prod_{k=1}^{N} \det[\begin{array}{c}
            \phi_k\!\left(z_i^{k-1},z_j^{k}\right)_{1\leq i\leq k-1\atop1\leq j\leq k} \\[8pt]
            \phi_k\!\left(\vir_{k},z_j^{k}\right)_{1\leq j\leq k} 
        \end{array}]\pf\left[\begin{array}{cc}
        \Psi(z_i^N,z_j^N)_{1\leq i,j\leq N} & \Xi_{N-k}(z_i^N)_{1\leq i\leq N\atop 1\leq k\leq M} \\[8pt]
        -\Xi_{N-k}(z_j^N)_{1\leq k\leq M\atop 1\leq j\leq N} & 0
    \end{array}\right].
    \end{equation}
    Due to \eqref{eq:GT det identity def}, the restriction of this measure to $\Lambda_N(x)$, with $x_1>x_2>\dotsm>x_N$, is supported on Gelfand--Tsetlin patterns.	
\end{defn}

Note that when $M=0$, which corresponds to half-space TASEP with empty initial condition, only the block $\Psi(z_i^N,z_j^N)_{1\leq i,j\leq N}$ remains in the Pfaffian on the right hand side of \eqref{eq:W-measure initial}.

\begin{remark}\label{rem:Xi notation}
    Just as in Corollary \ref{cor:schutz TASEP det}, the case $N=M$ recovers the known results for TASEP on the full line.
    In fact, the Pfaffian on the right hand side of \eqref{eq:W-measure initial} becomes $\det[\Xi_{N-k}(z_i^N;t)]_{1\leq i\leq N\atop 1\leq k\leq M}$ thanks to \eqref{eq:block pfaffian is det}, while the $\Xi_{N-k}$'s in this case become, modulo a sign change, the functions $\Psi^N_{N-k}$ appearing in \cite{borodin_fluctuation_2007}.
    In that setting, the family $(\Psi^N_k)_{k=0,\dotsc,N-1}$ can be thought of follows: $\Psi^N_k$ is the Charlier orthogonal polynomial of degree $k$, multiplied by the Charlier (i.e. Poisson) weight, and shifted by $y_{N-k}$.
    In our case (for general $N$), the family $(\Xi_k)_{k=N-M,\dotsc,N-1}$ can be interpreted similarly: $\Xi_k$ is the Charlier orthogonal polynomial of degree $N-M+k$, multiplied by the Charlier weight, and shifted by $y_{M-k}$.
\end{remark}

Using Propositions \ref{prop:GT pattern empty} and \ref{prop:GT pattern initial}, the transition probability of the half-space TASEP can be recovered as a marginal of the measure from Definition \ref{defn:W-measure} as:
\begin{equation}\label{eq:TASEP probab WN}
\begin{aligned}
    \mathbb{P}_t(y\to x) &= \mathrm{e}^{-\alpha t} \sum_{\bm{z}\in\Lambda_N(x)} \mathcal{W}_{N,M}(\bm{z}|y),
\end{aligned}
\end{equation}
for $x$ of size $N\geq1$ (with $x_1>x_2>\dotsm>x_N\geq1$) and $y$ of size $M\geq0$, and with $N+M$ being even.

\subsection{Half-space TASEP as a Pfaffian point process}
In this section, the measures defined in Definition \ref{defn:W-measure} are shown to be conditional Pfaffian $L$-ensembles (see Definition \ref{defn:cond pf L-ensemble}). From this, using Proposition \ref{prop:cond pf L is PPP}, we define a Pfaffian point process. We will first present the case of empty initial conditions using \eqref{eq:W-measure initial} with $y=\emptyset$ (or $M=0$) before presenting the case of more general initial conditions.

\subsubsection{$L$-ensemble for empty initial conditions}
For fixed $N$, consider the spaces
\begin{equation}
	\label{eq:L-ensemble spaces empty}
	\mathcal{X}=\{1,\dots,N\}\times\mathbb{N},\qquad\mathcal{V}=\{\vir_1,\dots,\vir_N\},\qquad\text{and}\qquad\widetilde{\mathcal{X}}=\mathcal{V}\cup\mathcal{X}.
\end{equation}
We think of $\widetilde{\mathcal{X}}$ as the state space of a point process; the $\vir_i$'s correspond to the virtual coordinates, while $(i,x)$ correspond to physical variables, representing a point with label $i$ at $x$. We will refer to a point $(i,x)\in \mathcal{X}$ as belonging to the $i$-th fiber of $\mathcal{X}$. 

We may identify a collection of points $A\subset \mathcal{X}$ as a triangular array, $A\in\Lambda_N$, if $A$ has exactly $i$ elements in its $i$-th fiber for all $1\leq i\leq N$. In this way, we may (and will) identify $\Lambda_N$ as a subset of $\mathcal{X}$. 
For example, the collection $A=\{(1,2),(2,2),(2,3)\}$ is in $\Lambda_2$. 

\begin{thm}
	\label{thm:L-ensemble empty}
	Let $N$ be a fixed even integer. Let $L:\widetilde{\mathcal{X}}\times\widetilde{\mathcal{X}}\to\mathbb{C}^{2\times 2}$ be defined as follows:
	\begin{equation}\label{eq:L-ensemble empty}
	\begin{split}
		L(\virst_i,\virst_j)&\,=\,\left[\begin{array}{cc}
			1 & 0 \\
			0 & 0
		\end{array}\right]\mathbbm{1}_{i+1=j}
		+\left[\begin{array}{cc}
			-1 & 0 \\
			0 & 0
		\end{array}\right]\mathbbm{1}_{j+1=i},\\
		L(\virst_i,(j,x_2))&\,=\,\left[\begin{array}{cc}
			0 & 0 \\
			\phi_{i}(\virst_i,x_2) & 0
		\end{array}\right]\mathbbm{1}_{i=j},\\
		L((i,x_1),\virst_j)&\,=\,-L(\virst_j,(i,x_1))^T,\\
		L((i,x_1),(j,x_2))&\,=\,\left[\begin{array}{cc}
			0 & 0 \\
			\phi_{i}(x_1,x_2) & 0
		\end{array}\right]\mathbbm{1}_{i+1=j<N}+\left[\begin{array}{cc}
			0 & \phi_{j}(x_2,x_1) \\
			0 & 0
		\end{array}\right]\mathbbm{1}_{j+1=i<N} +\left[\begin{array}{cc}
			0 & 0 \\
			0 & \Psi(x_1,x_2)
		\end{array}\right]\mathbbm{1}_{i=j=N},\hspace{-40pt}
	\end{split}
	\end{equation}
	for $1\leq i,j\leq N$ and $x_1,x_2\in\mathbb{N}$.
        Then for $A\subset \mathcal{X}$ we have
	\[\Pf\left[L\big|_{\mathcal{V}\cup A}\right]=\begin{dcases*}
		\mathcal{W}_{N,0}(A|\emptyset) & if $A\in\Lambda_N$,\\
		0 & otherwise.
	\end{dcases*}\]
\end{thm}

\begin{proof}
    The first part of the proof will consist of showing that for $A\in\Lambda_N$, the measure $\mathcal{W}_{N,0}(A|\emptyset)$ defined in \eqref{eq:W-measure initial} can be written as the Pfaffian of the matrix $L\big|_{\mathcal{V}\cup A}$.
	
    The measure $\mathcal{W}_{N,0}(\cdot|\emptyset)$ is defined over triangular arrays $\bm{z}=\left(z^k_i\right)_{1\leq i\leq k}\in\Lambda_N$.
    For $1\leq k\leq N$ let $\mathcal{Z}^{(k)}$ denote the set of points corresponding to the $k$-th row of the triangular array, i.e. $\{z_1^k,\dots,z_k^k\}$, and set $\mathcal{Z}^{(0)}=\emptyset$.
    Let also $\left(\mathcal{Z}^{(k)}\right)'$ and $\left(\mathcal{Z}^{(k)}\right)''$ denote identical copies of $\mathcal{Z}^{(k)}$ (this is just a notational device to help keep track of the two rows and columns associated to the $2\times2$ matrix entries of $L$).
    For $k=0,\dotsc,N-1$ we introduce kernels $V_k$ on $\left(\left(\mathcal{Z}^{(k)}\right)''\cup\{\vir_{k+1}\}\right)\times\left(\mathcal{Z}^{(k+1)}\right)'$  defined by 
    \[V_k\left((z_i^k)'',(z_j^{k+1})'\right) = \phi_{k+1}(z_i^k,z_j^{k+1}), \qquad V_k(\vir_{k+1},(z_j^{k+1})') = \phi_{k+1}(\vir_{k+1},z_j^{k+1}).\]
    We also regard $\Psi$ as being defined on $\left(\mathcal{Z}^{(N)}\right)''\times\left(\mathcal{Z}^{(N)}\right)''$.
    
    Using the Pfaffian identity \eqref{eq:block pfaffian is det}, we may write the measure \eqref{eq:W-measure initial} as
	\begin{equation}
		\label{eq:measure proof 1}
		\mathcal{W}_{N,0}(\bm{z}|\emptyset)=\pf\left[
		\begin{array}{cccccc}
			0 & V_0 & \cdots & 0 & 0 & 0 \\
			-V_0^T & 0 & \cdots & 0 & 0 & 0 \\
			\vdots & \vdots & \ddots & \vdots & \vdots & \vdots \\
			0 & 0 & \cdots & 0 & V_{N-1} & 0 \\
			0 & 0 & \cdots & -V_{N-1}^T & 0 & 0 \\
			0 & 0 & \cdots & 0 & 0 & {\Psi} \\
		\end{array}
		\right](\bm{\widetilde{z}},\bm{\widetilde{z}}),
	\end{equation}
	where the rows and the columns of the matrix are indexed by
	\[\{\vir_1\}\cup\left(\mathcal{Z}^{(1)}\right)' \cup \left(\left(\mathcal{Z}^{(1)}\right)'' \cup \{\vir_2\}\right) \cup\left(\mathcal{Z}^{(2)}\right)'\cup\cdots\cup \left(\left(\mathcal{Z}^{(N-1)}\right)''\cup\{\vir_{N}\}\right) \cup \left(\mathcal{Z}^{(N)}\right)' \cup \left(\mathcal{Z}^{(N)}\right)''\]
	and where $\bm{\widetilde{z}}$ denotes a version of $\bm z$ with duplicated $\mathcal{Z}^{(k)}$ variables for $k=1,\dotsc,N$ and with the virtual variables added, all in the order specified by the above space. The kernel $V_k$ can be split into its physical and virtual coordinate dependence as
	\[V_k = \left[
	\begin{array}{c}
		W_k \\
		e_k
	\end{array}
	\right],\]
	with $W_k$ a kernel on $(\mathcal{Z}^{(k)})''\times(\mathcal{Z}^{(k+1)})'$ and $e_k$ a kernel on $\{\vir_{k+1}\}\times(\mathcal{Z}^{(k+1)})'$, given by 
    \[W_k\left((z^k_i)'',(z^{k+1}_j)'\right)=\phi_{k+1}(z^k_i,z^{k+1}_j),\qquad e_k\left(\vir_{k+1},(z^{k+1}_j)'\right)=\phi_{k+1}(\vir_{k+1},z^{k+1}_j).\]
    We emphasize here that $W_k$ depends only on the physical coordinates of the triangular array.
    Note also that $W_0$ is an empty array.
    
    We aim to define a measure on the space of duplicated physical coordinates
    \[\mathcal{Z}= \left(\mathcal{Z}^{(1)}\right)'\cup \left(\mathcal{Z}^{(1)}\right)'' \cup \cdots \cup \left(\mathcal{Z}^{(N)}\right)' \cup \left(\mathcal{Z}^{(N)}\right)''.\]
    It will be convenient first to duplicate the alphabet of virtual coordinates so that there are two copies $\vir_k',\vir_k''$ of each $\vir_k$. For each $0\leq k\leq N-1$, henceforth we will regard $e_{k}$ as a kernel on $\{\vir_{k+1}''\}\times \left(\mathcal{Z}^{(k+1)}\right)'$.
    The first copy of the virtual coordinates will then be identified through the kernel on $\{\vir_1',\dots,\vir_N'\}\times \{\vir_1',\dots,\vir_N'\}$ defined by the matrix\footnote{In principle, we could have used any skew-symmetric $N\times N$ matrix whose Pfaffian is equal to 1.}
	\begin{equation}
		\label{eq:U-mat defn}
		\Lambda_{N}(\vir_i',\vir_j') = \delta_{i,j+1} - \delta_{i+1,j}.
	\end{equation}
    When $N$ is even, the Pfaffian of the matrix \eqref{eq:U-mat defn} satisfies
    $\pf\left[\Lambda_{N}(\vir_i',\vir_j')\right]_{1\leq i,j\leq N} = 1$, so that \eqref{eq:measure proof 1} can be expressed as
    \begin{equation}	
        \label{eq:measure proof 3}
        \mathcal{W}_{N,0}(\bm{z}|\emptyset) = \pf\left[
        \begin{array}{ccccccccc}
		\Lambda_{N} & 0 & 0 & \cdots & 0 & 0 & 0 \\
            0 & 0 & V_0 & \cdots & 0 & 0 & 0 \\
            0 & -V_0^T & 0 & \cdots & 0 & 0 & 0 \\
            \vdots & \vdots & \vdots & \ddots & \vdots & \vdots & \vdots \\
            0 & 0 & 0 & \cdots & 0 & V_{N-1} & 0 \\
            0 & 0 & 0 & \cdots & -V_{N-1}^T & 0 & 0 \\
            0 & 0 & 0 & \cdots & 0 & 0 & {\Psi} \\
        \end{array}
        \right](\widetilde{\bm z},\widetilde{\bm z}),
    \end{equation}
    where $\widetilde{\bm z}$ is now a version of $\bm{z}$ with duplicated physical and virtual variables.
    Now, recall the property by which a Pfaffian changes sign under the simultaneous interchange of two rows and columns. Performing row and column swaps shows that
    \begin{equation}
        \label{eq:measure proof 4}
        \mathcal{W}_{N,0}(\bm{z}|\emptyset) = \pf \left[
        \begin{array}{cc|ccccccc}
            \Lambda_{N} & 0 & 0 & 0 & 0 & \cdots & 0 & 0 & 0 \\
            0 & 0 & E_0 & 0 & E_1 & \cdots & 0 & E_{N-1} & 0 \\[1pt]
            \hline
            \rule{0pt}{13pt}
            0 & -E_0^T & 0 & 0 & 0 & \cdots & 0 & 0 & 0 \\
            0 & 0 & 0 & 0 & W_1 & \cdots & 0 & 0 & 0 \\
            0 & -E_1^T & 0 & -W_1^T & 0 & \cdots & 0 & 0 & 0 \\
            \vdots & \vdots & \vdots & \vdots & \vdots & \ddots & \vdots & \vdots & \vdots \\
            0 & 0 & 0 & 0 & 0 & \cdots & 0 & W_{N-1} & 0 \\
            0 & -E_{N-1}^T & 0 & 0 & 0 & \cdots & -W_{N-1}^T & 0 & 0 \\
            0 & 0 & 0 & 0 & 0 & \cdots & 0 & 0 & {\Psi} \\
        \end{array}
        \right](\widetilde{\bm z},\widetilde{\bm z}),
    \end{equation}
    where, for each $0\leq k\leq N-1$, the kernel $E_k$ over $\{\vir_1'',\dots,\vir_N''\}\times \left(\mathcal{Z}^{(k)}\right)'$ is defined by
    \[E_{k}\left(\vir_i'',(z_j^{k+1})'\right) = \begin{cases}
        e_{k}(\vir_{k+1},z^{k+1}_j) & \text{ if } i=k+1 \\
        0 & \text{otherwise}
    \end{cases},\]
    for each $1\leq i\leq N,1\leq j\leq k+1$.
    Note here that, since each $\vir_k''$ has to be swapped with all of the duplicated virtual variables of lower index, the number of simultaneous interchanges of rows and columns is even.
    
    The rows and columns of the matrix in \eqref{eq:measure proof 4} are both indexed by the space
    \[\Big(\mathcal{V}'\cup \mathcal{V}''\Big)\cup \Big(\left(\mathcal{Z}^{(1)}\right)'\cup \left(\mathcal{Z}^{(1)}\right)'' \cup \cdots \cup\left(\mathcal{Z}^{(N)}\right)'\cup \left(\mathcal{Z}^{(N)}\right)''\Big)\]
    (and the duplicated variables in $\widetilde{\bm z}$ are expressed now according to this ordering).
    Now we may permute rows and columns simultaneously so that pairs of duplicate variables appear next to each other. This corresponds to applying the permutation appearing in \eqref{eq:matrix shuffle pfaffian} separately to the blocks $\mathcal{V}'\cup \mathcal{V}''$ and $\left(\mathcal{Z}^{(k)}\right)'\cup \left(\mathcal{Z}^{(k)}\right)''$, $k=1,\dotsc,N$, so using that identity we get, after reordering, an extra factor of $(-1)^{\binom{N}{2}+\sum_{k=1}^N\binom{k}{2}}=(-1)^{\binom{N}{3}}$, which equals $1$ since $N$ is even. 
    If we now identify $\bm z$ with a set $A\in\mathcal{X}$, we may regard the resulting identity as the Pfaffian of the minor of a suitable $2\times2$ matrix kernel $L$ over the variables $\mathcal{V}\cup A$.
    The kernel $L$ as given in  \eqref{eq:L-ensemble empty} can now be identified by reading from the entries of \eqref{eq:measure proof 4}.
    This shows that
    \[\Pf\left[L\big|_{\mathcal{V}\cup A}\right]=\mathcal{W}_{N,0}(\bm{z}|\emptyset)\]
    whenever $A\in\Lambda_N$.
    
    It remains to extend the identity to $A\in2^\mathcal{X}\setminus\Lambda_N$, in which case we need to show that the left hand side vanishes. 
    Fix such an $A$	and suppose first that there is no point in $A$ in the $i$-th fiber of $\mathcal{X}$.
    Then the $(2i)$-th row of $L\big|_{\mathcal{V}\cup A}$ (corresponding to the second row coming from the $2\times2$ blocks indexed by row $\vir_i$ in \eqref{eq:L-ensemble empty}) is clearly just zero, which means that the Pfaffian of the matrix has to vanish.
    So in order for the Pfaffian not to vanish there necessarily have to be points in all of the $N$ fibers of $\mathcal{X}$.

    Now suppose that $A$ has at least two points in the first fiber of $\mathcal{X}$, say $(1,a)$ and $(1,b)$, and consider rows $2N+1$ and $2N+3$ of $L\big|_{\mathcal{V}\cup A}$.
    They correspond to the first rows coming from the blocks indexed respectively by rows $(1,a)$ and $(1,b)$.
    Each of these two blocks vanishes except at the second coordinate coming from the $\vir_1$ column, so the two rows are linearly dependent and thus the Pfaffian again vanishes.
    So there has to be exactly one point in the first fiber of $\mathcal{X}$.
    The same argument allows one to show inductively that there have to be $k$ points in the $k$-th fiber.
\end{proof}

Now, we define the following measure on $\mathcal{X}=\{1,\dots,N\}\times\mathbb{N}$:
\begin{equation}
    \label{eq:normalized measure empty}
    \mu_N(A) = \frac{\pf\left[L_{\mathcal{V}\cup A}\right]}{\pf\left[J_\mathcal{X}+L\right]}.
\end{equation}
The following result follows from the last theorem and Proposition \ref{prop:cond pf L is PPP}:

\begin{cor}[Pfaffian point process for empty initial configuration]\label{cor:cond prob kernel empty}
    The measure $\mu_N$ given in \eqref{eq:normalized measure empty} defines a Pfaffian point process with correlation kernel given by
    \begin{equation}\label{eq:cond prob kernel empty}
        K = J_\mathcal{X} + \left(J_\mathcal{X}+L\right)^{-1}\Big|_{\mathcal{X}\times\mathcal{X}}.
    \end{equation}
    Moreover, $\mu_N$ is supported on configurations in $\Lambda_N$, and is proportional to $\mathcal{W}_{N,0}(\cdot|\emptyset)$ (defined in \eqref{eq:W-measure initial} with $M=0$) there.
\end{cor}

\begin{proof}
    Fix $m\in\mathbb{N}$ and let $\mu_N^m$ denote the restriction of the measure $\mu_N$ to $\mathcal{X}_m=\{1,\dotsc,N\}\times\{1,\dotsc,m\}$.
    Proposition \ref{prop:cond pf L is PPP} implies that  $\mu^m_N$ defines a Pfaffian point process with correlation kernel given by
    \[K_m= J_{\mathcal{X}_m} + \left(J_{\mathcal{X}_m}+L\right)^{-1}\Big|_{\mathcal{X}_m\times\mathcal{X}_m}.\]
    This means that given any collection of points $x_1,\dotsc,x_n\in\mathcal{X}_m$,
    \begin{equation}\label{eq:muNM corr}
        \mu_N^m(\{A\subset\mathcal{X}_m\!:x_1,\dotsc,x_n\in A\})=\Pf\big[K_m(x_i,x_j)\big]_{i,j=1}^n.
    \end{equation}
    Now we take $m\to\infty$ on both sides of the identity to get
    \[\mu_N(\{A\subset\mathcal{X}\!:x_1,\dotsc,x_n\in A\})=\Pf\big[K(x_i,x_j)\big]_{i,j=1}^n\]
    with $K$ as in the statement of the result, which means precisely that the point process $\mu_N$ is Pfaffian, with correlation kernel $K$.
    That the left hand side of \eqref{eq:muNM corr} converges as claimed follows directly, while for the right hand side it is enough to use that $\left(J_{\mathcal{X}_m}+L\right)^{-1}\Big|_{\mathcal{X}_m\times\mathcal{X}_m}(x_i,x_j)=\left(J_{\mathcal{X}_m}+L\right)^{-1}(x_i,x_j)$ for each $i,j$ as long as $m>\max\{x_1,\dotsc,x_n\}$ while $J_{\mathcal{X}_m}$ converges to $J_{\mathcal{X}}$ in operator norm and $L$ defines a bounded operator (which can be checked from its definition; we omit the details). 

    The remaining claims follow from Theorem~\ref{thm:L-ensemble empty}.
    \end{proof}

\subsubsection{$L$-ensemble for general initial conditions}

Our goal now is to extend the above construction to the case of general initial configuration.
In this case, in addition to the spaces $\mathcal{X}$ and $\mathcal{V}$ from \eqref{eq:L-ensemble spaces empty}, and for fixed integers $M,N$ and half-space TASEP ordered coordinates $(y_1,\dots,y_M$), we need to introduce the space
\begin{equation}
	\mathcal{U} = \{\vir_{N+1},\dots,\vir_{N+M}\}
\end{equation}
and redefine $\widetilde{\mathcal{X}}$ as
\[\widetilde{\mathcal{X}} = \mathcal{V}\cup\mathcal{U}\cup\mathcal{X}.\]

We state the result corresponding to Theorem \ref{thm:L-ensemble empty} and Corollary \ref{cor:cond prob kernel empty} as a single result:
\begin{thm}[Pfaffian point process for general initial conditions]\label{thm:L-ensemble initial}
    Let $N\geq M\geq 0$ be integers such that $N+M$ is even and fix $y=(y_1,\dotsc,y_M)$ satisfying $y_M>N-M+1$.
    Let $L:\widetilde{\mathcal{X}}\times\widetilde{\mathcal{X}}\to\mathbb{C}^{2\times 2}$ be defined by \eqref{eq:L-ensemble empty} where the domain of $L(\vir_i,\vir_j)$ is extended to $1\leq i,j\leq N+M$. Additionally we require 
    \begin{align}
    \begin{split}
        \label{eq:L-ensemble initial}
        L((i,x),\virst_{N+\ell}) & = \left[\begin{array}{cc}
            0 & 0 \\
            0 & \Xi_{N-\ell} (x)
        \end{array}\right]\mathbbm{1}_{i=N}, \\
        L(\virst_{N+k},(j,x))&\,=\,-L((j,x),\virst_{N+k})^T,
    \end{split}
    \end{align}
    for $1\leq i,j\leq N$, $1\leq k,\ell\leq M$ and $x\in\mathbb{N}$.
    Then
    \begin{equation}
        \mu_{N,M}(A|y) = \frac{\pf\left[L_{\mathcal{V}\cup\mathcal{U}\cup A}\right]}{\pf\left[J_\mathcal{X}+L\right]}.
    \end{equation}
    defines a Pfaffian point process on $\mathcal{X}$ with correlation kernel
    \begin{equation}
        K = J + \left(J_\mathcal{X}+L\right)^{-1}\Big|_{\mathcal{X}\times\mathcal{X}}\label{eq:corr ker initial}
    \end{equation}
    (with $\left(J_\mathcal{X}+L\right)^{-1}$ now computed on the whole space $\widetilde{\mathcal{X}}=\mathcal{V}\cup\mathcal{U}\cup\mathcal{X}$) which is supported on configurations in $\Lambda_N$, and is proportional to $\mathcal{W}_{N,M}(\cdot|y)$ (defined in \eqref{eq:W-measure initial}) there.
\end{thm} 

\begin{proof}
    We will keep the notation employed in the proof of Theorem \ref{thm:L-ensemble empty}.
    We also let $\mathcal{U}'$ and $\mathcal{U}''$ be two copies of the new set of virtual coordinates $\mathcal{U}$ and
    introduce a kernel $\widetilde{\Xi}$ acting on $\left(\mathcal{Z}^{(N)}\right)''\times \mathcal{U}'$ as
    \[\widetilde{\Xi}\left((z_i^N)'',\vir'_{N+\ell}\right) := \Xi_{N-\ell}\left(z_i^N\right),\]
    for each $1\leq i\leq N$ and $1\leq \ell\leq M$, where the function $\Xi_{N-\ell}$ is defined by \eqref{eq:Xi func defn GT pattern} (with explicit dependence on the initial coordinate $y_\ell$).
    
    Now, beginning with the measure \eqref{eq:W-measure initial}, we follow the same procedure as the proof of Theorem \ref{thm:L-ensemble empty}. This leads us to the Pfaffian expression (analogous to \eqref{eq:measure proof 4})
	\begin{equation}
        \label{eq:W-measure initial proof 1}
		\mathcal{W}_{N,M}(\bm{z}|y) = \pf \left[
		\begin{array}{cc|ccccccc|c}
			\Lambda_{N+M} & 0 & 0 & 0 & 0 & \cdots & 0 & 0 & 0 & 0 \\
			0 & 0 & E_0 & 0 & E_1 & \cdots & 0 & E_{N-1} & 0 & 0\\[1pt]
                \hline
                \rule{0pt}{13pt}
			0 & -E_0^T & 0 & 0 & 0 & \cdots & 0 & 0 & 0 & 0\\
			0 & 0 & 0 & 0 & W_1 & \cdots & 0 & 0 & 0 & 0\\
			0 & -E_1^T & 0 & -W_1^T & 0 & \cdots & 0 & 0 & 0 & 0\\
			\vdots & \vdots & \vdots & \vdots & \vdots & \ddots & \vdots & \vdots & \vdots & \vdots\\
			0 & 0 & 0 & 0 & 0 & \cdots & 0 & W_{N-1} & 0 & 0\\
			0 & -E_{N-1}^T & 0 & 0 & 0 & \cdots & -W_{N-1}^T & 0 & 0 & 0\\
			0 & 0 & 0 & 0 & 0 & \cdots & 0 & 0 & {\Psi} & \widetilde{\Xi}\\[1pt]
                \hline
                \rule{0pt}{13pt}
			0 & 0 & 0 & 0 & 0 & \cdots & 0 & 0 & -\widetilde{\Xi}^T & 0
		\end{array}
		\right](\widetilde{\bm z},\widetilde{\bm z}),
	\end{equation}
    whose rows and columns are indexed by the space
    \[\Big(\mathcal{V}'\cup \mathcal{U}'\cup \mathcal{V}''\Big)\cup \Big(\left(\mathcal{Z}^{(1)}\right)'\cup \left(\mathcal{Z}^{(1)}\right)'' \cup \cdots \cup\left(\mathcal{Z}^{(N)}\right)'\cup \left(\mathcal{Z}^{(N)}\right)''\Big)\cup \mathcal{U}''.\]
    Here we are taking $\Lambda_{N+M}$ to be indexed by $\mathcal{V}'\cup \mathcal{U}'$ in the obvious way. Now, we move the $\mathcal{U}''$ rows and columns to end up with a matrix indexed by 
    \begin{equation}\label{eq:tilde X reordered}
        \Big(\mathcal{V}'\cup \mathcal{U}'\cup \mathcal{V}''\cup \mathcal{U}''\Big)\cup \Big(\left(\mathcal{Z}^{(1)}\right)'\cup \left(\mathcal{Z}^{(1)}\right)'' \cup \cdots \cup\left(\mathcal{Z}^{(N)}\right)'\cup \left(\mathcal{Z}^{(N)}\right)''\Big),
    \end{equation}
    leading to 
    \begin{equation}
        \label{eq:W-measure initial proof 2}  
        \mathcal{W}_{N,M}(\bm{z}|y) = \pf\left[\begin{array}{c:cc|cccccccc}
			\Lambda_{N+M} & 0 & 0 & 0 & 0 & 0 & \cdots & 0 & 0 & 0\\
			\hdashline
			     &  &  &  &  &  &  & &  &  & \\[-12pt]
			0 & 0 & 0 & E_0 & 0 & E_1 & \cdots & 0 & E_{N-1} & 0\\
			0 & 0 & 0 & 0 & 0 & 0 &  \cdots & 0 & 0 & -\widetilde{\Xi}^T\\[1pt]
                \hline
                \rule{0pt}{13pt}
			0 & -E_0^T & 0 & 0 & 0 & 0 & \cdots & 0 & 0 & 0\\
			0 & 0 & 0 & 0 & 0 & W_1 & \cdots & 0 & 0 & 0\\
			0 & -E_1^T & 0 & 0 & -W_1^T & 0 & \cdots & 0 & 0 & 0 \\
			\vdots & \vdots  & \vdots & \vdots & \vdots & \vdots & \ddots & \vdots & \vdots & \vdots\\
            0 & 0 & 0 & 0 & 0 & 0 & \cdots & 0 & W_{N-1} & 0 \\
            0 & -E_{N-1}^T & 0 & 0 & 0 & 0 & \cdots & -W_{N-1}^T & 0 & 0 \\
            0 & 0 & \widetilde{\Xi} & 0 & 0 & 0 & \cdots & 0 & 0 & {\Psi} \\
	   \end{array}\right](\widetilde{\bm z},\widetilde{\bm z}).
    \end{equation}
    The arguments and evaluations of the $L$-ensemble \eqref{eq:L-ensemble empty} and \eqref{eq:L-ensemble initial} can be identified from \eqref{eq:W-measure initial proof 2} in the same way as in the proof of Theorem \ref{thm:L-ensemble empty}.
    A careful counting of the number of interchanges of rows and columns gives the same prefactor of $(-1)^{\binom{N}{3}}$, which may now not equal 1 since $N$ may be odd. This shows that, for $A\in\Lambda_N$, 
    \[\Pf\left[L\big|_{\mathcal{V}\cup\mathcal{U}\cup A}\right]=(-1)^{\binom{N}{3}}\mathcal{W}_{N,M}(A|y).\]
    All that is left to prove is that if $A\in2^\mathcal{X}\setminus\Lambda_N$ then $\Pf\left[L\big|_{\mathcal{V}\cup\mathcal{U}\cup A}\right]=0$.
    But this follows from exactly the same argument as the one used in Theorem \ref{thm:L-ensemble empty} (noting that the relevant rows in the argument have been augmented with $0$'s on the $\mathcal{U}'\cup\mathcal{U}''$ coordinates).
\end{proof}

\section{Conditional joint distributions}\label{sec:conditional}

We begin this section by explaining how \eqref{eq:pfaffian gap probab} can be used to extract information about the finite-dimensional distributions of half-space TASEP.
Consider the case of general initial conditions, and fix $1\leq p_1<\dotsm<p_m\leq N$ and $a_1,\dotsc,a_m\geq0$.
The identity \eqref{eq:TASEP probab WN} implies that
\begin{align*}
    \mathbb{P}\left[\bigcap_{k=1}^m \left\{X_{t}(p_k) > a_k\right\},\, \abs{X_t} = N\right]
    &=\mathrm{e}^{-\alpha t}\sum_{x_1>\dotsm>x_N\geq 1: \atop x_{p_k}>a_k,\,k=1,\dotsc,m} \sum_{\bm{z}\in\Lambda_N(x)} \mathcal{W}_{N,M}(\bm{z}|y)\\
    &=\mathrm{e}^{-\alpha t}\sum_{x_1>\dotsm>x_N\geq 1} \sum_{\bm{z}\in\Lambda_N(x)} \prod_{k=1}^m\mathbbm{1}_{z_1^{p_k}>a_k} \mathcal{W}_{N,M}(\bm{z}|y).
\end{align*}
Using Theorem \ref{thm:L-ensemble initial}, this can be rewritten as
\[\mathbb{P}\left[\bigcap_{k=1}^m \left\{X_{t}(p_k) > a_k\right\},\, \abs{X_t} = N\right]
=c_N \sum_{A\in2^\mathcal{X}\setminus\mathcal{B}}\frac{\Pf\big[L\big|_{\mathcal{V}\cup\mathcal{U}\cup A}\big]}{\Pf[J_\mathcal{X}+L]}\bigg|_{v_1=\dotsm=v_N=1}\]
for some $c_N\neq0$ and for $\mathcal{B}=\{B\in2^\mathcal{X}\!:(p_k,x)\in B\text{ for some $k=1,...,m$ and some $x\leq a_k$}\}$.
The summand on the right hand side corresponds to the measure $\mu_{N,M}(A|y)$, which defines a Pfaffian point process thanks to the same theorem, with the correlation kernel $K$ given in \eqref{eq:corr ker initial}.
So \eqref{eq:pfaffian gap probab} implies that
\begin{equation}
        \mathbb{P}\left[\bigcap_{k=1}^m \left\{X_{t}(p_k) > a_k\right\},\, \abs{X_t} = N\right]
        =c_N\Pf\!\left(J-K\right)_{\ell^2(\mathcal{B})}
        =c_N\Pf\!\left(J-\bar\chi_aK\bar\chi_a\right)_{\ell^2(\{p_1,\dotsc,p_k)\times\mathbb{N})},      
  \label{eq: TASEP fd intersect}
\end{equation}
where
\begin{equation}\label{eq: barchi projection}
    \bar\chi_a(p,x)=\mathbbm{1}_{x\leq a_p}.
\end{equation}
Setting $a_k=0$ for all $k$, the identity becomes $\mathbb{P}\left[\abs{X_t} = N\right]=c_N\Pf(J)_{\ell^2(\{p_1,\dotsc,p_k)\times\mathbb{N})}=c_N$\footnote{We note that the inverse of the normalization coefficient $c_N$ coincides with result of Corollary \ref{cor:boundary current schutz pf}}. 
This implies that the Fredholm Pfaffian on the right hand side of \eqref{eq: TASEP fd intersect}, which comes from representing half-space TASEP as a conditional Pfaffian $L$-ensemble, computes the finite-dimensional distributions of the system conditioned on the number of particles at time $t$:
\begin{equation}\label{eq: TASEP gap Pfaffian}
    \mathbb{P}\left[\bigcap_{k=1}^m \left\{X_{t}(p_k) > a_k\right\} \given \abs{X_t} = N\right]
    =\Pf\!\left(J-\bar\chi_aK\bar\chi_a\right)_{\ell^2(\{p_1,\dotsc,p_k)\times\mathbb{N})}.
\end{equation}
The goal of the rest of this section is to obtain an expression for the kernel $K$ appearing on the right hand side of \eqref{eq: TASEP gap Pfaffian}.

\subsection{Preliminaries}

In order to explicitly evaluate the kernel, we must first introduce some notation and some preliminary computations.
Given two kernels $A,B$ we will denote their \emph{half-space convolution} as 
\[(A\conv B)(x,y)=\sum_{\nu\geq1}A(x,\nu)B(\nu,y).\]
Similarly, if $A$ is a kernel and $f$ a function of a single variable, we write
\[(A\conv f)(x)=\sum_{\nu\geq1}A(x,\nu)f(\nu).\]
We will also need to employ a convolution over the whole integers: we will denote the \emph{full-space convolution} as
\[(A\fconv B)(x,y)=\sum_{\nu\in\mathbb{Z}}A(x,\nu)B(\nu,y).\]

We use the half-space convolution to define the following family of kernels on $\mathbb{N}\times\mathbb{N}$: for $1\leq k,\ell,\leq N$,
\begin{equation}
    \label{eq:little phi convolved}
    \phi_{(k,\ell]}(x,y) = \begin{cases}
	(\phi_{k+1} \conv  \cdots \conv  \phi_{\ell})(x,y) & \text{ for }k<\ell, \\
	\id(x,y) & \text{ for } k=\ell, \\
	0 & \text{ for }k>\ell.
	\end{cases}
\end{equation}
These kernels play an important role in the arguments that follow and we will now demonstrate some of their properties. For $k<\ell$ the kernel \eqref{eq:little phi convolved} can be evaluated explicitly as
\begin{equation}
    \label{eq:little phi negative conv proof}
    \phi_{(k,\ell]}(x,y) = \oint_{\gamma_0}\frac{\dd w}{2\pi\ii}\frac{w^{x-y-1}}{(1-w)^{\ell-k}} = \binom{y-x+\ell-k-1}{\ell-k-1} \mathbbm{1}_{x\le y},
\end{equation}
where $\gamma_0$ encloses only the singularity at the origin $w=0$.
This follows from a simple computation which is in fact more general, and which we state as a lemma (see also \cite{matetski_tasep_2023}, Lemma 5.6, for a full space version): 

\begin{lm}
Let $S_1$ and $S_2$ be two kernels defined on $\mathbb{N}\times\mathbb{N}$, which are given by
\[
S_i(x,y)=\frac1{2\pi\ii}\oint_{\gamma}\dd w\,w^{x-y-1}g_i(w), 
\]
where $g_1,g_2$ are complex functions which are both analytic on the disk $\{z\in\mathbb{C}\!:|z|<r\}$ for some $r>0$ and $\gamma$ is any simple closed contour contained in that disk.
Then the sum defining the convolution $S_1\conv S_2$ is absolutely convergent and
\begin{equation}\label{eq: S1S2 conv}
S_1\conv S_2(x,y)=\frac1{2\pi\ii}\oint_{\gamma}\dd w\,w^{x-y-1}g_1(w)g_2(w)
\end{equation}
for all $x,y\in\mathbb{N}$.
\end{lm}

\begin{proof}
We have
\begin{align}
	S_1\conv S_2(x,y)&=\sum_{z\geq1}\frac1{(2\pi\ii)^2}\oint_{\gamma_1}\dd w\oint_{\gamma_2}\dd v\,w^{x-z-1}v^{z-y-1}g_1(w)g_2(v)\\
	&=\frac1{(2\pi\ii)^2}\oint_{\gamma_1}\dd w\oint_{\gamma_2}\dd v\,\frac{w^{x-1}v^{-y}}{w-v}g_1(w)g_2(v)
	=\frac1{2\pi\ii}\oint_{\gamma}\dd v\,v^{x-y-1}g_1(v)g_2(v),
\end{align}
where in the second equality the contours are chosen so that $|w|>|v|$ and in the third one we have computed the residue at $w=v$ (note that there is no residue at $w=0$ because $x\geq1$).
\end{proof}

The formula \eqref{eq:little phi negative conv proof} follows from this and the simple fact that $\phi_i(x,y)=\oint_{\gamma_0}\frac{\dd w}{2\pi\ii}\frac{w^{x-y-1}}{1-w}$ (recall that $\phi_i(x,y)$ actually does not depend on $i$, see the footnote in page \pageref{ft:phi i dependence}).

We introduce another family of kernels on $\mathbb{N}\times\mathbb{N}$, defined for $1\leq j\leq m\leq N$ as
\begin{equation}
    \label{eq:little phi negative}
    \phi_{-(j,m]}(x,y) = \oint_{\gamma_0}\frac{\dd w}{2\pi\ii}w^{x-y-1}(1-w)^{m-j} = (-1)^{y-x}\binom{m-j}{y-x} \mathbbm{1}_{x\le y},
\end{equation}
where $\gamma_0$ again encloses only the singularity at the origin $w=0$. Thanks to \eqref{eq: S1S2 conv}, the kernels \eqref{eq:little phi negative} satisfy the following annihilation relation with those defined by \eqref{eq:little phi convolved}:
\begin{equation}
        \label{eq:little phi negative conv}
        \phi_{(\ell,N]}\conv \phi_{-(j,N]} = \begin{cases}
            \phi_{(\ell,j]} & \text{if}\quad \ell\leq j,\\
            \phi_{-(j,\ell]} & \text{if}\quad \ell>j.
        \end{cases}
\end{equation}
We extend the notation $\phi_{(k,\ell]}$ to $\phi_{[k,\ell]}=\phi_{(k-1,\ell]}$ and $\phi_{[k,\ell)}=\phi_{[k,\ell-1]}$, and we also extend $\phi_{-(k,\ell]}$ analogously.

The kernel $\phi_{[k,\ell]}$, which is originally defined on $\mathbb{N}\times\mathbb{N}$, has a natural extension to virtual coordinates: we define it on $\{\vir_k\}\times\mathbb{N}$ by
\begin{equation}\label{eq:pre phi contour virtual}
    \phi_{[k,\ell]}(\vir_{k},x) = \phi_{k}\conv \phi_{(k,\ell]}(\vir_k,x) = \sum_{\nu\geq1}\phi_{(k,\ell]}(\nu,x)
    =\frac{1}{2\pi\ii}\oint_\gamma \dd w\, \frac{w^{-x}}{(1-w)^{\ell-k+1}}\mathbbm{1}_{k\leq\ell},
\end{equation}
where we computed again a geometric sum, with the contour $\gamma$ chosen so that $|w|<1$.

\begin{remark}\label{rem:phi flip}
    The identity \eqref{eq:pre phi contour virtual} is the reason why it is useful to flip the indicator functions in \eqref{eq:GT det identity}.
    In fact, it is straightforward to check that if the $\phi_k(x,y)$'s were defined as $\mathbbm{1}_{x>y}$ instead of $\mathbbm{1}_{x\leq y}$, the convolution after the second equal sign above would be divergent.
    An alternative, though more cumbersome solution in our setting, would be to introduce additional parameters to our measure which, under the right ordering, make the convolutions convergent, and then to extend the final answer analytically to the original situation; see e.g. \cite{borodin_large_2008,matetski_tasep_2023}, where this is done for full-space models, and where the additional parameters play the role of particle speeds.
\end{remark}

For $x\geq1$ and $k\leq\ell$ the integral in \eqref{eq:pre phi contour virtual} can be computed explicitly as $\binom{x+\ell-k-1}{\ell-k}$, which is a polynomial in $x$ on the positive integers of degree $\ell-k$.
It will be convenient for us to extend the definition of $\phi_{[k,\ell]}(\vir_{k},x)$ to a polynomial in all $x\in\mathbb{Z}$, which may be done as follows:
\begin{equation}\label{eq:phi contour virtual}
	\phi_{[k,\ell]}(\vir_{k},x) 
    = \frac{1}{2\pi\ii}\oint_\gamma \dd w\, \frac{(1-w)^{-x}}{w^{\ell-k+1}}=\frac{(x)_{\ell-k}}{\Gamma(\ell-k+1)},
\end{equation}
where $(x)_n=x(x+1)\dotsm(x+n-1)$ is the Pochhammer symbol (with $(x)_0=1$ and $(x)_n=0$ for $n<0$).
Note that this extension coincides with the right hand side of \eqref{eq:pre phi contour virtual} for all $x\geq1$ and all $k,\ell$, that it vanishes for $k>\ell$, and that it defines a polynomial of degree $\ell-k$ for $k\leq \ell$.

One reason why the extension \eqref{eq:phi contour virtual} will be useful is the following full-space convolution identity\footnote{This identity fails in general if the full-space convolution $\diamond$ is replaced by its half-space version $\conv$; for example, one has $\phi_{[N-1,N]}\conv\phi_{-(N-1,N]}(\vir_{N-1},x)=\mathbbm{1}_{x=1}$, while $\phi_{[N,N-1]}(\vir_N,x)=0$ by definition.
It is not strictly necessary to convolve over all integers to get this identity in our setting, where we are only interested in evaluating at $x\geq1$: convolving over $\mathbb{Z}_{>-N}$ would do.
But since it makes no difference and it makes the presentation simpler, we stick to the full-space convolution in \eqref{eq:little phi negative}.
}
\begin{equation}\label{eq:little phi vir negative conv}
    \phi_{[\ell,N]}\fconv\phi_{-(j,N]}(\vir_\ell,x)=\phi_{[\ell,j]}(\vir_\ell,x)
\end{equation}
for every $j,\ell=1,\dotsc,N$ and any $x\in\mathbb{Z}$, and where $\phi_{-(j,N]}(x,y)$ is being extended to a kernel in $\mathbb{Z}\times\mathbb{Z}$ straightforwardly using \eqref{eq:little phi negative}.
To prove \eqref{eq:little phi vir negative conv}, compute the convolution over $\mathbb{N}$ and over $\mathbb{Z}\setminus\mathbb{N}$ separately through geometric sums to get $\frac1{(2\pi\ii)^2}\oint\!\oint_{|u|<|1-w|}\dd w\,\dd u\,\frac{u^{-x}(u-1)^{N-j}}{w^{N-j+1}(w+u-1)}-\frac1{(2\pi\ii)^2}\oint\!\oint_{|u|>|1-w|}\dd w\,\dd u\,\frac{u^{-x}(u-1)^{N-j}}{w^{N-j+1}(w+u-1)}$, and then shrink the $u$ contour in the second term to cancel the first, picking up a residue a $u=1-w$ which gives $\phi_{[\ell,j]}(\vir_\ell,x)$.

The functions \eqref{eq:little phi convolved} also allow for a natural extension of the kernel \eqref{eq:Psi-kernel}:
\begin{equation}
    \label{eq:Psi-phi mat defn}
    \Psi^{N}_{(k,\ell)} = \left(\phi_{(k,N]}\conv \Psi \conv\left(\phi_{(\ell,N]}\right)^T\right),
\end{equation}
where $\left(\phi_{(\ell,N]}\right)^T(x,y) = \phi_{(\ell,N]} (y,x)$. Noting the form of the functions \eqref{eq:little phi negative conv proof}, the convolutions in \eqref{eq:Psi-phi mat defn} may be explicitly calculated as $\sum_{\nu\geq x}\sum_{\nu'\geq y}\phi_{(k,N]}(x,\nu)\Psi(\nu,\nu')\phi_{(\ell,N]}(\nu',y)$, which can be expressed explicitly as
\begin{multline}\label{eq:PsiNkl contour}
\Psi^{N}_{(k,\ell)}(x,y) = \frac{\alpha^2}{(2\pi\ii)^4}\oint\!\!\!\oint\!\!\!\oint\!\!\!\oint\!\dd u\,\dd\tilde u\,\dd w\,\dd\tilde w\, \frac{1}{(1-\tilde w)^{N-k}}\frac{1}{(1-\tilde u)^{N-\ell}}\\
\times\frac{u-w}{1-u-w}\frac {w^{2-x}}{1-w\tilde w}\frac {u^{2-y}}{1-u\tilde u}\frac{\mathrm{e}^{t(w-1)}}{(w-\alpha)(w-1)} \frac{\mathrm{e}^{t(u-1)}}{(u-\alpha)(u-1)},
\end{multline}
whose contours satisfy $|u \tilde u|>1$ and $|w\tilde w|>1$ so that the $\tilde u$ and $\tilde w$ contours include the singularities at $\tilde w=1/w$ and $\tilde u=1/u$ while the $u$ and $w$ contours need to be large enough and include $0$, $1$, $\alpha$ and $1-\alpha$. The simple poles at $\tilde w=1/w$ and $\tilde u=1/u$ can be evaluated to give
\begin{equation}\label{eq:PsiNkl contour2}
\Psi^{N}_{(k,\ell)}(x,y) = \frac{\alpha^2}{(2\pi\ii)^2}\oint\!\!\!\oint\!\dd u\,\dd w\, \frac{u-w}{1-u-w} \frac{w^{N-k+1-x}}{(w-1)^{N-k+1}}\frac{u^{N-l+1-y}}{(u-1)^{N-\ell+1}}  \frac{\mathrm{e}^{t(w-1)}}{(w-\alpha)(w-1)} \frac{\mathrm{e}^{t(u-1)}}{(u-\alpha)(u-1)}.
\end{equation}
The kernel $\Psi^{N}_{(k,\ell)}$ is defined on $\mathbb{N}\times\mathbb{N}$, but we can extend the definition naturally to virtual variables through
\begin{equation}\label{eq:PsiNkl extended}
\begin{gathered}
    \Psi^{N}_{(k,\ell]}(x,\vir_\ell) = \left(\phi_{(k,N]}\conv \Psi \conv\left(\phi_{[\ell,N]}\right)^T\right)(x,\vir_\ell),\quad
    \Psi^{N}_{[k,\ell)}(\vir_k,y) = \left(\phi_{[k,N]}\conv \Psi \conv\left(\phi_{(\ell,N]}\right)^T\right)(\vir_k,y),\\
    \Psi^{N}_{[k,\ell]}(\vir_k,\vir_\ell) = \left(\phi_{[k,N]}\conv \Psi \conv\left(\phi_{[\ell,N]}\right)^T\right)(\vir_k,\vir_\ell).    
\end{gathered}
\end{equation}

\subsection{Empty initial conditions}
\label{se:EmptyIC}

We are ready to state our main result on the conditional multipoint distribution of half-space TASEP.
We do it first for empty initial data.

\begin{thm}[Conditional multipoint distribution for empty initial condition]
    \label{thm:cond prob Fredholm pf empty}
    Let $N>0$ be a fixed even integer and let $\Phi_0(x),\dots,\Phi_{N-1}(x)$ be a family of polynomials where $\Phi_k$ is of degree $k$. Let these polynomials satisfy the following skew-biorthogonality relations:
    \begin{align}
        \begin{split}
            \label{eq:Big Phi skew-biorhtog rels empty}
            \sum_{x,y=1}^\infty \Phi_{2i}(x) \Psi(x,y) \Phi_{2j}(y) & = 0, \\
            \sum_{x,y=1}^\infty \Phi_{2i+1}(x) \Psi(x,y) \Phi_{2j+1}(y) & = 0, \\
            \sum_{x,y=1}^\infty \Phi_{2i}(x) \Psi(x,y) \Phi_{2j+1}(y) & = -\delta_{i,j},
        \end{split}
    \end{align}
    for all $0\leq i,j< N/2$. These polynomials are specified uniquely up to $N$ free parameters which play no role in the sequel (see Remark \ref{rem:skew nonuniq initial}). 
    Then for a fixed integer $m\geq 0$, let $\{p_1,\dots,p_m\}\subseteq\{1,\dots,N\}$ be indices denoting particle labels such that $p_1<p_2<\cdots<p_m$. The half-line TASEP with empty initial conditions has conditional joint distribution given by a Fredholm Pfaffian of the form
    \begin{equation}
        \label{eq:cond joint dist TASEP empty}
        \mathbb{P}\left[\bigcap_{k=1}^m \left\{X_{t}(p_k) > a_k\right\} \given \abs{X_t} = N\right] = \pf\left(J - \bar\chi_a K\bar\chi_a\right)_{\ell^2(\{p_1,\dots,p_m\}\times \mathbb{N})},
    \end{equation}
    where $\bar\chi_a(p,x) = \mathbbm{1}_{x\leq a_p}$ as in \eqref{eq: barchi projection}, where $a_1,\dotsc,a_k\geq0$.
    The kernel $K$ can be expressed as the sum of two terms
    \begin{equation}\label{eq: K K0 KA}
        K = K^0 - K^\mathsf{A}
    \end{equation}
    which have $2\times2$ block matrix entries given as follows:
    \begin{equation}
            \label{eq:cond prob kernel empty 0}
        K^0(i,x_1;j,x_2) = \left[\begin{array}{cc}
        \Psi_{(i,j)} & -\mathbbm{1}_{i<j}\phi_{(i,j]} \\[2pt]
        \mathbbm{1}_{j<i}\phi_{(j,i]}^T & 0
    \end{array}\right](x_1,x_2)
    \end{equation}
    and
    \begin{equation}
        \label{eq:cond prob kernel empty A}
        \text{$K^\mathsf{A}$}(i,x_1;j,x_2) = \left[\begin{array}{cc} \Psi_{(i,N)}\conv\mathcal{T}_N \conv\Psi_{(N,j)} & -\Psi_{(i,N)}\conv\mathcal{T}_N \fconv\phi_{-(j,N]}\\[3pt]
        \phi_{-(i,N]}^T\fconv\mathcal{T}_N\conv\Psi_{(N,j)} & -\phi_{-(i,N]}^T\fconv\mathcal{T}_N\fconv\phi_{-(j,N]} \end{array}\right](x_1,x_2),
    \end{equation}
    with
    \begin{equation}
        \label{eq:cond prob proof empty T defn statement}
        \mathcal{T}_N(\nu_1,\nu_2) = \sum_{k=0}^{N/2-1}\Phi_{2k}(\nu_1)\Phi_{2k+1}(\nu_2) -\sum_{k=0}^{N/2-1}\Phi_{2k+1}(\nu_1)\Phi_{2k}(\nu_2).
    \end{equation}
\end{thm}

\begin{remark}\label{rem:skew nonuniq initial}
    Choosing the family of polynomials $\{\Phi_0,\dotsc,\Phi_{N-1}\}$ corresponds to specifying $\frac12N(N+1)$ parameters, while \eqref{eq:Big Phi skew-biorhtog rels empty} corresponds to solving $\frac12(N-1)N$ equations, which leaves us with $N$ free parameters.
    On the other hand, a skew-biorthogonalization problem of the form \eqref{eq:Big Phi skew-biorhtog rels empty} always has $N/2$ free parameters, because the solution is preserved under the replacements
    \[\Phi_{2k+1}\longmapsto\Phi_{2k+1}+\beta_{k}\Phi_{2k}\]
    for $k=0,\dotsc,N/2-1$ and any choices of $\beta_0\dotsc,\beta_{N/2-1}\in\mathbb{R}$.
    But the final answer \eqref{eq: K K0 KA} for the kernel $K$ is clearly also preserved under these replacements, due to the structure of $K^\mathsf{A}$, so we can indeed free to choose those $N/2$ parameters freely.
    This leaves us with $N/2$ additional free parameters, but they can be accounted for similarly from the structure of $K^\mathsf{A}$, which only depends on the products $\Phi_{2k}(x)\Phi_{2k+1}(y)$ for $k=0,\dotsc,N/2-1$, and not on each function separately\footnote{So, for instance, we could specify the $N$ free parameters by insisting that $\Phi_{2k+1}(x)$ is monic and its $x^{2k}$ coefficient vanishes.}.
\end{remark}

\begin{remark}
    While the $\Phi_k$'s are defined implicitly in the statement of the theorem as the solution of the skew-biorthogonalization problem \eqref{eq:Big Phi skew-biorhtog rels empty}, in the proof they will be constructed somewhat more explicitly in terms of the functions $\phi_{[j,N]}(\vir_j,\cdot)$ and the skew-Borel decomposition of the matrix $\big[\Psi^N_{[k,\ell]}(\vir_k,\vir_\ell)\big]_{k,\ell=1}^N$, see \eqref{eq:cond prob proof empty Big Phi defn}.
    In Appendix \ref{appdx:skew} we provide formulas which can be used in principle to derive explicit formulas for the $\Phi_k$'s (and, hence, for the solution of the skew-biorthogonalization problem) based on that representation.
\end{remark}

Before presenting the proof of the theorem, we need to state the following simple lemma.

\begin{lm}
	\label{lm:2x2inverse}
	The following skew-symmetric matrix written in block form has an inverse given by
	\begin{equation}
		\left[\begin{array}{cc}
			A & B \\
			-B^T & D
		\end{array}\right]^{-1} = 
		\left[\begin{array}{cc}
			\mathcal{H}^{-1} & -\mathcal{H}^{-1} BD^{-1} \\
			D^{-1}B^T\mathcal{H}^{-1} & D^{-1} - D^{-1} B^T \mathcal{H}^{-1}BD^{-1}
		\end{array}\right],\qquad \mathcal{H}=A+B D^{-1}B^T,
	\end{equation}
	assuming that all relevant inverses exist.
\end{lm}

\begin{proof}[Proof of Theorem \ref{thm:cond prob Fredholm pf empty}]
    Recall that the kernel $K$ is given as
    \begin{equation}\label{eq:K ker pf initial}
        K = J_\mathcal{X} + \left(J_\mathcal{X}+L\right)^{-1}\Big|_{\mathcal{X}\times\mathcal{X}}
    \end{equation}
    with $L$ as in Theorem \ref{thm:L-ensemble empty}.
    Consider the matrix form of the kernel $L$ indexed by the space $\widetilde{\mathcal{X}}\times\widetilde{\mathcal{X}}$.
    It involves duplicate sets of variables as in the proof of Theorem \ref{thm:L-ensemble empty}, which we may order as in that proof\footnote{As seen in that proof, this reordering does not change the sign in front of the Pfaffian (and, in any case, we will undo it once we get to the final formula for $K$ in this proof).}.
    With this ordering convention, $L$ can be expressed as the matrix on the right hand side of \eqref{eq:measure proof 4}, and then $J_\mathcal{X}+L$ is represented by the matrix
    \begin{equation}
        \label{eq:cond prob proof empty 1}  \left(J_\mathcal{X}+L\right)\Big|_{\widetilde{\mathcal{X}}\times\widetilde{\mathcal{X}}} = \left[\begin{array}{cc|ccccccc}
			\Lambda_{N} & 0 & 0 & 0 & 0 & \cdots & 0 & 0 & 0 \\
			0 & 0 & E_0 & 0 & E_1 & \cdots & 0 & E_{N-1} & 0 \\[1pt]
                \hline
                \rule{0pt}{13pt}
			0 & -E_0^T & 0 & \id & 0 & \cdots & 0 & 0 & 0 \\
			0 & 0 & -\id & 0 & W_1 & \cdots & 0 & 0 & 0 \\
			0 & -E_1^T & 0 & -W_1^T & 0 & \cdots & 0 & 0 & 0 \\
			\vdots & \vdots & \vdots & \vdots & \vdots & \ddots & \vdots & \vdots & \vdots \\
			0 & 0 & 0 & 0 & 0 & \cdots & 0 & W_{N-1} & 0 \\
			0 & -E_{N-1}^T & 0 & 0 & 0 & \cdots & -W_{N-1}^T & 0 & \id \\
			0 & 0 & 0 & 0 & 0 & \cdots & 0 & -\id & \Psi \\
	   \end{array}\right].
    \end{equation}
    
    Let us define the following kernels using the blocks of \eqref{eq:cond prob proof empty 1}\footnote{The last kernel is defined explicitly on $\mathcal{X}\times\mathcal{X}$ as $D=J+L|_{\mathcal{X}\times\mathcal{X}}$.}:
    \begin{align*}
        A(\vir_i,\vir_j) & = L(\vir_i,\vir_j), \\
        B(\vir_i,(j,x_2)) & = L(\vir_i,(j,x_2)), \\
        D((i,x_1),(j,x_2)) & = \left[\begin{array}{cc}
            0 & \id(x_1,x_2) \\
            -\id(x_1,x_2) & 0
        \end{array}\right]\delta_{i,j} +  L((i,x_1),(j,x_2)).
    \end{align*}
    In a similar way to \cite{borodin_eynardmehta_2005}, the matrix inverse of the function $D$ is calculated explicitly as
    \begin{equation}\label{eq:Dinv formulas}
    D^{-1}((i,x_1),(j,x_2))=\left[\begin{array}{cc}
		\Psi_{(i,j)}(x_1,x_2) & -\phi_{(i,j]}(x_1,x_2) \\
		(\phi_{(j,i]})^T(x_1,x_2) & 0
    	\end{array}\right].
    \end{equation}
    In order to invert the matrix \eqref{eq:cond prob proof empty 1} using Lemma \ref{lm:2x2inverse} we are required to take products like $B D^{-1}$.
    These are evaluated as matrix products indexed by $\mathcal{X}$, as 
    \[B\conv D^{-1}(\vir_i,(j,x))=\sum_{(k,\nu)\in\mathcal{X}}B(\vir_i,(k,\nu)) D^{-1}((k,\nu),(j,x)) = B(\vir_i,(i,\cdot))\conv D^{-1}((i,\cdot),(j,x)),\]
    for $1\leq i,j\leq N$ and $x\in\mathbb{N}$. 
    Here, due to the form of the $L$-ensemble \eqref{eq:L-ensemble empty}, the matrix product reduces to a half-space convolution.
    Through a slight abuse of notation, we will denote this and similar products as $\left(B\conv D^{-1}\right)(\vir_i,(j,x))$. 
    So, the matrix may be inverted using Lemma \ref{lm:2x2inverse} as
    \begin{equation}
        \label{eq:cond prob proof empty 2}
        \left[\left(J_\mathcal{X}+L\right)^{-1}\right]((i,x_1),(j,x_2)) = \left[D^{-1}-D^{-1}\conv B^T\mathcal{H}^{-1}B\conv D^{-1}\right]((i,x_1),(j,x_2)),
    \end{equation}
    where $\mathcal{H}$ is the $N\times N$-dimensional skew-symmetric matrix where entries are given by $2\times2$ matrices defined by
    \begin{equation}
        \left[\mathcal{H}\right]_{k,\ell} = \left(A+B\conv D^{-1} \conv B^T\right)(\vir_k,\vir_\ell),
    \end{equation}
    and as long as $\mathcal{H}$ is invertible.
    The matrix entries of $\mathcal{H}$ are explicitly evaluated as 
    \[\left[\mathcal{H}\right]_{k,\ell} = \left[\begin{array}{cc}
        \left[\Lambda_{N}\right]_{k,\ell} & 0 \\
        0 & \left[\mathcal{N}\right]_{k,\ell}
    \end{array}\right]
    \]
    where $\Lambda_{N}$ and $\mathcal{N}$ are the $N\times N$-dimensional skew-symmetric matrices defined by \eqref{eq:U-mat defn} and
    \begin{equation}
        \label{eq:cond prob proof empty N defn}
        \left[\mathcal{N}\right]_{k,\ell} = \Psi^N_{[k,\ell]}(\vir_k,\vir_\ell)
    \end{equation}
    respectively.
    Note that this means $\left[\mathcal{N}\right]_{k,\ell} = \phi_{(k,N]}\conv\Psi^N\conv(\phi_{(\ell,N]})^T(\vir_k,\vir_\ell)$, so \eqref{eq:phi contour virtual} and Lemma \ref{lem:Theta} imply that
    \[\left[\mathcal{N}\right]_{k,\ell}=Q_{N-k+2,N-\ell+2}(1,1).\]
    In particular, $\Pf[\mathcal{N}]=(-1)^{N/2}\Pf[Q_{i+1,j+1}(1,1)]_{i,j=1}^N$ which is non-zero by Corollary \ref{cor:boundary current schutz pf} (it being proportional to the probability of there being $N$ particles within the system at time $t$).
    Hence $\mathcal{N}$ is invertible.

    We now turn our attention to the calculation of the second term from \eqref{eq:cond prob proof empty 2}. This is 
    \begin{multline*}
        \left[D^{-1}\conv B^T\mathcal{H}^{-1}B\conv D^{-1}\right]((i,x_1),(j,x_2)) \\
        = \sum_{k,\ell=1}^N \left[\begin{array}{cc}
            0 & \Psi_{(i,k]}(x_1,\vir_k) \\
            0 & \left(\phi_{[k,i]}\right)^T(x_1,\vir_k)
        \end{array}\right] \left[\mathcal{H}^{-1}\right]_{k,\ell} \left[\begin{array}{cc}
            0 & 0 \\
            \Psi_{[\ell,j)}(\vir_\ell,x_2) & -\phi_{[\ell,j]}(\vir_\ell,x_2)
        \end{array}\right].
    \end{multline*}
    The matrix multiplication is calculated as
    \begin{multline}
        \label{eq:cond prob proof empty 3}
        \left[D^{-1}\conv B^T\mathcal{H}^{-1}B\conv D^{-1}\right]((i,x_1),(j,x_2)) \\
        = \sum_{k,\ell=1}^N \left[\begin{array}{cc}
            \Psi_{(i,k]}(x_1,\vir_k) \left[\mathcal{N}^{-1}\right]_{k,\ell} \Psi_{[\ell,j)}(\vir_\ell,x_2) & -\Psi_{(i,k]}(x_1,\vir_k) \left[\mathcal{N}^{-1}\right]_{k,\ell} \phi_{[\ell,j]}(\vir_\ell,x_2) \\
            \left(\phi_{[k,i]}\right)^T(x_1,\vir_k) \left[\mathcal{N}^{-1}\right]_{k,\ell} \Psi_{[\ell,j)}(\vir_\ell,x_2) & -\left(\phi_{[k,i]}\right)^T(x_1,\vir_k) \left[\mathcal{N}^{-1}\right]_{k,\ell} \phi_{[\ell,j]}(\vir_\ell,x_2) 
        \end{array}\right].
    \end{multline}
    Now define the following skew-symmetric kernel $\mathcal{T}_N$ on $\mathbb{Z}\times \mathbb{Z}$:
    \begin{equation}
        \label{eq:cond prob proof empty T defn}
        \mathcal{T}_N(\nu_1,\nu_2) = \sum_{k,\ell=1}^N\left(\phi_{[k,N]}\right)^T(\nu_1,\vir_k) \left[\mathcal{N}^{-1}\right]_{k,\ell} \phi_{[\ell,N]}(\vir_\ell,\nu_2),
    \end{equation}
    Then using \eqref{eq:little phi vir negative conv}, we may rewrite \eqref{eq:cond prob proof empty 3} as
    \begin{multline}
        \label{eq:cond prob proof empty 4}
        \left[D^{-1}\conv B^T\mathcal{H}^{-1}B\conv D^{-1}\right]((i,x_1),(j,x_2)) \\
        = \left[\begin{array}{cc}
            \left(\Psi_{(i,N)}\conv \mathcal{T}_N\conv \Psi_{(N,j)}\right)(x_1,x_2) & -\left(\Psi_{(i,N)}\conv \mathcal{T}_N\fconv \phi_{-(j,N]}\right)(x_1,x_2) \\
            \left(\left(\phi_{-(i,N]}\right)^T\fconv \mathcal{T}_N\conv \Psi_{(N,j)}\right)(x_1,x_2) & -\left(\left(\phi_{-(i,N]}\right)^T\fconv \mathcal{T}_N\fconv \phi_{-(j,N]}\right)(x_1,x_2)
        \end{array}\right].
    \end{multline} 
    Recall now the skew-Borel decomposition of a skew-symmetric matrix (Proposition \ref{prop:skew-borel}) whereby the matrix $\mathcal{N}$ can be uniquely factorized as
    \begin{equation}
        \label{eq:cond prob proof empty N skew-Borel}
        \mathcal{N} = \mathcal{R}_N J_N \mathcal{R}_N^T,
    \end{equation}
    where $\mathcal{R}_N$ is a uniquely determined upper-triangular $N\times N$ matrix and $J_N$ is the $N\times N$ block-diagonal matrix defined by
    \[J_N=\diag\left\{\left[\begin{array}{cc}
    0 & 1 \\
    -1 & 0
    \end{array}\right],\dots,\left[\begin{array}{cc}
    0 & 1 \\
    -1 & 0
    \end{array}\right]\right\}\]
    (so that there are $N/2$ blocks in total). The inverse of $\mathcal{N}$ can be expressed as 
    \begin{equation}\label{eq:N inv}
    \mathcal{N}^{-1} = -\mathcal{R}_N^{-T} J_N \mathcal{R}_N^{-1},
    \end{equation}
    where $\mathcal{R}_N^{-T}$ is the inverse transpose of $\mathcal{R}_N$.
    With this in mind, let us define a family of polynomials $\Phi_0(x),\dots,\Phi_{N-1}(x)$ by 
    \begin{equation}
        \label{eq:cond prob proof empty Big Phi defn}
        \Phi_{N-k}(x) = \sum_{j=1}^N\left[\mathcal{R}_N^{-1}\right]_{k,j}\phi_{[j,N]}(\vir_j,x).
    \end{equation}
    Recall from \eqref{eq:phi contour virtual} that each $\phi_{[j,N]}(\vir_j,x)$ is a polynomial of degree $N-j$ so, since $\mathcal{R}_N^{-1}$ is upper-triangular, $\Phi_{N-k}(x)$ is a polynomial of degree $N-k$, as desired.
    
    The kernel \eqref{eq:cond prob proof empty T defn} can then be written in terms of these functions as
    \begin{equation*}
        \mathcal{T}_N(\nu_1,\nu_2) = - \sum_{a,b=1}^N \Phi_{N-a}(\nu_1) \left[J_N\right]_{a,b}\Phi_{N-b}(\nu_2) 
        = \sum_{k=0}^{N/2-1} \Phi_{2k}(\nu_1)\Phi_{2k+1}(\nu_2) - \sum_{k=0}^{N/2-1} \Phi_{2k+1}(\nu_1)\Phi_{2k}(\nu_2).
    \end{equation*}
    Substituting into \eqref{eq:cond prob proof empty 4} we obtain the kernel of the Fredholm Pfaffian \eqref{eq:cond joint dist TASEP empty} and using the result together with \eqref{eq:Dinv formulas} and \eqref{eq:cond prob proof empty 2} in \eqref{eq:K ker pf initial} leads to the desired formula \eqref{eq: K K0 KA} for $K$, with $K^0=J+D^{-1}$ and $K^\mathsf{A}=D^{-1}\conv B^T\mathcal{H}^{-1}B\conv D^{-1}$.

    To derive the skew-biorthogonality, recall the definition of $\mathcal{N}$ in \eqref{eq:cond prob proof empty N defn} and the skew-Borel factorization \eqref{eq:cond prob proof empty N skew-Borel}. They yield
    \begin{equation*}
        \label{eq:cond prob proof empty 5}
        \left[J_N\right]_{k,\ell}=\left[\mathcal{R}_N^{-1}\mathcal{N}\mathcal{R}_N^{-T}\right]_{k,\ell} = \sum_{i,j=1}^N \left[\mathcal{R}_N^{-1}\right]_{k,i} \Psi_{[i,j]}(\vir_i,\vir_j)\left[\mathcal{R}_N^{-T}\right]_{j,\ell}.
    \end{equation*}
    Using \eqref{eq:PsiNkl extended} and the definition \eqref{eq:cond prob proof empty Big Phi defn} of the functions $\Phi_k$, this identity can now be expressed using the half-space convolution as
    \begin{equation}
        \left[J_N\right]_{k,\ell} = \Phi_{N-k} \conv \Psi \conv \Phi_{N-\ell}.
    \end{equation}
    so the $\Phi_k$'s solve the skew-biorthogonalization problem \eqref{eq:Big Phi skew-biorhtog rels empty}.
\end{proof}

\subsection{General initial conditions}
\label{se:GeneralIC}

In this section we will present an extension of Theorem \ref{thm:cond prob Fredholm pf empty}. That is, a Fredholm Pfaffian expression for the conditional joint distribution of the half-space TASEP in the case of a restricted subset of general deterministic initial conditions. 
In order to state the result we need to introduce a new family of functions: for fixed $1\leq i \leq N$ and $1\leq k\leq M$, we let
\begin{equation}
    \label{eq:Xi i-N-k defn}
    \Xi^{(i)}_{N-k}(x) =\phi_{(i,N]}\conv\Xi_{N-k}(x),
\end{equation}
where $\Xi_{N-k}$ is given by \eqref{eq:Xi func defn GT pattern} (note in particular that when $i=N$ we have $\Xi^{(N)}_{N-k}(x)=\Xi_{N-k}(x)$).
We can compute the convolution explicitly as $\phi_{(i,N]}\conv\Xi_{N-k}(x)=\frac{(-1)^k}{(2\pi\ii)^2}\oint\!\!\!\oint\!\dd w\,\dd u\,\frac{w^{x-1}(u-1)^{N-k}}{(1-w)^{N-i}u^{N-k+1-y_k}}\frac{\mathrm{e}^{t(u-1)}}{uw-1}$ with contours enclosing $w=0$ and $u=0$ but not $w=1$, and such that $|w||u|>1$; note that if take the contours to be circles centered at the origin, this implies that $|u|>1$.
Since $x\geq1$, the $w$ integral has no singularity at the origin, and computing the remaining residue at $w=1/u$ we get
\begin{equation}
    \label{eq:Xi i N-k}
    \Xi^{(i)}_{N-k}(x) = (-1)^i\oint_{\gamma_{0,1}}\frac{\dd u}{2\pi\ii} \,\frac{(1-u)^{i-k}}{u^{x-y_{k}+i-k+1}}\mathrm{e}^{t(u-1)},
\end{equation}
where the contour $\gamma_{0,1}$ now encircles the singularities at both $u=0$ and $u=1$. 
It is important to emphasize here that this function has explicit dependence on the initial particle coordinate $y_k\in\mathbb{N}$ due to the definition of the function $\Xi_{N-k}$ (to be more precise, $\Xi_\ell$ depends explicitly on $y_{N-\ell}$).
It is useful to think of this as a dependence on the virtual coordinate $\vir_{N+k}$, which we may identify with $y_k$, but in any case we will omit this dependence from the notation.

In the proof of the theorem that follows we will also need to use an extension of $\Xi^{(i)}_{N-k}$ to virtual coordinates:
\begin{equation}
    \label{eq:Xi square definition}
    \Xi^{[i)}_{N-k}(\vir_i)=\phi_i\conv\Xi^{(i)}_{N-k}(\vir_i)
    = (-1)^{i+1}\oint_{\gamma_{0,1}}\frac{\dd u}{2\pi\ii} \,\frac{(1-u)^{i-k-   1}}{u^{-y_{k}+i-k+1}}\mathrm{e}^{t(u-1)}.
\end{equation}
From this it follows that, for fixed $k\in\{1,\dotsc,M\}$,
\begin{equation}
    \label{eq:Xi square zero}
    \Xi^{[i)}_{N-k}(\vir_i)=0\quad\text{for all }i>k\quad\text{if }y_k>i-k.
\end{equation}

\begin{thm}
    \label{thm:cond prob Fredholm pf initial}
    Let $0<N\geq M\geq 0$ be fixed integers such that $N+M$ is even.
    Consider two families of polynomials $\Phi_0(x),\dots,\Phi_{N-M-1}(x)$ and $\Upsilon_{N-M}(x),\dots,\Upsilon_{N-1}(x)$ which are such that $\Phi_k$ is of degree $k$ and $\Upsilon_{\ell}$ is of degree $\ell$. Additionally, for all $1\leq j\leq N,1\leq k\leq M$, define
    \begin{equation}
        \label{eq:Upsilon-j conv defn}
        \Upsilon^{(j)}_{N-k} (x) = \Upsilon_{N-k}\fconv\phi_{-(j,N]}(x),
    \end{equation}
    so that in particular $\Upsilon^{(N)}_{N-k}=\Upsilon_{N-k}$.
    Suppose that the first family of polynomials satisfies the following skew-biorthogonality relations:
    \begin{align}
        \begin{split}
            \label{eq:Big Phi skew-biorhtog rels initial}
            \sum_{x,y=1}^\infty \Phi_{2i}(x) \Psi(x,y) \Phi_{2j}(y) & = 0, \\
            \sum_{x,y=1}^\infty \Phi_{2i+1}(x) \Psi(x,y) \Phi_{2j+1}(y) & = 0, \\
            \sum_{x,y=1}^\infty \Phi_{2i}(x) \Psi(x,y) \Phi_{2j+1}(y) & = -\delta_{i,j},
        \end{split}
    \end{align}
    for all $0\leq i,j< (N-M)/2$.
    Next, suppose that the second family of polynomials satisfy the following biorthogonality relation:
    \begin{equation}
        \label{eq:Upsilon-biorhtog rels initial}
        \sum_{x=1}^\infty \Upsilon_{N-k}(x) \Xi_{N-\ell}(x) = \delta_{k,\ell},
    \end{equation}
    for all $1\leq k,\ell\leq M$. Additionally, suppose that both families satisfy the following orthogonality relation:
    \begin{equation}
        \label{eq:Upsilon-Big Phi-biorhtog rels initial}
        \sum_{x,y=1}^\infty \Upsilon_{N-k}(x) \Psi(x,y) \Phi_{i}(y) = 0
    \end{equation}
    for all $1\leq k\leq M$ and $0\leq i< N-M$. Then for a fixed integer $m\geq 0$, let $\{p_1,\dots,p_m\}\subseteq\{1,\dots,N\}$ be indices denoting particles such that $p_1<p_2<\cdots<p_m$. The half-space TASEP with initial conditions given by $y=(y_1,\dots,y_M)$ where $y_M>N-M+1$ has conditional joint distribution given by the Fredholm Pfaffian 
    \begin{equation}
        \label{eq:cond joint dist TASEP initial}
        \mathbb{P}\left[\bigcap_{k=1}^m \left\{X_{t}(p_k) > a_k\right\} \given \abs{X_t} = N\right] = \pf\left(J - \bar\chi_a K \bar\chi_a\right)_{\ell^2(\{p_1,\dots,p_m\}\times \mathbb{N})},
    \end{equation}
    where $\bar\chi_a(p,x) = \mathbbm{1}_{x\leq a_p}$ and $K$ is the kernel given by
    \begin{equation}\label{eq: K K0 KA KB KC}
        K = K^0 - \left(K^\mathsf{A}+K^\mathsf{B}+K^\mathsf{C}\right)
    \end{equation}
    with the blocks on the right hand side specified as follows.
    Firstly, $K^0$ is given by \eqref{eq:cond prob kernel empty 0}, namely
    \begin{equation}
            \label{eq:cond prob kernel initial 0}
            K^0(i,x_1;j,x_2) = \left[\begin{array}{cc}
        \Psi_{(i,j)} & -\mathbbm{1}_{i<j}\phi_{(i,j]}\\[3pt]
        \mathbbm{1}_{j<i}\phi_{(j,i]}^T & 0
    \end{array}\right](x_1,x_2),
    \end{equation}
    while $K^\mathsf{A}$ is given as in \eqref{eq:cond prob kernel empty A} but with $\mathcal{T}_N$ replaced by $\mathcal{T}_{N-M}$, namely
    \begin{equation}
        \label{eq:cond prob kernel initial A}
        \text{$K^{\mathsf{A}}$}(i,x_1;j,x_2) = \left[\begin{array}{cc} \Psi_{(i,N)}\conv\mathcal{T}_{N-M} \conv\Psi_{(N,j)} & -\Psi_{(i,N)}\conv\mathcal{T}_{N-M} \fconv\phi_{-(j,N]}\\[3pt]
        \phi_{-(i,N]}^T\fconv\mathcal{T}_{N-M}\conv\Psi_{(N,j)} & -\phi_{-(i,N]}^T\fconv\mathcal{T}_{N-M}\fconv\phi_{-(j,N]} \end{array}\right](x_1,x_2),
    \end{equation}
    Next, $K^\mathsf{B}$ is given as
    \begin{equation}
        \label{eq:cond prob kernel initial B}
        K^\mathsf{B}(i,x_1;j,x_2) = \left[\begin{array}{cc}
        K^\mathsf{B}_{11}(i,x_1;j,x_2) & K^\mathsf{B}_{12}(i,x_1;j,x_2) \\
        K^\mathsf{B}_{21}(i,x_1;j,x_2) & K^\mathsf{B}_{22}(i,x_1;j,x_2)
    \end{array}\right]
    \end{equation}
    with 
    \begin{align}
    \begin{split}
    \label{eq:cond prob kernel initial B blocks}
    K^\mathsf{B}_{11}(i,x_1;j,y_2) &= \sum_{k=1}^M \left(\Xi^{(i)}_{N-k}(x_1)\left(\Upsilon_{N-k}\hspace{-1pt}\conv\hspace{-1pt}\Psi_{(N,j)}\right)(x_2)+\left(\Psi_{(i,N)}\hspace{-1pt}\conv\hspace{-1pt}\Upsilon_{N-k}\right)(x_1)\,\,\Xi^{(j)}_{N-k}(x_2)\right), \\
    K^\mathsf{B}_{12}(i,x_1;j,x_2) &= -\sum_{k=1}^M \Xi^{(i)}_{N-k}(x_1)\,\,\Upsilon^{(j)}_{N-k}(x_2),\\
    K^\mathsf{B}_{21}(i,x_1;j,x_2) & = \sum_{k=1}^M \Upsilon^{(i)}_{N-k}(x_1)\,\,\Xi^{(j)}_{N-k}(x_2)\\
    K^\mathsf{B}_{22}(i,x_1;j,x_2) &= 0.
    \end{split}
    \end{align}
    Finally, $K^\mathsf{C}$ is given by
    \begin{equation}
        \label{eq:cond prob kernel initial C}
        K^\mathsf{C}(i,x_1;j,x_2) = 
        \left[\begin{array}{cc}
        \sum_{k,\ell=1}^M  \Xi^{(i)}_{N-k}(x_1)\left(\Upsilon_{N-k} \conv \Psi\conv\Upsilon_{N-\ell} \right)\Xi^{(j)}_{N-\ell} (x_2) & 0\\
        0 & 0
        \end{array}\right].
    \end{equation}
\end{thm}

\begin{remark}\label{rem:free params initial}
    Determining the combined family of polynomials $\{\Phi_0,\dots,\Phi_{N-M-1},\Upsilon_{N-M},\dots,\Upsilon_{N-1}\}$ corresponds to specifying $\frac12N(N+1)$ parameters.
    Just as in Theorem \ref{thm:cond prob Fredholm pf empty} (see Remark \ref{rem:skew nonuniq initial}), the family of polynomials $\{\Phi_0,\dots,\Phi_{N-M-1}\}$ is specified uniquely up to $N-M$ parameters which do not affect the Pfaffian since they only appear explicitly in the kernel through the form of $K^\mathsf{A}$ \eqref{eq:cond prob kernel initial A}.
    So it remains to determine the family of polynomials $\{\Upsilon_{N-M},\dots,\Upsilon_{N-1}\}$, a task which ultimately requires specifying $\frac12M(M+1)+(N-M)M$ parameters.
    They are determined partially from the biorthogonality relations \eqref{eq:Upsilon-biorhtog rels initial} for $1\leq k\leq\ell\leq M$, which correspond to $\frac12M(M+1)$ equations.
    
    \noindent The $(N-M)M$ remaining parameters are determined by the orthogonality relations \eqref{eq:Upsilon-Big Phi-biorhtog rels initial}, which yield precisely that number of equations.
    However, we need to make sure that the  outstanding $N-M$ free parameters involved in the choice of the $\Phi_i$ polynomial family do not interfere with the uniqueness of the $\Upsilon_{N-k}$ family (and hence do not affect the $K^\mathsf{B}$ and $K^\mathsf{C}$ parts of the kernel).
    To see this, recall that these $N-M$ free parameters arise from the freedom we have in making the replacements
    \[\Phi_{2k+1}\longmapsto\Phi_{2k+1}+\beta_{k}\Phi_{2k}\]
    for $k=0,\dotsc,(N-M)/2-1$ and any choices of $\beta_0\dotsc,\beta_{(N-M)/2-1}\in\mathbb{R}$ and the simultaneous replacements 
    \[\Phi_{2k+1}\longmapsto\gamma_k\Phi_{2k+1}, \qquad \Phi_{2k}\longmapsto\Phi_{2k}/\gamma_{k}\]
    for $k=0,\dotsc,(N-M)/2-1$ and any choices of $\gamma_0\dotsc,\gamma_{(N-M)/2-1}\in\mathbb{R}\setminus\{0\}$. It is not hard to see that the solution of the $\Upsilon_{N-k}$ polynomials in the orthogonality relations \eqref{eq:Upsilon-Big Phi-biorhtog rels initial} are independent of the values of the free parameters $\beta_i,\gamma_j$ for all $0\leq i,j<(N-M)/2$. 
\end{remark}

\begin{proof}[Proof of Theorem \ref{thm:cond prob Fredholm pf initial}]
    The argument will follow the proof of Theorem \ref{thm:cond prob Fredholm pf empty} very closely. 
    We will assume first for simplicity that $N$ and $M$ are even, and explain later how to extend the proof to the case when $N+M$ remains even but both $N$ and $M$ are odd.
    
    The kernel $K$ is given as
    \[K = J + \left(J_\mathcal{X}+L\right)^{-1}\Big|_{\mathcal{X}\times\mathcal{X}}.\]
    Using \eqref{eq:W-measure initial proof 2}, we may represent $J_\mathcal{X}+L$ as the following matrix
    \begin{equation}
        \label{eq:cond prob proof initial 1}  \left(J_\mathcal{X}+L\right)\Big|_{\widetilde{\mathcal{X}}\times\widetilde{\mathcal{X}}} = \left[\begin{array}{c:cc|ccccccc}
			\Lambda_{N+M} & 0 & 0 &  0 & 0 & 0 & \cdots & 0 & 0 & 0\\
			\hdashline
			     &  &  &  &  &  & &  &  & \\[-12pt]
			0 & 0 & 0  & E_0 & 0 & E_1 & \cdots & 0 & E_{N-1} & 0\\
			0 & 0 & 0  & 0 & 0 & 0 &  \cdots & 0 & 0 & -\widetilde{\Xi}^T\\
			\hline
			     &  &  &  &  &  & &  &  & \\[-12pt]
			0 & -E_0^T & 0 & 0 & \id & 0 & \cdots & 0 & 0 & 0\\
			0 & 0 & 0 & -\id & 0 & W_1 & \cdots & 0 & 0 & 0\\
			0 & -E_1^T & 0 & 0 & -W_1^T & 0 & \cdots & 0 & 0 & 0 \\
			\vdots & \vdots & \vdots & \vdots & \vdots & \vdots & \ddots & \vdots & \vdots & \vdots\\
            0 & 0  & 0 & 0 & 0 & 0 & \cdots & 0 & W_{N-1} & 0 \\
            0 & -E_{N-1}^T & 0 & 0 & 0 & 0 & \cdots & -W_{N-1}^T & 0 & \id \\
            0 & 0 & \widetilde{\Xi} & 0 & 0 & 0 & \cdots & 0 & -\id & \widetilde{\Psi} \\
	   \end{array}\right].
    \end{equation}
    Note that we are presenting the right hand side with rows and columns in the order prescribed by \eqref{eq:tilde X reordered}; this incurs in a sign change in front of the Pfaffian, which will be undone once we get to the expressions for $K$ in the theorem, so for simplicity we omit it.
    
    We define the following kernels as extensions of those that appear in the proof of Theorem \ref{thm:cond prob Fredholm pf empty}:
    \begin{align*}
        \widetilde{A}(\vir_m,\vir_n) & = L(\vir_m,\vir_n),\\
        \widetilde{B}(\vir_m,(j,x)) & = L(\vir_m,(j,x)),
    \end{align*}
    for each $1\leq m,n\leq N+M$ and $1\leq j\leq N,x\in\mathbb{N}$ and where the $L$-ensemble is defined by \eqref{eq:L-ensemble empty} and \eqref{eq:L-ensemble initial}. We emphasize here that the definitions of $\widetilde{A},\widetilde{B}$ agree with the definitions of $A,B$ from the proof of Theorem \ref{thm:cond prob Fredholm pf empty} when $1\leq m,n\leq N$, but are otherwise extended to incorporate the extended virtual alphabet. The kernel denoted by $D$ remains the same as in Theorem \ref{thm:cond prob Fredholm pf empty}.
    
    The $\mathcal{X}\times\mathcal{X}$ block of the inverse of the matrix \eqref{eq:cond prob proof initial 1} can be written, using Lemma \ref{lm:2x2inverse}, as
    \begin{equation}
        \label{eq:cond prob proof initial 2} 
        \left(J_\mathcal{X}+L\right)^{-1}\Big|_{\mathcal{X}\times\mathcal{X}} = D^{-1} - D^{-1}\conv\widetilde{B}^T\widetilde{\mathcal{H}}^{-1}\widetilde{B}\conv D^{-1}
    \end{equation}
    with $\widetilde{\mathcal{H}} = \widetilde{A}+\widetilde{B}\conv D^{-1} \conv \widetilde{B}^{T}$.
    This matrix is represented as a matrix of $2\times2$ blocks which is explicitly calculated to yield
    \begin{equation}
        \left[\widetilde{\mathcal{H}}\right]_{a,b} = \begin{dcases}
                \left[\begin{array}{cc}
                     \left[\Lambda_{N+M}\right]_{a,b} & 0 \\
                    0 & \left[\mathcal{N}\right]_{a,b}
                \end{array}\right] & \text{ if }a,b\leq N, \\
                \left[\begin{array}{cc}
                    \left[\Lambda_{N+M}\right]_{a,b} & 0 \\
                    0 & -\left[\mathcal{P}\right]_{a,b-N}
                \end{array}\right] & \text{ if }a\leq N,b>N, \\        
                \left[\begin{array}{cc}
                     \left[\Lambda_{N+M}\right]_{a,b} & 0\\
                    0 & \left[\mathcal{P}^T\right]_{a-N,b} \\
                \end{array}\right] & \text{ if }a> N,b\leq N, \\ 
                \left[\begin{array}{cc}
                    \left[\Lambda_{N+M}\right]_{a,b} & 0 \\
                    0 & 0
                \end{array}\right] & \text{ if }a,b> N,
        \end{dcases}
    \end{equation}
    where $\mathcal{N}$ is the $N\times N$-dimensional skew-symmetric matrix defined by \eqref{eq:cond prob proof empty N defn} while the $N\times M$-dimensional matrix $\mathcal{P}$ is defined by
    \begin{equation}
        \label{eq:cond prob proof initial P defn}
        \left[\mathcal{P}\right]_{a,b} = \Xi^{[a)}_{N-b}(\vir_a)
    \end{equation}
    for $1\leq a\leq N$, $1\leq b\leq M$, where the function $\Xi^{[a)}_{N-k}(\vir_a)$ is defined by \eqref{eq:Xi square definition}. 
    The inverse of $\widetilde{\mathcal{H}}$ can be calculated using Lemma \ref{lm:2x2inverse} again, after suitably reordering rows and columns, giving
    \begin{equation}\label{eq:tilde H inverse}
        \left[\widetilde{\mathcal{H}}^{-1}\right]_{a,b} = \begin{dcases}
                \left[\begin{array}{cc}
                    \left[\Lambda_{N+M}^{-1}\right]_{a,b} & 0 \\
                    0 & \left[\mathcal{N}^{-1}-\mathcal{N}^{-1} \mathcal{P}\mathcal{M}^{-1}\mathcal{P}^T\mathcal{N}^{-1}\right]_{a,b}
                \end{array}\right] & \text{ if }a,b\leq N \\
                \left[\begin{array}{cc}
                    \left[\Lambda_{N+M}^{-1}\right]_{a,b} & 0 \\
                    0 & \left[\mathcal{N}^{-1}\mathcal{P}\mathcal{M}^{-1}\right]_{a,b-N}
                \end{array}\right] & \text{ if }a\leq N,b>N \\        
                \left[\begin{array}{cc}
                    \left[\Lambda_{N+M}^{-1}\right]_{a,b} & 0 \\
                    0 & -\left[\mathcal{M}^{-1}\mathcal{P}^T\mathcal{N}^{-1}\right]_{a-N,b}
                \end{array}\right] & \text{ if }a> N,b\leq N \\ 
                \left[\begin{array}{cc}
                    \left[\Lambda_{N+M}^{-1}\right]_{a,b} & 0 \\
                    0 & \left[\mathcal{M}^{-1}\right]_{a-N,b-N}
                \end{array}\right] & \text{ if }a,b> N
        \end{dcases},
    \end{equation}
    where $\mathcal{M} = \mathcal{P}^T\mathcal{N}^{-1}\mathcal{P}$.
    $\mathcal{M}$ is $M\times M$-dimensional and skew-symmetric, and we will see shortly that it is in fact invertible (recall also that $\mathcal{N}$ is invertible under our assumptions, as was checked in the proof of Theorem \ref{thm:cond prob Fredholm pf empty}).
    
    And so, the Pfaffian kernel can be identified, using \eqref{eq:cond prob proof initial 2}, as
    \[K = K^0 - \left(K^\mathsf{A}+K^\mathsf{B}+K^\mathsf{C}\right) =J+D^{-1} - D^{-1}\conv\widetilde{B}^T\mathcal{H}^{-1}\widetilde{B}\conv D^{-1}.\]
    The first term in the kernel is identified as $K^0 = J+D^{-1}$.
    This is explicitly calculated as \eqref{eq:cond prob kernel initial 0} using the explicit formula for $D^{-1}$ given by \eqref{eq:Dinv formulas} (and is the same as $K^0$ in the case of empty initial condition).
    
    Ultimately, in order to calculate the remaining terms in the kernel we will identify $K^\mathsf{A}+K^\mathsf{B}+K^\mathsf{C} = D^{-1}\conv \widetilde{B}^T\widetilde{\mathcal{H}}^{-1}\widetilde{B}\conv D^{-1}$.    
    This may be written explicitly as
    \begin{equation}\label{eq:cond prob proof initial 3}
    \begin{aligned}
        &\left[D^{-1}\conv \widetilde{B}^T\widetilde{\mathcal{H}}^{-1}\widetilde{B}\conv D^{-1}\right]((i,x_1),(j,x_2)) \\
        &\qquad= \sum_{a,b=1}^N \left[\begin{array}{cc}
            0 & \Psi_{(i,a]}(x_1,\vir_a) \\[2pt]
            0 & \left(\phi_{[a,i]}\right)^T(x_1,\vir_a)
        \end{array}\right] \left[\widetilde{\mathcal{H}}^{-1}\right]_{a,b} \left[\begin{array}{cc}
            0 & 0 \\
            \Psi_{[b,j)}(\vir_b,x_2) & -\phi_{[b,j]}(\vir_b,x_2)
        \end{array}\right] \\
        &\hskip0.8in+ \sum_{a=1}^N\sum_{d=1}^M \left[\begin{array}{cc}
            0 & \Psi_{(i,a]}(x_1,\vir_a) \\[2pt]
            0 & \left(\phi_{[a,i]}\right)^T(x_1,\vir_a)
        \end{array}\right] \left[\widetilde{\mathcal{H}}^{-1}\right]_{a,N+d} \left[\begin{array}{cc}
            0 & 0\\
            -\Xi^{(j)}_{N-d}(x_2) & 0
        \end{array}\right] \\
        &\hskip0.8in+ \sum_{c=1}^M\sum_{b=1}^N \left[\begin{array}{cc}
           0 & -\Xi^{(i)}_{N-c}(x_1) \\
            0 & 0
        \end{array}\right] \left[\widetilde{\mathcal{H}}^{-1}\right]_{N+c,b} \left[\begin{array}{cc}
            0 & 0 \\
            \Psi_{[b,j)}(\vir_b,x_2) & -\phi_{[b,j]}(\vir_b,x_2)
        \end{array}\right] \\
        &\hskip0.8in+ \sum_{c,d=1}^M \left[\begin{array}{cc}
            0 & -\Xi^{(i)}_{N-c}(x_1) \\
            0 & 0
        \end{array}\right] \left[\widetilde{\mathcal{H}}^{-1}\right]_{N+c,N+d} \left[\begin{array}{cc}
           0 & 0\\
            -\Xi^{(j)}_{N-d}(x_2) & 0
        \end{array}\right].
    \end{aligned}
    \end{equation}        
    Recall from the proof of Theorem \ref{thm:cond prob Fredholm pf empty} that the matrix $\mathcal{N}$ has the skew-Borel decomposition given by \eqref{eq:cond prob proof empty N skew-Borel} whereby the upper-triangular matrix $\mathcal{R}_N$ defines the family of skew-orthogonal polynomials $\Phi_0(x),\dots,\Phi_{N-1}(x)$ by
    \begin{equation}
        \label{eq:cond prob proof initial Big Phi defn}
        \Phi_{N-k}(x) = \sum_{j=1}^N\left[\mathcal{R}_N^{-1}\right]_{k,j}\phi_{[j,N]}(\vir_j,x).
    \end{equation}
    We emphasize here that, while this defines a family of $N$ functions $\Phi_k$, only the first $N-M$ will explicitly appear in final expression for the kernel.
    
    Now consider the matrix $\mathcal{P}$ defined by \eqref{eq:cond prob proof initial P defn}. Using \eqref{eq:Xi square zero}, and due to the restriction on the initial coordinates $y_M>N-M+1$, which implies that $y_k>N-k+1$ for each $1\leq k\leq M$, the matrix $\mathcal{P}$ may be decomposed as
    \begin{equation}
        \label{eq:cond prob proof initial P-factor}
        \mathcal{P} = \left[\begin{array}{c}
           \mathcal{S}_M \\ 0
        \end{array}\right],
    \end{equation}
    where $\mathcal{S}_M$ is an upper-triangular $M\times M$-dimensional matrix.
    Its diagonal entries are calculated using \eqref{eq:Xi square definition} as $\left[\mathcal{S}_M\right]_{k,k}= \Xi^{[k)}_{N-k}(\vir_k) = (-1)^{k+1}\oint_{\gamma_{0,1}}\frac{\dd u}{2\pi\ii} \,\frac{1}{(1-u)u^{-y_{k}+i-k+1}}\mathrm{e}^{t(u-1)}=(-1)^k$ for $1\leq k\leq M$ (again due to our condition on $y_M$), so it follows that this matrix is invertible.
    
    From this and the skew-Borel decomposition of
    $\mathcal{N}$, the skew-symmetric matrix $\mathcal{M}$ is given by 
    \[\mathcal{M} = -\left[\begin{array}{cc}
        \mathcal{S}_M^T & 0 
    \end{array}\right]\hspace{0.005em}\mathcal{R}_N^{-T}J_N\hspace{0.005em}\mathcal{R}_N^{-1} \left[\begin{array}{c}
         \mathcal{S}_M \\
         0
    \end{array}\right].\]
     Since the matrix $\mathcal{R}_N$ in the skew-Borel decomposition of $\mathcal{N}$ is upper-triangular, one checks directly that if $\widehat{\mathcal{R}}_M^N$ is the $M\times M$-dimensional minor of $\mathcal{R}_N$ consisting of only the first $M$ rows and columns\footnote{In general, the minor $\widehat{\mathcal{R}}_M^N$ is not equal to the matrix $\mathcal{R}_M$ which emerges from the skew-Borel decomposition of the analogous matrix $\mathcal{N}$ for a system of size $M$. Rather, $\mathcal{R}_M$ is the minor consisting of the \emph{last} $M$ rows and columns of $\mathcal{R}_N$.}, then the top left $M\times M$ block of $\mathcal{R}_N^{-1}$ is $(\widehat{\mathcal{R}}_M^N)^{-1}$, and thus the above can be written as 
     \[\mathcal{M} = -\mathcal{S}_M^T\hspace{0.005em} (\widehat{\mathcal{R}}_M^N)^{-T} J_M (\widehat{\mathcal{R}}_M^N)^{-1}\mathcal{S}_M.\]
     Thus the inverse of $\mathcal{M}$ can be calculated as 
    \begin{equation}
        \label{eq:cond prob proof initial M-inv}
        \mathcal{M}^{-1} = \mathcal{S}_M^{-1}\widehat{\mathcal{R}}_M^N J_M (\widehat{\mathcal{R}}_M^N)^T\mathcal{S}_M^{-T}.
    \end{equation}
    Let us now define the family of functions $\Upsilon_{N-M}(x),\dots,\Upsilon_{N-1}(x)$ by 
    \begin{equation}
        \label{eq:cond prob proof initial Upsilon defn}
        \Upsilon_{N-k}(x) = \sum_{j=1}^N \left[\mathcal{S}_M^{-1}\widehat{\mathcal{R}}_M^N\left[\begin{array}{cc} I_M & 0 \end{array}\right]\mathcal{R}^{-1}_N\right]_{k,j}\phi_{[j,N]}(\vir_{j},x)
    \end{equation}
    for each $1\leq k\leq M$.
    
    We now proceed to evaluate each of the sums in \eqref{eq:cond prob proof initial 3} individually in order to calculate the terms $K^\mathsf{A},K^\mathsf{B},K^\mathsf{C}$ in the Pfaffian kernel explicitly. 
    \begin{enumerate}[leftmargin=*, labelindent=4pt, itemsep=3pt]
        \item[$(\mathsf{A})$] The function $K^\mathsf{A}(i,x_1;j,x_2)$ is identified with the first sum of \eqref{eq:cond prob proof initial 3} which is over $1\leq a,b\leq N$. The calculation of this sum follows in a similar way to that of the proof of Theorem \ref{thm:cond prob Fredholm pf empty}. 
        We note that, using \eqref{eq:cond prob proof initial M-inv}, we may express the product $\mathcal{N}^{-1} \mathcal{P}\mathcal{M}^{-1}\mathcal{P}^T\mathcal{N}^{-1}$ as
        \begin{multline*}
        \qquad\mathcal{R}_N^{-T} J_N \mathcal{R}_N^{-1}  \left[\begin{array}{c}
               \mathcal{S}_M \\ 0
            \end{array}\right] \mathcal{S}_M^{-1}\widehat{\mathcal{R}}_M^N J_M (\widehat{\mathcal{R}}_M^N)^T\mathcal{S}_M^{-T} \left[\begin{array}{cc}
            \mathcal{S}_M^T & 0 
            \end{array}\right] \mathcal{R}_N^{-T} J_N \mathcal{R}_N^{-1}\\
            = \mathcal{R}_N^{-T} J_N \mathcal{R}_N^{-1} \left[\begin{array}{c}
               \widehat{\mathcal{R}}_M^N \\ 0
            \end{array}\right] J_M \left[\begin{array}{cc}
            (\widehat{\mathcal{R}}_M^N)^T & 0 
             \end{array}\right] \mathcal{R}_N^{-T} J_N \mathcal{R}_N^{-1}
         = -\mathcal{R}_N^{-T}\left[\begin{array}{cc}
                J_M & 0 \\
            0 & 0
        \end{array}\right] \mathcal{R}_N^{-1}.
        \end{multline*}
        In this calculation we have used the fact that $\mathcal{R}_N^{-1}  \left[\begin{smallmatrix}
           \widehat{\mathcal{R}}_M^N \\ 0
        \end{smallmatrix}\right] =  \left[\begin{smallmatrix}
           I_M \\ 0
        \end{smallmatrix}\right]$,
        which follows from the upper-triangularity of $\mathcal{R}_N$. 
        And so, we may express
        \begin{equation}
            \label{eq:cond prob proof initial sum 1 1}
            \mathcal{N}^{-1}-\mathcal{N}^{-1} \mathcal{P}\mathcal{M}^{-1}\mathcal{P}^T\mathcal{N}^{-1} = -\mathcal{R}_N^{-T} \left[\begin{array}{cc}
            0 & 0 \\
            0 & J_{N-M}
        \end{array}\right] \mathcal{R}_N^{-1}.
        \end{equation}
        Now let us define a kernel $\mathcal{T}_\mathsf{A}:\mathbb{Z}\times \mathbb{Z}\to \mathbb{C}$ by
        \begin{equation}
            \label{eq:cond prob proof empty T1 defn}
            \mathcal{T}_\mathsf{A}(\nu_1,\nu_2) = \sum_{k,\ell=1}^N\left(\phi_{[k,N]}\right)^T(\nu_1,\vir_k) \left[\mathcal{N}^{-1}-\mathcal{N}^{-1} \mathcal{P}\mathcal{M}^{-1}\mathcal{P}^T\mathcal{N}^{-1}\right]_{k,\ell} \phi_{[\ell,N]}(\vir_\ell,\nu_2).
        \end{equation}
        Using this kernel, we may write $K^\mathsf{A}$ given by the first sum in \eqref{eq:cond prob proof initial 3} as
        \begin{equation}
            \label{eq:cond prob proof initial sum 1 2}
            \begin{aligned}
            &K^\mathsf{A}((i,x_1),(j,x_2)) \\
            &\qquad= \left[\begin{array}{cc}
                \left(\Psi_{(i,N]}\conv \mathcal{T}_\mathsf{A}\conv \Psi_{[N,j)}\right)(x_1,x_2) & -\left(\Psi_{(i,N]}\conv \mathcal{T}_\mathsf{A}\conv \phi_{-(j,N]}\right)(x_1,x_2) \\
                \left(\left(\phi_{-(i,N]}\right)^T\conv \mathcal{T}_\mathsf{A}\conv \Psi_{[N,j)}\right)(x_1,x_2) & -\left(\left(\phi_{-(i,N]}\right)^T\conv \mathcal{T}_\mathsf{A}\conv \phi_{-(j,N]}\right)(x_1,x_2)
            \end{array}\right],
            \end{aligned}
        \end{equation}
        where, using the matrix decomposition \eqref{eq:cond prob proof initial sum 1 1} and \eqref{eq:cond prob proof initial Big Phi defn},
        the kernel \eqref{eq:cond prob proof empty T1 defn} can be expressed as
        \begin{align*}
            \mathcal{T}_\mathsf{A}(\nu_1,\nu_2) & = - \sum_{a,b=M+1}^N \Phi_{N-a}(\nu_1) \left[J_N\right]_{a,b}\Phi_{N-b}(x) \\
            & = \sum_{k=0}^{(N-M)/2-1} \Phi_{2k}(\nu_1)\Phi_{2k+1}(\nu_2) - \sum_{k=0}^{(N-M)/2-1} \Phi_{2k+1}(\nu_1)\Phi_{2k}(\nu_2),
        \end{align*}
        so that $\mathcal{T}_\mathsf{A}$ coincides with the definition of $\mathcal{T}_{N-M}$.
        After substituting this into \eqref{eq:cond prob proof initial sum 1 2} we may identify the form of the kernel \eqref{eq:cond prob kernel initial A}.
        
        \item[$(\mathsf{B})$] The function $K^\mathsf{B}(i,x_1;j,x_2)$ is identified with the sum of the second and third sums of \eqref{eq:cond prob proof initial 3} which are over $1\leq a\leq N,1\leq d\leq M$ and $1\leq c\leq M,1\leq b\leq N$ respectively. We will explicitly calculate the third sum since the second sum follows in a similar fashion via skew-symmetry.

        We have
        \begin{equation*}
            \mathcal{M}^{-1}\mathcal{P}^T\mathcal{N}^{-1}  
             = - \mathcal{S}_M^{-1}\widehat{\mathcal{R}}_M^N J_M \!\left( \widehat{\mathcal{R}}_M^N\right)^T\!\mathcal{S}_M^{-T} \left[\begin{array}{cc}
            \mathcal{S}_M^T & 0 
            \end{array}\right] \mathcal{R}_N^{-T} J_N \mathcal{R}_N^{-1} 
            = \mathcal{S}_M^{-1}\widehat{\mathcal{R}}_M^N \left[\begin{array}{cc}
                I_M & 0 
            \end{array}\right] \mathcal{R}_N^{-1}.
        \end{equation*}
        We may then calculate the third sum of \eqref{eq:cond prob proof initial 3} as
        \begin{align*}
            &\sum_{c=1}^M\sum_{b=1}^N \left[\begin{array}{cc}
            0 & -\Xi^{(i)}_{N-c}(x_1) \\
            0 & 0
        \end{array}\right] \left[\begin{array}{cc}
            0 & 0 \\
            0 & -\left[\mathcal{S}_M^{-1}\widehat{\mathcal{R}}_M^N \left[\begin{array}{cc}
                I_M & 0 
            \end{array}\right] \mathcal{R}_N^{-1}\right]_{c,b}
        \end{array}\right] \\
        &\hspace{3in}\times
        \left[\begin{array}{cc}
            0 & 0 \\
            \Psi_{[b,j)}(\vir_b,x_2) & -\phi_{[b,j]}(\vir_b,x_2)
        \end{array}\right].
        \end{align*}
        Using the definition of the family of functions $\Upsilon_{N-k}$ \eqref{eq:cond prob proof initial Upsilon defn}, the previous expression may be written as
        \begin{align*}
            &\sum_{c=1}^M\sum_{k=1}^M \left[\begin{array}{cc}
            0 & -\Xi^{(i)}_{N-c}(x_1) \\
            0 & 0
            \end{array}\right] \left[\begin{array}{cc}
            0 & 0 \\
            0 & \left[I_M\right]_{c,k}
            \end{array}\right]  
             \left[\begin{array}{cc}
            0 & 0 \\
            -\left(\Upsilon_{N-k}\conv \Psi_{(N,j)}\right)(x_2) & \left(\Upsilon_{N-k}\fconv \phi_{-(j,N]}\right)(x_2)
        \end{array}\right]\\
        &\hspace{0.8in}=\sum_{k=1}^M \left[\begin{array}{cc}
            \Xi^{(i)}_{N-k}(x_1) \left(\Upsilon_{N-k}\conv \Psi_{(N,j)}\right)(x_2) & -\Xi^{(i)}_{N-k}(x_1)\left(\Upsilon_{N-k}\fconv \phi_{-(j,N]}\right)(x_2) \\
            0 & 0
        \end{array}\right],
        \end{align*}
        in which the function $\Upsilon^{(j)}_{N-k}$ defined by \eqref{eq:Upsilon-j conv defn} may be identified. 
        Now accounting for both the second and third sums of \eqref{eq:cond prob proof initial 3}, we may express the kernel term as the expression \eqref{eq:cond prob kernel initial B} for $K^\mathsf{B}$.
       
        \item[$(\mathsf{C})$] The function $K^\mathsf{C}(i,x_1;j,x_2)$ is identified with the fourth sum of \eqref{eq:cond prob proof initial 3} which is over $1\leq c,d\leq M$. This is explicitly calculated as 
        \begin{equation}
            \label{eq:cond prob proof initial sum C 1}
            K^\mathsf{C}(i,x_1;j,x_2)=\sum_{c,d=1}^M \left[\begin{array}{cc}
                \Xi^{(i)}_{N-c}(x_1)\left[\mathcal{M}^{-1}\right]_{c,d}\Xi^{(j)}_{N-d}(x_2) & 0 \\
                0 & 0
            \end{array}\right].
        \end{equation}
        To see that this coincides with the expression for $K^\mathsf{C}$ given by \eqref{eq:cond prob kernel initial C} we need to show that
        \[\left[\mathcal{M}^{-1}\right]_{c,d}=\Upsilon_{N-c}\conv\Psi\conv\Upsilon_{N-d}.\]
        In order to check this we use the definition of the functions $\Upsilon_k$ to write the right hand side as
        \begin{align}
            \qquad&\sum_{i,j=1}^N \left[\mathcal{S}_M^{-1}\widehat{\mathcal{R}}_M^N\left[\begin{array}{cc} I_M & 0 \end{array}\right]\mathcal{R}^{-1}_N\right]_{c,i}
            \phi_{[i,N]}(\vir_{i},\cdot)
            \conv\Psi\conv(\phi_{[j,N]})^T (\cdot,\vir_{j})\left[\mathcal{R}^{-T}_N\left[\begin{array}{c} I_M \\ 0 \end{array}\right](\widehat{\mathcal{R}}_M^N)^T\mathcal{S}_M^{-T}\right]_{j,d}\\
            &\qquad=\sum_{i,j=1}^N \left[\mathcal{S}_M^{-1}\widehat{\mathcal{R}}_M^N\left[\begin{array}{cc} I_M & 0 \end{array}\right]\mathcal{R}^{-1}_N\right]_{c,i}\left[\mathcal{N}\right]_{i,j}\left[\mathcal{R}^{-T}_N\left[\begin{array}{c} I_M \\ 0 \end{array}\right](\widehat{\mathcal{R}}_M^N)^T\mathcal{S}_M^{-T}\right]_{j,d}\\
            &\qquad=\left[\mathcal{S}_M^{-1} \widehat{\mathcal{R}}_M^N J_M (\widehat{\mathcal{R}}_M^N)^T\mathcal{S}_M^{-T}\right]_{c,d},
        \end{align}
        where we used again the skew-Borel decomposition of $\mathcal{N}$.
        In view of \eqref{eq:cond prob proof initial M-inv}, the right hand side is exactly $[\mathcal{M}^{-1}]_{c,d}$.
    \end{enumerate}
    
    With these calculations, we have explicitly determined the Pfaffian kernel in terms of known functions and the new functions $\Phi_0,\dotsc,\Phi_{N-M-1}$ and $\Upsilon_{N-M},\dotsc,\Upsilon_{N-1}$.
    In order to finish proving that the kernel coincides with the one in the statement of the theorem, we need to prove that these functions are polynomials of the correct degrees and that they satisfy their defining relations \eqref{eq:Big Phi skew-biorhtog rels initial}, \eqref{eq:Upsilon-biorhtog rels initial} and \eqref{eq:Upsilon-Big Phi-biorhtog rels initial}.

    The functions $\Phi_0,\dotsc,\Phi_{N-M-1}$ were defined by \eqref{eq:cond prob proof initial Big Phi defn} in the exact same way as the extended family $\Phi_0,\dots,\Phi_{N-1}$ was in the proof of Theorem \ref{thm:cond prob Fredholm pf empty}, so, the first $N-M$ of these functions are polynomials of the correct degrees and they inherit their skew-orthogonality relations in exactly the same manner.
    Note that the relations \eqref{eq:Big Phi skew-biorhtog rels initial} for the first $N-M$ functions are self-contained, so these functions are determined by these relations (up to $N-M$ free parameters which are not consequential) as in the previous proof.
    
    That $\Upsilon_\ell$ is a polynomial of degree $\ell$ follows from \eqref{eq:cond prob proof initial Upsilon defn} because $\big[\mathcal{S}_M^{-1}\widehat{\mathcal{R}}_M^N\big[I_M \;\, 0\big]\mathcal{R}^{-1}_N\big]_{k,j}=0$ for $j>k$ and each $\phi_{[j,N]}(\vir_j,x)$ is a polynomial of degree $N-j$.
    In order to determine the relation \eqref{eq:Upsilon-biorhtog rels initial} we use the decomposition of the matrix $\mathcal{P}$ from \eqref{eq:cond prob proof initial P-factor} to write $\left(\phi_{[j,N]}\conv\Xi_{N-\ell}\right)(\vir_{j}) = \left[\begin{smallmatrix}
         \mathcal{S}_M \\
         0 
    \end{smallmatrix}\right]_{j,\ell}$, so that
    \begin{align}
        \label{eq:cond prob proof initial 4}
        \sum_{x=1}^\infty \Upsilon_{N-k}(x) \Xi_{N-\ell}(x) &= \sum_{j=1}^N\left[\mathcal{S}_M^{-1}\widehat{\mathcal{R}}_M^N\left[\begin{array}{cc} I_M & 0 \end{array}\right]\mathcal{R}^{-1}_N\right]_{k,j} \left(\phi_{[j,N]}\conv\Xi_{N-\ell}\right)(\vir_{j})\\
        & = \left[\mathcal{S}_M^{-1}\widehat{\mathcal{R}}_M^N\left[\begin{array}{cc} I_M & 0 \end{array}\right]\mathcal{R}^{-1}_N\left[\begin{smallmatrix}
         \mathcal{S}_M \\
         0 
    \end{smallmatrix}\right]\right]_{k,\ell}
    = \left[\left[\begin{array}{cc} \mathcal{S}_M^{-1} & \triangle \end{array}\right]\left[\begin{smallmatrix}
         \mathcal{S}_M \\ 0 
    \end{smallmatrix}\right]\right]_{k,\ell}=\delta_{k,\ell},
    \end{align}
    where in the third equality we used once again the fact the top left $M\times M$ block of $\mathcal{R}_N^{-1}$ is $(\widehat{\mathcal{R}}_M^N)^{-1}$ and where the value of the $\triangle$ block is irrelevant.
    
    To prove the skew-orthogonality relation between the two families of functions, \eqref{eq:Upsilon-Big Phi-biorhtog rels initial}, we write, for $1\leq k\leq M$ and $1\leq i\leq N-M$,
    \begin{multline*}
        \sum_{x_1,x_2=1}^\infty \Upsilon_{N-k}(x_1) \Psi(x_1,x_2)\Phi_i(x_2) = \sum_{j=1}^N \left[\mathcal{S}_M^{-1}\widehat{\mathcal{R}}_M^N\left[\begin{array}{cc} I_M & 0 \end{array}\right]\mathcal{R}^{-1}_N\right]_{k,j} \phi_{[j,N]}\conv\Psi\conv\Phi_i(\vir_j)\\
        = \sum_{\ell=1}^N \left[\mathcal{S}_M^{-1}\widehat{\mathcal{R}}_M^N\left[\begin{array}{cc} I_M & 0 \end{array}\right]\right]_{k,\ell}\Phi_{N-\ell}\conv\Psi\conv\Phi_i 
        = \left[\mathcal{S}_M^{-1}\widehat{\mathcal{R}}_M^N\left[\begin{array}{cc} J_M & 0 \end{array}\right]\right]_{k,N-i}\hspace{0.4in}\mbox{}
    \end{multline*}
    where we have used both the definition and the skew-biorthogonality relation of the entire family of functions $\{\Phi_{0},\dots,\Phi_{N-1}\}$.
    The relation \eqref{eq:Upsilon-Big Phi-biorhtog rels initial} follows by noting that the matrix entries of the final expression above are equal to zero whenever $i\leq N-M$.

    To finish the proof we need to extend the result to the case of odd $N$ and $M$.
    We will do this by comparing with a system with $\bar M=M+1$ initial particles and $\bar N=N+1$ final ones, placing an additional initial particle very far to the right. 
    To this end, given the initial particle positions $y=(y_1,\dotsc,y_M)$, we introduce a new collection of particle positions $\bar y=(\bar y_1,\dotsc,\bar y_{\bar M})$ by
    \[\bar y_1=L,\qquad \bar y_i=y_{i-1},\quad i=2,\dotsm M+1,\]
    where $L$ is an auxiliary parameter which we will ultimately take to infinity.
    Note that, for large $L$, $\bar y$ satisfies $\bar y_k>\bar N-k+1$ thanks to the analogous conditions for $y$.
    Then from the case with $N$ and $M$ even, 
    \begin{equation}
        \label{eq:barN barM}\mathbb{P}_{\bar y}\left[\bigcap_{k=1}^m \left\{X_{t}(\bar p_k) > a_k\right\} \given \abs{X_t} = \bar N\right] = \pf\left(J - \bar\chi_a\bar K \bar\chi_a\right)_{\ell^2(\{\bar p_1,\dots,\bar p_m\}\times \mathbb{N})}
        = \pf\left(J - \bar\chi_a\hat K\bar\chi_a\right)_{\ell^2(\{p_1,\dots,p_m\}\times \mathbb{N})},
    \end{equation}
    where we have included the initial condition $\bar y$ explicitly in the subscript, where $\bar p_i=p_i+1$ with $1\leq p_1<\dotsm<p_m\leq N$ and where $\bar K$ denotes the kernel $K$ computed above, with $N$ and $M$ replaced by $\bar N$ and $\bar M$ and $\hat K(i,\cdot;j,\cdot)=\bar K(i+1,\cdot;j+1,\cdot)$.
    Now for very large $L$, the first particle (which starts at $L$) is very likely to not interact with the rest of the system by time $t$: in fact, the probability that the second particle is blocked by the first one at any moment before time $t$ is bounded by the probability that a Poisson random variable with parameter $t$ is greater than $L-y_1-2$.
    Using this, and after suitably shifting the particle labels, it is not hard to see that the left hand side above converges as $L\to\infty$ to
    \[\mathbb{P}_{y}\left[\bigcap_{k=1}^m \left\{X_{t}(p_k) > a_k\right\} \given \abs{X_t} = N\right].\]
    So what remains is to compute the limit of the Fredholm Pfaffian on the right hand side of \eqref{eq:barN barM}.
    The argument is relatively simple, so we will only sketch it. 
    
    Consider each of the terms making up $\hat K$.
    The first one is $\hat K^0$, which does not depend on $N$, $M$ or $L$, and which clearly satisfies $\hat K^0=K^0$.
    The second one is given by \eqref{eq:cond prob kernel initial A}, and now reads
    \[\text{$\hat K^{\mathsf{A}}$}(i,x_1;j,x_2) = \left[\begin{array}{cc} \Psi_{(i+1,\bar N)}\conv\mathcal{T}_{N-M} \conv\Psi_{(\bar N,j+1)} & -\Psi_{(i+1,\bar N)}\conv\mathcal{T}_{N-M} \fconv\phi_{-(j+1,\bar N]}\\[3pt]
        \phi_{-(i+1,\bar N]}^T\fconv\mathcal{T}_{N-M}\conv\Psi_{(\bar N,j+1)} & -\phi_{-(i+1,\bar N]}^T\fconv\mathcal{T}_{N-M}\fconv\phi_{-(j+1,\bar N]} \end{array}\right](x_1,x_2).\]
    Note that it also does not depend on $L$ and that, in fact, $\hat K^{\mathsf{A}}=K^{\mathsf{A}}$.
    
    For the remaining two terms we need to compute limits in $L$.
    But this is simple once one notes that $\Xi_{\bar N-1}(x)\xrightarrow[L\to\infty]{}0$, as follows easily from \eqref{eq:Xi func defn GT pattern} since $y_1=L$.
    So, for instance,
    \[
    \lim_{L\to\infty}\hat K^\mathsf{B}_{12}(i,x_1;j,x_2) = -\sum_{k=2}^{\bar M} \Xi^{(i+1)}_{\bar N-k}(x_1)\,\,\Upsilon^{(j+1)}_{\bar N-k}(x_2)= K^\mathsf{B}_{12}(i,x_1;j,x_2).\]
    Similar expressions hold for the other limits, and we deduce that $\lim_{L\to\infty}\hat K^\mathsf{B}=K^\mathsf{B}$.
    In a similar way we get $\lim_{L\to\infty}\hat K^\mathsf{C}=K^\mathsf{C}$.
    A bit more work shows that the limit holds in a strong enough sense (e.g. trace class) so that it can be taken inside the Fredholm Pfaffian.
    Using this on the right hand side of \eqref{eq:barN barM} yields the claimed formula.
\end{proof}

\subsection{Reduction to full-space TASEP}

One important consequence of Theorem \ref{thm:cond prob Fredholm pf initial} is that the conditional joint distribution of the half-space TASEP reduces the well-known joint distribution of the full-space TASEP. In this way, Theorem \ref{thm:cond prob Fredholm pf initial} may be regarded as a generalization of the Fredholm determinant expressions obtained in the literature.

\begin{cor}[of Theorem \ref{thm:cond prob Fredholm pf initial}]\label{cor:full space TASEP}
    Let $N>0$ be a fixed integer and let $y=(y_1,\dots,y_N)$ be an ordered set of coordinates, and assume that $y_N>1$.
    Moreover, let $\Upsilon_0(x),\dots,\Upsilon_{N-1}(x)$ be a family of polynomials where $\Upsilon_k$ is of degree $k$. Let these polynomials satisfy the biorthogonality relation \eqref{eq:Upsilon-biorhtog rels initial}. The full-space TASEP with initial conditions given by $y$ has joint distribution given by the Fredholm determinant expression
    \begin{equation}\label{eq:full cond prob Fredholm pf initial}
         \mathbb{P}\left[\bigcap_{k=1}^m \left\{X_{t}(p_k) > a_k\right\}\right] = \det\!\left(I - \bar\chi_a K \bar\chi_a\right)_{\ell^2(\{p_1,\dots,p_m\}\times \mathbb{N})},
    \end{equation}
    whose kernel is given by 
    \begin{equation}\label{eq:full-space TASEP det kernel}
        K(i,x_1;j,x_2) = -\mathbbm{1}_{i<j}\phi_{[i,j)}(x_1,x_2) +\sum_{k=1}^j \Xi^{(i)}_{N-k}(x_1)\Upsilon^{(j)}_{N-k}(x_2).
    \end{equation}
    More generally, if $y=(y_1,\dots,y_N)$ is any ordered set of coordinates in $\mathbb{Z}$, the full-space TASEP joint distribution with initial conditions given by $y$ can be written as (note that the Fredholm determinant is now over $\ell^2(\{p_1,\dots,p_m\}\times \mathbb{Z})$)
    \begin{equation}\label{eq:full cond prob Fredholm pf initial 2}
         \mathbb{P}\left[\bigcap_{k=1}^m \left\{X_{t}(p_k) > a_k\right\}\right] = \det\!\left(I - \bar\chi_a K \bar\chi_a\right)_{\ell^2(\{p_1,\dots,p_m\}\times \mathbb{Z})},
    \end{equation}
    and the kernel $K$ can be expressed as 
    \begin{equation}\label{eq:full-space TASEP det kernel 2}
        K(i,x_1;j,x_2) = -\mathbbm{1}_{i<j}\phi_{[i,j)}(x_1,x_2) +\sum_{k=1}^j f^{i}_{i-k}(x_1)g^{j}_{j-k}(x_2),
    \end{equation}
    with
    \[f^i_{i-k}(x)=(-1)^k\Xi^{(i)}_{N-k}(x),\]
    which we regard as being defined on $\mathbb{Z}$ directly through \eqref{eq:Xi i N-k}, and where the family $\{g^j_0,\dotsc,g^j_{j-1}\}$ is defined as follows: $g^j_\ell$ is a polynomial of degree $\ell$, and these $j$ polynomials satisfy the (full-space) biorthogonality relations
    \begin{equation}
        \label{eq:full-space biorth}
        \sum_{x\in\mathbb{Z}}f^j_k(x)g^j_\ell(x)=\delta_{k,\ell}.
    \end{equation}
    Moreover, the formula \eqref{eq:full cond prob Fredholm pf initial 2} recovers the one derived for this model in Proposition 3.1 of \cite{borodin_large_2008}.
\end{cor}

The formula derived in \cite{borodin_large_2008} is for a more general model in full-space, namely PushASEP with jump rates which depend on time and on the particle; to compare their formula to \eqref{eq:full cond prob Fredholm pf initial 2} one needs to set all their $v_i$'s to $1$ and take $a(t)=1$ and $b(t)=0$.
Similar formulas were derived earlier for full-space TASEP in \cite{sasamoto_spatial_2005,borodin_fluctuation_2007}, under a formalism where the $\phi_i$'s are flipped with respect to our setting.

\begin{proof}[Proof of Corollary \ref{cor:full space TASEP}]
    If $y_N>1$, the probability on the left hand side of \eqref{eq:full cond prob Fredholm pf initial} can be computed from Theorem \ref{thm:cond prob Fredholm pf initial} with $\alpha=0$ and $M=N$, and is given by the Fredholm Pfaffian on the right hand side of \eqref{eq:cond joint dist TASEP initial} satisfying $\abs{X_t}=N$. 
    Now, since $M=N$, we note that the only (bi)orthogonality relation of the $\Upsilon$ polynomials is \eqref{eq:Upsilon-biorhtog rels initial}, while $K^\mathsf{A}$ vanishes.
    In particular, the $K_{22}$ of the kernel completely vanishes. 
    Now recall the following property, found in Section 8 of \cite{rains_correlation_2000}, whereby a Fredholm Pfaffian reduces to a Fredholm determinant:
    \[\pf(J-K)_{L^2(\mathcal{X})} = \det(I-K_{12})_{L^2(\mathcal{X})},\]
    whenever $K_{22}(x_1,x_2)=0$ for all $x_1,x_2\in\mathcal{X}$. This property implies that the right hand side of \eqref{eq:full cond prob Fredholm pf initial} reduces as
    \[\det\left(I-\bar\chi_a K \bar\chi_a\right)_{\ell^2(\{p_1,\dots,p_m\}\times \mathbb{N})},
    \quad\text{with}\quad K=K_{12}^0-K_{12}^\mathsf{B}.\]
    $K$ can be expressed explicitly as
    \begin{equation}
        \label{eq:full-space TASEP det kernel proof 1}
        K(i,x_1;j,x_2) = -\mathbbm{1}_{i<j}\phi_{[i,j)}(x_1,x_2) +\sum_{k=1}^N \Xi^{(i)}_{N-k}(x_1)\Upsilon^{(j)}_{N-k}(x_2).
    \end{equation}
    Now note that since $g\fconv\phi_{-(j,j+1]}(x)=g(x)-g(x-1)$ and since $\Upsilon_{N-k}$ is a polynomial of degree $N-k$, then $\Upsilon_{N-k}\fconv\phi_{-(j,j+1]}$ is a polynomial of degree $N-k-1$, and thus $\Upsilon^{(j)}_{N-k}$ is a polynomial of degree $j-k$ if $k\leq j$, while
    \begin{equation}
        \Upsilon^{(j)}_{N-k}=0\quad\text{if}\;k>j.
    \end{equation}
    In particular, the formula \eqref{eq:full-space TASEP det kernel proof 1} becomes \eqref{eq:full-space TASEP det kernel}.

    Now, assuming still that $y_N>1$, define $f^i_{i-k}$ as in the statement of the corollary and let
    \[g^j_{j-k}(x)=(-1)^k\hspace{0.1em}\Upsilon^{(j)}_{N-k}(x),\]
    so that \eqref{eq:full-space TASEP det kernel} can be expressed as \eqref{eq:full-space TASEP det kernel 2}.
    We know already that $g^j_{\ell}$ is a polynomial of degree $\ell$ for $\ell=0,\dotsc,j-1$, and we claim that the two families of functions satisfy \eqref{eq:full-space biorth}.
    In fact,
    \begin{align}
        \sum_{x\in\mathbb{Z}}f^j_{j-k}(x)g^{j}_{j-\ell}(x)&=\sum_{x,y\in\mathbb{Z}}\Xi^{(j)}_{N-k}(x)\Upsilon_{N-\ell}(y)\phi_{-(j,N]}(y,x)
        =\sum_{x\in\mathbb{Z}}\phi_{-(j,N]}\fconv\Xi^{(j)}_{N-k}(x)\Upsilon_{N-\ell}(x)\\
        &=\sum_{x\in\mathbb{Z}}\Xi_{N-k}(x)\Upsilon_{N-\ell}(x),
    \end{align}
    and the last sum can be restricted to $x\geq1$ because $\Xi_{N-k}(x)=0$ for $x\leq0$ thanks to the condition $y_k\geq y_N+N-k>N-k$, so \eqref{eq:Upsilon-biorhtog rels initial} ensures that it equals $\delta_{k,\ell}$.    

    Consider now a general choice of initial condition $y=(y_1,\dotsc,y_N)$.
    By translation invariance, if we fix $L>0$ and shift all the $y_i$'s and all the $a_i$'s by $L$, the probability on the left hand side of \eqref{eq:full cond prob Fredholm pf initial 2} does not change.
    Then letting $\bar y_i=y_i+L$ and $\bar a_i=a_i+L$, if $L$ is large enough so that $\bar y_N>1$ then previous case yields
    \begin{equation}\label{eq:full det 2}
         \mathbb{P}\left[\bigcap_{k=1}^m \left\{X_{t}(p_k) > a_k\right\}\right] = \det\!\left(I - \bar\chi_{\bar a} \bar K \bar\chi_{\bar a}\right)_{\ell^2(\{p_1,\dots,p_m\}\times \mathbb{N})},
    \end{equation}
    where $\bar K$ is defined as in \eqref{eq:full-space TASEP det kernel 2} but using $\bar y$ instead of $y$.
    But it is easy to see, from the definition of $\phi_{[i,j)}$ and $f^i_{i-k}$, and using the fact that the $g^i_{i-k}$'s are defined through the biorthogonalization problem \eqref{eq:full-space biorth}, that $\bar K(i,x_1+L;j,x_2+L)=K(i,x_1;j,x_2)$ (i.e., shifting the variables in the kernel is equivalent to shifting back the $y_i$'s).
    So performing this change of variables in the kernel inside the Fredholm determinant in \eqref{eq:full det 2} we get
    \begin{equation}
         \mathbb{P}\left[\bigcap_{k=1}^m \left\{X_{t}(p_k) > a_k\right\}\right] = \det\!\left(I - \bar\chi_{a} K \bar\chi_{a}\right)_{\ell^2(\{p_1,\dots,p_m\}\times\mathbb{Z}_{\geq-M})},
    \end{equation}
    Taking $L\to\infty$ yields \eqref{eq:full cond prob Fredholm pf initial 2}. Checking that this formula coincides with the one given in \cite{borodin_large_2008} in the case of full-space TASEP is straightforward.
    \end{proof}
    
\section{Details of the ASEP transition probability proof}
\label{se:ASEPproof}
In this section we present the details of the proof of Theorem \ref{thm:ASEP transition prob} which were omitted from Section \ref{sec:ASEP}, namely the proofs of the eigenvalue relation Lemma \ref{lm: F-eigenvector property} and the initial condition Lemma \ref{lm: orthogonality}.
\subsection{Preliminary results}
\label{sec:F-preliminary results}
Before proceeding to the proofs of the intermediate results of Lemmas \ref{lm: F-eigenvector property} and \ref{lm: orthogonality}, it is useful to present some initial properties of the function \eqref{eq:F initial coord function}.
\begin{prop}
	\label{prop:BCempty sum fac}
	For any fixed integer $N\geq0$ the following factorization holds on any alphabet $(w_1,\dots,w_N)$:
	\begin{equation}
		\label{eq:BCempty sum fac}
		\sum_{\sigma\in\mathcal{B}_N} \sigma\left(\prod_{1\leq i<j\leq N}
		\left[\frac{w_i-qw_j}{w_i-w_j}\frac{1-w_iw_j}{1- qw_iw_j}\right]\right)=\frac{1}{V_N},
	\end{equation}
    where the constant $V_N$ is given by \eqref{eq:V-normalization}.
\end{prop}
\begin{proof}
	Let us define the functions
	\[
	f(w_1,\dots,w_N):=\prod_{1\leq i<j\leq N}		\left[\frac{w_i-qw_j}{w_i-w_j}\frac{1-w_iw_j}{1- qw_iw_j}\right],\quad g(z_1,\dots,z_N) := \prod_{1\leq i<j\leq N}\left[\frac{z_i-qz_j}{z_i-z_j}\frac{1-qz_i^{-1}z_j^{-1}}{1-z_i^{-1}z_j^{-1}}\right],
	\]
	so that $f(w_1,\dots,w_N)$ is the summand of the symmetrization \eqref{eq:BCempty sum fac}. 
	Let us also consider the alphabet $(z_1,\dots,z_N)$ defined by $z_i=q^{1/2}w_i$, so that $f(w_1,\dots,w_N)=q^{-\binom{N}{2}}g(z_1,\dots,z_N)$.
	The standard identification for negative indices under the action of $\mathcal{B}_N$ is $w_{-i}=1/(q w_i)$ which may be written in terms of $z_i$ as $w_{-i}=q^{-1/2}z_i^{-1}$. And so, on the $z$-alphabet we make the identification $z_{-i}=z_i^{-1}$. It follows that the symmetrization \eqref{eq:BCempty sum fac} may be evaluated as
	\[\sum_{\sigma\in\mathcal{B}_N}\sigma(f(w_1,\dots,w_N)) = q^{-\binom{N}{2}}\sum_{\sigma\in \mathcal{B}_N}\sigma(g(z_1,\dots,z_N)),\]
	where the action of the signed permutations $\sigma\in\mathcal{B}_N$ on the $z$-alphabet now uses $z_{-k}=z_k^{-1}$  for negative indices.
	The right hand side symmetrization is equivalent to a known identity from \cite{venkateswaran_symmetric_2015}:
	\begin{equation}
		\sum_{\sigma\in \mathcal{B}_N}\sigma(g(z_1,\dots,z_N)) = \prod_{i=1}^N \frac{\left(1-q^i\right)\left(1+q^{i-1}\right)}{1-q},
	\end{equation}
	which follows from the fact that a $BC_N$-symmetric Hall--Littlewood polynomial indexed by the empty partition is equal to 1, i.e. $K_\emptyset\left(z_1^\pm,\dots,z_N^\pm\right)=1$, with boundary parameters $t_0=t_1=t_2=t_3=0$ and $a=-b=1$.
	Application of this identity yields the result \eqref{eq:BCempty sum fac}.
\end{proof}
Observe that the scattering has the following invariance under $w_j\mapsto w_{-j}=1/(qw_j)$:
\[\frac{w_i-qw_j}{w_i-w_j}\frac{1-w_iw_j}{1- qw_iw_j} \Bigg|_{w_j\mapsto w_{-j}} = \frac{w_i-qw_j}{w_i-w_j}\frac{1-w_iw_j}{1- qw_iw_j}.\]
This observation allows for the  symmetrization identity of Proposition \ref{prop:BCempty sum fac} has the following extension to a partial symmetrization:
\begin{equation}
    \label{eq:BCempty sum fac partial}
    \sum_{\sigma\in\widetilde{\mathcal{B}}_{N-M}} \sigma\left(\prod_{1\leq i<j\leq N}
    \left[\frac{w_i-qw_j}{w_i-w_j}\frac{1-w_iw_j}{1- qw_iw_j}\right]\right)=\frac{1}{V_{N-M}} \prod_{i=1}^M\prod_{j=i+1}^N\left[\frac{w_i-qw_j}{w_i-w_j}\frac{1-w_iw_j}{1- qw_iw_j}\right].
\end{equation}
for any fixed integer $0\leq M<N$ where $\widetilde{\mathcal{B}}_{N-M}\cong\mathcal{B}_{N-M}$ is the subgroup consisting of all signed permutations which act on $w$ by leaving $w_1,\dots,w_M$ fixed whilst permuting  $w_{M+1},\dots,w_N$. 
While Proposition \ref{prop:BCempty sum fac} implies that whenever indexed by an empty state $\mathcal{F}_\emptyset(w_1,\dots,w_N)=\alpha^N$, the following result allows for an alternative expression of the function \eqref{eq:F initial coord function} using a partial symmetrization.
\begin{prop}
	Let $N\geq M\geq 1$ be fixed integers and let $y=(y_1,\dots,y_M)$ be a half-space configuration. Then, the function \eqref{eq:F initial coord function} has the following alternative expression:
	\begin{multline}
		\label{eq:F initial coord function alternative}
		\mathcal{F}_y(w_1,\dots,w_N)=\alpha^{N-M} \sum_{T\subseteq\{1,\dots,N\}\atop\abs{T}=M} \sum_{\sigma\in\mathcal{B}_M} \prod_{1\leq i<j\leq M} \left[\frac{w_{T_{\sigma(i)}}-qw_{T_{\sigma(j)}}}{w_{T_{\sigma(i)}}-w_{T_{\sigma(j)}}}\frac{1-w_{T_{\sigma(i)}}w_{T_{\sigma(j)}}}{1- qw_{T_{\sigma(i)}}w_{T_{\sigma(j)}}}\right] \prod_{i=1}^M\varphi_{y_i}\left(
		w_{T_{\sigma(i)}}\right) \\ \times\prod_{i=1}^M\prod_{j=1\atop j\notin T}^{N} \left[\frac{w_{T_{\sigma(i)}}-qw_{j}}{w_{T_{\sigma(i)}}-w_{j}}\frac{1-w_{T_{\sigma(i)}}w_{j}}{1- qw_{T_{\sigma(i)}}w_{j}}\right].
	\end{multline}
\end{prop}
\begin{proof}
	Consider the symmetrization formula \eqref{eq:F initial coord function} and let us decompose the sum over $\mathcal{B}_N$ as
	\begin{multline}
		\label{eq:F alternative proof 1}
		\mathcal{F}_y(w_1,\dots,w_N)=V_{N-M}\alpha^{N-M} \\
		\times\sum_{T\subseteq\{1,\dots,N\}\atop\abs{T}=M} \sum_{\sigma\in\mathcal{B}_M}\sum_{\rho\in\mathcal{B}_{N-M}} \prod_{1\leq i<j\leq M} \left[\frac{w_{T_{\sigma(i)}}-qw_{T_{\sigma(j)}}}{w_{T_{\sigma(i)}}-w_{T_{\sigma(j)}}}\frac{1-w_{T_{\sigma(i)}}w_{T_{\sigma(j)}}}{1- qw_{T_{\sigma(i)}}w_{T_{\sigma(j)}}}\right] \prod_{i=1}^M\varphi_{y_i}\left(w_{T_{\sigma(i)}}\right )\\
		\times\prod_{i=1}^M\prod_{j=1}^{N-M} \left[\frac{w_{T_{\sigma(i)}}-qw_{T^\mathsf{C}_{\rho(j)}}}{w_{T_{\sigma(i)}}-w_{T^\mathsf{C}_{\rho(j)}}}\frac{1-w_{T_{\sigma(i)}}w_{T^\mathsf{C}_{\rho(j)}}}{1- qw_{T_{\sigma(i)}}w_{T^\mathsf{C}_{\rho(j)}}}\right] \prod_{1\leq i<j\leq N-M}\left[\frac{w_{T^\mathsf{C}_{\rho(i)}}-qw_{T^\mathsf{C}_{\rho(j)}}}{w_{T^\mathsf{C}_{\rho(i)}}-w_{T^\mathsf{C}_{\rho(j)}}}\frac{1-w_{T^\mathsf{C}_{\rho(i)}}w_{T^\mathsf{C}_{\rho(j)}}}{1- qw_{T^\mathsf{C}_{\rho(i)}}w_{T^\mathsf{C}_{\rho(j)}}}\right].
	\end{multline}
	We will proceed by explicitly evaluating the sum over $\rho\in\mathcal{B}_{N-M}$. Now, observe that, for fixed $T$, the first factor on the bottom line of \eqref{eq:F alternative proof 1} is invariant under the choice of $\rho\in\mathcal{B}_{N-M}$. Consequently, we may evaluate the sum over $\rho$ using Proposition \ref{prop:BCempty sum fac} as
	\[\sum_{\rho\in\mathcal{B}_{N-M}}\prod_{1\leq i<j\leq N-M}\left[\frac{w_{T^\mathsf{C}_{\rho(i)}}-qw_{T^\mathsf{C}_{\rho(j)}}}{w_{T^\mathsf{C}_{\rho(i)}}-w_{T^\mathsf{C}_{\rho(j)}}}\frac{1-w_{T^\mathsf{C}_{\rho(i)}}w_{T^\mathsf{C}_{\rho(j)}}}{1- qw_{T^\mathsf{C}_{\rho(i)}}w_{T^\mathsf{C}_{\rho(j)}}}\right] = \frac1{V_{N-M}},\]
	which yields the desired result \eqref{eq:F initial coord function alternative} once substituted into \eqref{eq:F alternative proof 1}.
\end{proof}
\begin{prop}
	\label{prop:F-recurrence}
	Let $N\geq M\geq 1$ be fixed integers and let $w=(w_1,\dots,w_N)$ be a fixed alphabet and let $y_1,\dots,y_M\in\mathbb{Z}$ be fixed so that $y=(y_1,\dots,y_M)$  is a collection of integers\footnote{Typically we consider the case where $y_1\geq \cdots\geq y_M\geq 1$ so that $y$ is a half-space configuration. However we will need the more general statement in the proof of Lemma \ref{lm: F-eigenvector property} later on.}. The function \eqref{eq:F initial coord function} satisfies the following recursion relation for all $k\in\{1,\dots,N\}$:
	\begin{equation}
		\mathcal{F}_y(w_1,\dots,w_N)\Big|_{w_k=0} = \alpha \mathcal{F}_y(w_1,\dots,w_{k-1},w_{k+1},\dots,w_N).
	\end{equation}
\end{prop}
\begin{proof}
	Observe that, whenever $w_k=0$, the following property holds for any $y\in\mathbb{Z}$ 
	\[\varphi_y(w_k) \Big|_{w_k=0} = \varphi_y\left(\frac{1}{qw_k}\right)\Bigg|_{w_k=0} = 0.\]
	By setting $w_k=0$ within the the symmetrization formula \eqref{eq:F initial coord function} observe that terms corresponding to $\sigma(\ell) = \pm k$ for any $1\leq \ell \leq M$ vanish. Now observe the following property for any $\sigma\in\mathcal{B}_N$ which satisfies $\sigma(\ell) = \pm k$:
	\begin{multline*}
		\left.\prod_{1\leq i\leq j\leq N} \left[\frac{w_{\sigma(i)}-qw_{\sigma(j)}}{w_{\sigma(i)}-w_{\sigma(j)}}\frac{1-w_{\sigma(i)}w_{\sigma(j)}}{1- qw_{\sigma(i)}w_{\sigma(j)}}\right]\right|_{w_k=0} \\ = \prod_{1\leq i\leq j\leq N \atop i,j\neq\ell} \left[\frac{w_{\sigma(i)}-qw_{\sigma(j)}}{w_{\sigma(i)}-w_{\sigma(j)}}\frac{1-w_{\sigma(i)}w_{\sigma(j)}}{1- qw_{\sigma(i)}w_{\sigma(j)}}\right]\times\begin{cases}
			q^{N-\ell} & \text{if }\sigma(\ell)=+k\\
			q^{-(N-\ell)} & \text{if }\sigma(\ell)=-k
		\end{cases}.
	\end{multline*}
	The application of this specialization to the symmetrization formula \eqref{eq:F initial coord function} yields:
	\begin{multline}
		\label{eq:F-recurrence proof 1}
		\mathcal{F}_y(w_1,\dots,w_N)\Big|_{w_k=0} = V_{N-M}\alpha^{N-M} \sum_{\ell=M+1}^N q^{N-\ell}\sum_{\sigma\in\mathcal{B}_N\atop\sigma(\ell)=+k}\sigma\left(\prod_{1\leq i\leq j\leq N\atop i,j\neq\ell}\left[\frac{w_i-qw_j}{w_i-w_j}\frac{1-w_iw_j}{1- qw_iw_j}\right] 	\prod_{i=1}^M \varphi_{y_i}(w_i)\right)\\
		+V_{N-M}\alpha^{N-M} \sum_{\ell=M+1}^N q^{-(N-\ell)}\sum_{\sigma\in\mathcal{B}_N\atop\sigma(\ell)=-k}\sigma\left(\prod_{1\leq i\leq j\leq N\atop i,j\neq\ell}\left[\frac{w_i-qw_j}{w_i-w_j}\frac{1-w_iw_j}{1- qw_iw_j}\right] 	\prod_{i=1}^M \varphi_{y_i}(w_i)\right).
	\end{multline}
	Now, for each fixed value of $\ell$, both restricted sums over signed permutations can be identified as being proportional to $\mathcal{F}_y(w_1\dots,w_{k-1},w_{k+1},\dots,w_N)$. We emphasize that this identification is independent of $\ell$ so that we may re-write \eqref{eq:F-recurrence proof 1} by factoring out the $\mathcal{F}_y$ function as
	\[\mathcal{F}_y(w_1,\dots,w_N)\Big|_{w_k=0} = \frac{V_{N-M}}{V_{N-M-1}}\sum_{\ell=M+1}^N\left(q^{N-\ell}+q^{-(N-\ell)}\right)\cdot\alpha\mathcal{F}_y(w_1\dots,w_{k-1},w_{k+1},\dots,w_N).\]
	Finally, Proposition \eqref{prop:F-recurrence} follows by noting the following property of the normalization factor \eqref{eq:V-normalization}
	\[\frac{V_{m-1}}{V_{m}} = \sum_{j=0}^{m-1}\left(q^j+q^{-j}\right),\]
	which holds for any integer $m>0$.
\end{proof}
\begin{prop}
	\label{prop:F function meromorphic}
	Let $N\geq M\geq 0$ be fixed integers and let $w=(w_1,\dots,w_N)$ be a fixed alphabet and let $y_1,\dots,y_M\in\mathbb{Z}$ be fixed so that $y=(y_1,\dots,y_M)$ is a collection of integers. Then the function $\mathcal{F}_y(w)$, given by the symmetrization formula \eqref{eq:F initial coord function}, is a meromorphic function in each variable $w_1,\dots,w_N$ whose singularities are all located at the points $w_i=1,q^{-1}$ for each $1\leq i\leq N$.
\end{prop}
\begin{proof}
	Let us consider the function $\mathcal{F}_y(w_1,\dots,w_N)$ as a function of the variable $w_k$ for some fixed integer $1\leq k\leq N$. Excluding poles at $w_k=1,q^{-1}$, by observing the individual terms of symmetrization formula \eqref{eq:F initial coord function}, the other possible singularities are at the points $w_k=w_j,1/(qw_j)$ for integers $j\neq k$ and also at $w_k=\pm q^{-1/2}$. We will proceed by showing that these singularities are removable, so that the function $\mathcal{F}_y(w)$ may be defined to be analytic at those points. Firstly, the simple poles at $w_k=w_j,1/(qw_j)$ are removable by virtue of the symmetrization over $\mathcal{B}_N$ and is therefore symmetric under the action of any signed permutation.
	
	Let us now consider the apparent singularities at $w_k=\pm q^{-1/2}$, which are due to factors $\left(1-qw_k^2\right)$ which appear in the denominator of $\varphi_{y_\ell}(w_k)$ and $\varphi_{y_\ell}\left(\frac{1}{qw_k}\right)$ for some integer $1\leq\ell\leq M$. Let us consider the terms in \eqref{eq:F initial coord function} associated with signed permutations $\sigma,\sigma'\in\mathcal{B}_N$ satisfying $\sigma(\ell)=+k$ and $\sigma'(\ell)=-k$ but otherwise $\sigma(i)=\sigma'(i)$ for all other integers $i\neq\ell$. It is sufficient to consider the $w_k$-dependent factors of the sum of these two terms:
	\begin{multline}
		\label{eq:F function meromorphic proof 1}
		\prod_{i=1}^{\ell-1}\left[\frac{w_k-q w_{\sigma(i)}}{w_k-w_{\sigma(i)}}\frac{1-w_kw_{\sigma(i)}}{1-qw_kw_{\sigma(i)}}\right]\prod_{i=\ell+1}^N\left[\frac{w_{\sigma(i)}-q w_k}{w_{\sigma(i)}-w_k}\frac{1-w_kw_{\sigma(i)}}{1-qw_kw_{\sigma(i)}}\right]\varphi_{y_\ell}(w_k) \\
		+\prod_{i=1}^{\ell-1}\left[\frac{qw_k-w_{\sigma(i)}}{q(w_k-w_{\sigma(i)})}\frac{1-w_kw_{\sigma(i)}}{1-qw_kw_{\sigma(i)}}\right]\prod_{i=\ell+1}^N\left[\frac{w_{\sigma(i)}-q w_k}{w_{\sigma(i)}-w_k}\frac{1-w_kw_{\sigma(i)}}{1-qw_kw_{\sigma(i)}}\right]\varphi_{y_\ell}\left(\frac{1}{qw_k}\right),
	\end{multline}
	where overall shared factors are neglected. Now, observe that each of these terms both contain the factor $\frac{w_\ell}{1-qw_\ell^2}$ which changes sign under $w_\ell\mapsto 1/(qw_\ell)$. This means we may write \eqref{eq:F function meromorphic proof 1} as 
	\begin{equation}
		\label{eq:F function meromorphic proof 2}
		\frac{w_k}{1-qw_k^2}\left(h(w_k)-h\left(\frac{1}{qw_k}\right)\right)\prod_{i=\ell+1}^N\left[\frac{w_{\sigma(i)}-q w_k}{w_{\sigma(i)}-w_k}\frac{1-w_kw_{\sigma(i)}}{1-qw_kw_{\sigma(i)}}\right],
	\end{equation}
	where the following function is analytic at $w_k=\pm q^{-1/2}$:
	\[h(w_k) = \frac{(1-q)(1-q-\alpha +\alpha w_k)}{(1-w_k)} \left(\frac{1-q w_k}{1-w_k}\right)^{y_\ell}\prod_{i=1}^{\ell-1}\left[\frac{w_k-qw_{\sigma(i)}}{w_k-w_{\sigma(i)}}\frac{1-w_k w_{\sigma(i)}}{1- qw_k w_{\sigma(i)}}\right].\]
	From here, by showing that it has vanishing residue, we observe that the singularities in expression \eqref{eq:F function meromorphic proof 2} at the points $w_k=\pm q^{-1/2}$ are removable.
\end{proof}
The following evaluation of the function \eqref{eq:F initial coord function} is particularly useful.
\begin{lm}
	\label{lm:F-1 evaluation}
	Let $N>0$ be a fixed integer and let $w=(w_1,\dots,w_N)$ be a fixed alphabet. The function \eqref{eq:F initial coord function} indexed by the half-space configuration $y=(1)$ is explicitly given by:
	\begin{equation}
		\label{eq:F-1 evaluation}
		\mathcal{F}_{(1)}(w_1,\dots,w_N) = \alpha^{N-1} \sum_{k=1}^N \frac{(1-q)^2w_k}{(1-w_k)(1-qw_k)}.
	\end{equation}
\end{lm}
\begin{proof}
	Let us define the function $P(w_1,\dots,w_N)$ as the right hand side of \eqref{eq:F-1 evaluation}. Then observe that 
	\begin{equation}
		\label{eq:F-1 evaluation proof 1}
		P(w_1,\dots,w_N)\cdot\prod_{i=1}^N\left[(1-w_i)(1-qw_i)\right]
	\end{equation}
	is clearly a symmetric polynomial of maximal degree 2 in each variable $w_1,\dots,w_N$. As such, the polynomial \eqref{eq:F-1 evaluation proof 1} is completely determined by three independent specializations. The following are easily verifiable for each $k\in \{1,\dots,N\}$:
	\begin{enumerate}[label=\textbf{(\roman*)},wide=0pt,itemsep=5pt]
		\item The residue at $w_k=1$:
		\[\res_{w_k=1} P(w_1,\dots,w_N) = \alpha^{N-1}(q-1).\]
		\item The residue at $w_k=q^{-1}$:
		\[\res_{w_k=q^{-1}} P(w_1,\dots,w_N) = \alpha^{N-1}\left(q^{-1}-1\right).\]
		\item The recurrence at $w_k=0$:
		\[P(w_1,\dots,w_N)\Big|_{w_k=0} = \alpha P(w_1,\dots,w_{k-1},w_{k+1},\dots,w_N).\]
	\end{enumerate}
	Lemma~\ref{lm:F-1 evaluation} follows by showing that each of these properties holds for the symmetrization formula \eqref{eq:F initial coord function} for the function $\mathcal{F}_{(1)}(w)$. But first we must show that 
	\begin{equation}
		\label{eq:F_1 polynomial}
		\mathcal{F}_{(1)}(w_1,\dots,w_N)\cdot\prod_{i=1}^N\left[(1-w_i)(1-qw_i)\right]
	\end{equation}
	is a polynomial in each $w$-variable of maximal degree 2. By virtue of Proposition \ref{prop:F function meromorphic}, the function \eqref{eq:F_1 polynomial} is a meromorphic function whose only possible poles occur at $w_k=1,q^{-1}$ for each integer $1\leq k \leq N$. By observing individual terms in its symmetrization formula \eqref{eq:F initial coord function}, we observe that $F_{(1)}(w)$ has at most simple poles at $w_k=1,q^{-1}$, so that \eqref{eq:F_1 polynomial} is indeed a polynomial in $w_k$. And so, we may determine its degree by investigating the behavior as $w_k\to\infty$ term-wise for each individual signed permutation $\sigma\in\mathcal{B}_N$. It is easily verifiable that for each $\sigma\in \mathcal{B}_N$:
	\[\prod_{i=1}^N\left[(1-w_i)(1-qw_i)\right]\cdot\sigma\left(\prod_{1\leq i<j\leq N}
	\left[\frac{w_i-qw_j}{w_i-w_j}\frac{1-w_iw_j}{1- qw_iw_j}\right] \varphi_{1}(w_1)\right) = \mathcal{O}\left(w_k^2\right)\]
	as $w_k\to\infty$ for each integer $1\leq k\leq N$. So we conclude that \eqref{eq:F_1 polynomial} has maximal degree 2.  
	
	Let us now show that $\mathcal{F}_{(1)}(w_1,\dots,w_N)$ satisfies the required properties \textbf{(i)--(iii)}. In what follows, let $k\in \{1,\dots,N\}$ be fixed positive integer.
	\begin{enumerate}[label=\textbf{(\roman*)},wide=0pt,itemsep=5pt]
		\item Observe that, in the symmetrization formula \eqref{eq:F initial coord function}, the only terms with singularities at $w_k=1$ are those where $\varphi_y(w_k)$ appears. In our case, these terms will correspond with signed permutations satisfying $\sigma(1)=+k$ where there is a simple pole at $w_k=1$. Then observe that the scattering factor has the following specialization:
		\[\left.\frac{w_k-qw_j}{w_k-w_j}\frac{1-w_kw_j}{1-qw_kw_j}\right|_{w_k=1}=1.\]
		This means that for individual signed permutations satisfying $\sigma(1)=+k$, we may evaluate the residue of the simple pole at $w_k=1$ by evaluating
		\[\res_{w_k=1}\varphi_1(w_k) = q-1.\]
		Meanwhile, on the symmetrization formula we have
		\begin{align*}
			\res_{w_k=1} \mathcal{F}_{(1)}(w_1,\dots,w_N) & = \alpha^{N-1}V_{N-1} \sum_{\sigma\in\mathcal{B}_N\atop\sigma(1)=k}\prod_{2\leq i\leq j\leq N} \sigma\left(\left[\frac{w_i-qw_j}{w_i-w_j}\frac{1-w_iw_j}{1- qw_iw_j}\right]\right) \res_{w_k=1}\varphi_1(w_k) \\
			& = \alpha^{N-1}(q-1),
		\end{align*}
		where we have used Proposition \ref{prop:BCempty sum fac} to simplify the last line. This is the desired result.
		\item It follows by a similar argument that
		\[\res_{w_k=q^{-1}}\mathcal{F}_{(1)}(w_1,\dots,w_N) = \alpha^{N-1}\left(q^{-1}-1\right),\]
		where the residue isolates signed permutations satisfying $\sigma(1)= -k$.
		\item This property follows directly as a consequence of Proposition \ref{prop:F-recurrence}.\qedhere
		
	\end{enumerate}
\end{proof}

\subsection{Proof of eigenvector property}
Let us recall the result of Lemma \ref{lm: F-eigenvector property} where, for a fixed half-space configuration with $M$ particles $y$, the following holds on any alphabet $w=(w_1,\dots,w_N)$:
\begin{equation}
    \label{eq: F-eigenvector property repeat}
    \sum_{y'\in\mathbb{W}} \bra{y}\mathscr{L}\ket{y'} \mathcal{F}_{y'}(w_1,\dots,w_N) = \left(-\alpha + \sum_{i=1}^N\frac{(1-q)^2 w_i}{(1-w_i)(1-qw_i)}\right)\mathcal{F}_y(w_1,\dots,w_N).
\end{equation}
So that the function $\mathcal{F}_y(w)$ from \eqref{eq:F initial coord function} may be regarded as an eigenvector of the transition matrix of the half-space ASEP with $\gamma=0$.
\label{sec:F-eigenvector proof}
\begin{proof}[Proof of Lemma \ref{lm: F-eigenvector property}]
	Let us first examine the left hand side of \eqref{eq: F-eigenvector property repeat} more explicitly. Let us denote $\mathcal{F}_{y}(w_1,\dots,w_N)=f(y)$ for some fixed half-space configuration $y=(y_1,\dots,y_M)$. The left hand side of \eqref{eq: F-eigenvector property repeat} can be written as the following \emph{free-evolution equation}:
	\begin{multline}
		\label{eq: F-eigenvector property proof 2}
		\sum_{y'\in\mathbb{W}} \bra{y}\mathscr{L}\ket{y'} f(y') = q \sum_{i=1}^M f(y_1,\dots,y_i-1,\dots,y_M) +  \sum_{i=1}^Mf(y_1,\dots,y_i+1,\dots,y_M) \\
		+\alpha f(y_1,\dots,y_M,1) - (M(1+q)+\alpha)f(y_1,\dots,y_M),
	\end{multline}
	whose right hand side is subject to the following scattering conditions:
	\begin{enumerate}[label=\textbf{(\roman*)},wide=0pt,itemsep=5pt]
		\item \emph{Bulk scattering:} for each $1< i\leq M$, whenever $y_{i-1}=y_i+1$:
		\begin{multline}
			\label{eq: F-eigenvector property proof 3 bulk scatt}
			q f(y_1,\dots,y_i,y_i,\dots,y_M) + f(y_1,\dots,y_i+1,y_i+1,\dots,y_M) \\ = (1+q)f(y_1,\dots,y_i+1,y_i,\dots,y_M).
		\end{multline}
		\item \emph{Boundary scattering:} whenever $y_M = 1$:
		\begin{equation}
			\label{eq: F-eigenvector property proof 4 bound scatt}
			qf(y_1,\dots,y_{M-1},0) +\alpha f(y_1,\dots,y_{M-1},1,1) = (\alpha+q) f(y_1,\dots,y_{M-1},1).
		\end{equation}
	\end{enumerate}
	Observe that both scattering conditions are relations between functions indexed by both \emph{physical} and \emph{nonphysical} half-space configurations. For a fixed $y\in\mathbb{W}$, we may apply each of the scattering conditions to the right hand side of \eqref{eq: F-eigenvector property proof 2} to remove any nonphysical configurations. In this way, the bulk scattering \eqref{eq: F-eigenvector property proof 3 bulk scatt} enforces the particle exclusion condition and the boundary scattering \eqref{eq: F-eigenvector property proof 4 bound scatt} enforces the boundary dynamics at $\gamma=0$. One can readily check that this is consistent with the multiplication by the transition matrix on the left hand side, thus yielding the time-evolution of the half-space ASEP.
	
	We first prove that the function \eqref{eq:F initial coord function} satisfies the free-evolution equation \eqref{eq: F-eigenvector property proof 2} before showing both scattering conditions \eqref{eq: F-eigenvector property proof 3 bulk scatt}, \eqref{eq: F-eigenvector property proof 4 bound scatt}.
	
	\begin{enumerate}[label=\textbf{(\roman*)},wide=0pt,itemsep=5pt] 
		\item[\textbf{(a)}] \emph{Free-evolution equation:} We begin by showing that $f(y) = \mathcal{F}_y(w)$ satisfies the following:
		\begin{multline}
			\label{eq:free-evolution}
			\left(-\alpha + \sum_{i=1}^N\frac{(1-q)^2 w_i}{(1-w_i)(1-qw_i)}\right)f(y_1,\dots,y_M) \\
			=\alpha f(y_1,\dots,y_M,1) + q \sum_{i=1}^M f(y_1,\dots,y_i-1,\dots,y_M) +  \sum_{i=1}^Mf(y_1,\dots,y_i+1,\dots,y_M) \\ - (M(1+q)+\alpha)f(y_1,\dots,y_M).
		\end{multline}
		Firstly, we may immediately eliminate terms $-\alpha f(y)$ from each side of \eqref{eq:free-evolution}. Then, on what remains, consider the term corresponding to $\sigma\in\mathcal{B}_N$ within the symmetrization formula \eqref{eq:F initial coord function} individually within the remaining terms of \eqref{eq:free-evolution}. The following property will be useful:
		\[q \varphi_{y_i-1}(w_i) + \varphi_{y_i+1}(w_i)-(1+q)\varphi_{y_i}(w_i) = \frac{(1-q)^2w_i}{(1-w_i)(1-qw_i)}\varphi_{y_i}(w_i).\]
		Application of this property within the the right hand side of \eqref{eq:free-evolution} yields the following
		\begin{multline*}
			\sum_{k=1}^N\frac{(1-q)^2w_k}{(1-w_k)(1-qw_k)}\sum_{\sigma\in\mathcal{B}_N}\sigma\left(\prod_{1\leq i<j\leq N} \left[\frac{w_i-qw_j}{w_i-w_j}\frac{1-w_iw_j}{1- qw_iw_j}\right] 	\prod_{i=1}^M \varphi_{y_i}(w_i) \right) \\
			= \sum_{\sigma\in\mathcal{B}_N}\sigma\left( \prod_{1\leq i<j\leq N} \left[\frac{w_i-qw_j}{w_i-w_j}\frac{1-w_iw_j}{1- qw_iw_j}\right] 	\prod_{i=1}^M \varphi_{y_i}(w_i)\left(\frac{V_{N-M-1}}{V_{N-M}}\varphi_1(w_{M+1})+\sum_{k=1}^M\frac{(1-q)^2w_k}{(1-w_k)(1-qw_k)}\right)\right).
		\end{multline*}
		By re-arrangement, we can see that \eqref{eq:free-evolution} holds if the following vanishes under symmetrization:
		\begin{equation}
			\label{eq:free-evolution proof 1}
			\prod_{1\leq i<j\leq N} \left[\frac{w_i-qw_j}{w_i-w_j}\frac{1-w_iw_j}{1- qw_iw_j}\right] 	\prod_{i=1}^M \varphi_{y_i}(w_i)\left(\frac{V_{N-M-1}}{V_{N-M}}\varphi_1(w_{M+1})-\sum_{k=M+1}^N\frac{(1-q)^2w_k}{(1-w_k)(1-qw_k)}\right).
		\end{equation}
		Observe, as a consequence of Lemma \ref{lm:F-1 evaluation}, that \eqref{eq:free-evolution proof 1} vanishes under symmetrization of the variables $w_{M+1},\dots,w_{N}$ whilst leaving $w_1,\dots,w_M$ fixed and the result follows.
		
		\item[(\textbf{bi})] \emph{Bulk scattering:} Consider, for some fixed alphabet $w=(w_1,\dots,w_N)$ and configuration $y=(y_1,\dots,y_M)$ with $y_{k-1}=y_k+1$ for some $1< k\leq M$, the following
		\begin{multline}
			\label{eq:bulk scattering proof 1}
			q\mathcal{F}_{(y_1,\dots,y_k,y_k,\dots,y_M)}(w)+\mathcal{F}_{(y_1,\dots,y_k+1,y_k+1,\dots,y_M)}(w) - (1+q)\mathcal{F}_{(y_1,\dots,y_k+1,y_k,\dots,y_M)}(w) \\
			= V_{N-M} \alpha^{N-M} \sum_{\sigma\in\mathcal{B}_N} \sigma\Bigg(\prod_{1\leq i<j\leq N\atop (i,j)\neq(k-1,k)}
			\left[\frac{w_i-qw_j}{w_i-w_j}\frac{1-w_iw_j}{1- qw_iw_j}\right] \prod_{i=1}^M \varphi_{y_i}(w_i) \\
			\times\frac{w_{k-1}-qw_k}{w_{k-1}-w_k}\left[q+\frac{1-q w_{k-1}}{1-w_{k-1}}\frac{1-q w_k}{1-w_k}-(1+q)\frac{1-q w_{k-1}}{1-w_{k-1}}\right]\Bigg).
		\end{multline}
		The top-line of the summand on the right hand side of \eqref{eq:bulk scattering proof 1} is symmetric under the interchange of $w_{k-1}$ and $w_k$. Whereas the bottom line may be simplified as
		\[\frac{1-q}{(1-w_{k-1})(1-w_k)} \frac{(w_{k-1}-qw_k)(w_k-qw_{k-1})}{w_{k-1}-w_k},\]
		which is skew-symmetric under the same interchange. Therefore, the sum over over all permutations $\sigma\in S_N$ on the right hand side of \eqref{eq:bulk scattering proof 1} vanishes. The bulk scattering condition \eqref{eq: F-eigenvector property proof 3 bulk scatt} follows as a result.
		\item[(\textbf{bii})] \emph{Boundary scattering:}
		Let us, for a fixed alphabet $w=(w_1,\dots,w_N)$, consider the following
		\begin{equation}
			\label{eq:boundary scattering proof 1}
			q \mathcal{F}_{(y_1,\dots,y_{M-1},0)}(w)+\alpha\mathcal{F}_{(y_1,\dots,y_{M-1},1,1)}(w) - (q+\alpha)\mathcal{F}_{(y_1,\dots,y_{M-1},1)}(w),
		\end{equation}
		which will ultimately be shown to vanish. Let us first consider \eqref{eq:boundary scattering proof 1} in the case $N=M$ separately. It is convenient to treat this case separately since the function $\mathcal{F}_y(w)$ from \eqref{eq:F initial coord function} vanishes whenever the half-space configuration $y$ has more parts than its alphabet $w$. This means that the middle term in \eqref{eq:boundary scattering proof 1} does not appear whenever $N=M$. In this case, the symmetrization formula \eqref{eq:F initial coord function} can be compared term-wise, where the following function which appears in the summand
		\begin{equation}
			\label{eq:boundary scattering proof 2}
			q\varphi_0(w_N) - (q+\alpha)\varphi_1(w_N),
		\end{equation}
		while all other factors in the summand are invariant under the action of $w_N\mapsto1/(qw_N)$. Now note the property whereby
		\[\varphi_1(w)+\varphi_1\left(\frac{1}{q w}\right) = \frac{q}{q+\alpha}\left(\varphi_0(w)+\varphi_0\left(\frac{1}{q w}\right)\right) = \frac{(1-q)^2w}{(1-w)(1-qw)}.\]
		This property may be used to show that \eqref{eq:boundary scattering proof 2} changes sign under $w_N\mapsto1/(qw_N)$. And so, whenever $N=M$, \eqref{eq:boundary scattering proof 1} vanishes.
		
		Now let us consider \eqref{eq:boundary scattering proof 1} for $M< N$. We will ultimately show that the following function
		\begin{multline}
			\label{eq:boundary scattering proof 3}
			\qquad \quad f(w_M,\dots,w_N)=\prod_{M\leq i<j\leq N}
			\left[\frac{w_i-qw_j}{w_i-w_j}\frac{1-w_iw_j}{1- qw_iw_j}\right]\\
			\times V_{N-M}\left(q\varphi_0(w_M)+\frac{V_{N-M-1}}{V_{N-M}}\varphi_1(w_M)\varphi_1(w_{M+1})-(q+\alpha)\varphi_1(w_M)\right),
		\end{multline}
		vanishes when summed over signed permutations of the sub-alphabet $(w_M,\dots,w_N)$. The function $f(w_M,\dots,w_N)$ is part of the summand of the combined symmetrization formulas for the functions in \eqref{eq:boundary scattering proof 1}. The fact that the vanishing partial symmetrization enforces the vanishing of the fully-symmetrized functions in \eqref{eq:boundary scattering proof 1} follows in along the same lines as the partial symmetrization formula \eqref{eq:BCempty sum fac partial}. 
		The vanishing of the function \eqref{eq:boundary scattering proof 3} under symmetrization is equivalent to the following condition in the function \eqref{eq:F initial coord function} for any $N\geq 2$:
		\begin{equation}
			\label{eq:boundary scattering proof 4}
			\frac{q}{q+\alpha} \mathcal{F}_{(0)}(w_1,\dots,w_N) + \frac{\alpha}{q+\alpha} \mathcal{F}_{(1,1)}(w_1,\dots,w_N) = \mathcal{F}_{(1)}(w_1,\dots,w_N).
		\end{equation}
		Let us then denote the left hand side of \eqref{eq:boundary scattering proof 4} as the function $\mathcal{G}(w_1,\dots,w_N)$, which is is clearly a symmetric function. Moreover, we claim that the function defined by
		\[\mathcal{G}(w_1,\dots,w_N)\cdot\prod_{i=1}^N\left[(1-w_i)(1-qw_i)\right]\]
		is a symmetric polynomial of maximal degree 2 in each or the variables $w_1,\dots,w_N$. This claim follows from the observation that the function \eqref{eq:F_1 polynomial} is also a polynomial of maximal degree 2 when the configuration $(1)$ is replaced by either $(0)$ or $(1,1)$ by following an analogous argument which we do not explicitly present. 
		
		Before we proceed further, let us come up with an alternative expression for the function $\mathcal{G}(w_1,\dots,w_N)$. 
		First, by invoking the alternative expression \eqref{eq:F initial coord function alternative}, the function $\mathcal{F}_{(0)}(w_1,\dots,w_N)$ has the following expression:
		\begin{equation}
			\label{eq:boundary scattering proof 5}
			\mathcal{F}_{(0)}(w_1,\dots,w_N) = \alpha^{N-1}\sum_{1\leq\abs{i}\leq N} \frac{1-w_i}{1-qw_i}\varphi_1(w_i) \prod_{j=1\atop j\neq\abs{i}}^N\left[\frac{w_i-qw_j}{w_i-w_j}\frac{1-w_iw_j}{1- qw_iw_j}\right], 
		\end{equation}
		where the sum over $i$ is over $\{-N,\dots,1\}\cup\{1,\dots,N\}$ where $w_{-k}=1/(qw_k)$ for each $k>0$. Similarly, we may apply the expression \eqref{eq:F initial coord function alternative} to express
		\begin{multline}
			\label{eq:boundary scattering proof 6}
			\mathcal{F}_{(1,1)}(w_1,\dots,w_N) = \sum_{1\leq\abs{i}\leq N} \varphi_1(w_i)\prod_{j=1\atop j\neq\abs{i}}^N\left[\frac{w_i-qw_j}{w_i-w_j}\frac{1-w_iw_j}{1- qw_iw_j}\right]  \\ \times\alpha^{N-2}\sum_{1\leq\abs{k}\leq N\atop\abs{k}\neq\abs{i}}\varphi_1(w_k)\prod_{j=1\atop j\neq\abs{i},j\neq \abs{k}}^N\left[\frac{w_i-qw_j}{w_i-w_j}\frac{1-w_iw_j}{1- qw_iw_j}\right], 
		\end{multline}
		where we may recognize the bottom line of \eqref{eq:boundary scattering proof 6} as $\mathcal{F}_{(1)}(w_1,\dots,w_{\abs{i}-1},w_{\abs{i}+1},\dots,w_N)$. Then using Lemma \ref{lm:F-1 evaluation}, we may write \eqref{eq:boundary scattering proof 6} as
		\begin{equation}
			\label{eq:boundary scattering proof 7}
			\mathcal{F}_{(1,1)}(w_1,\dots,w_N) = \alpha^{N-2}\sum_{1\leq\abs{i}\leq N} \varphi_1(w_i)\left(\sum_{j=1\atop j\neq \abs{i}}^N\frac{(1-q)^2w_j}{(1-w_j)(1-qw_j)}\right)\prod_{j=1\atop j\neq\abs{i}}^N\left[\frac{w_i-qw_j}{w_i-w_j}\frac{1-w_iw_j}{1- qw_iw_j}\right].
		\end{equation} 
		Now, we claim that the function $\mathcal{G}(w)=(q\mathcal{F}_{(0)}(w)+\alpha\mathcal{F}_{(1,1)}(w))/(q+\alpha)$ can be written using expressions \eqref{eq:boundary scattering proof 5} and \eqref{eq:boundary scattering proof 7} to yield:
		\begin{multline}
			\label{eq:G func sum 1}
			\qquad \quad\mathcal{G}(w_1,\dots,w_N) \\ = \frac{\alpha^{N-1}}{q+\alpha}\sum_{1\leq\abs{i}\leq N} \varphi_1(w_i) \left(\frac{q-w_i}{1-w_i}+\sum_{j=1}^N\frac{(1-q)^2w_j}{(1-w_j)(1-qw_j)}\right)
			\prod_{j=1\atop j\neq \abs{i}}^N\left[\frac{w_i-qw_j}{w_i-w_j}\frac{1-w_iw_j}{1- qw_iw_j}\right],
		\end{multline}
		where, again, the sum over $i$ is over $\{-N,\dots,1\}\cup\{1,\dots,N\}$ where $w_{-k}=1/(qw_k)$ for each $k>0$. Now to complete the argument to prove the identity \eqref{eq:boundary scattering proof 4}, we show that \eqref{eq:G func sum 1} constitutes a suitable expression on which to demonstrate the properties \textbf{(i)--(iii)} from the proof of Lemma \ref{lm:F-1 evaluation}.

        \vspace{5pt}
		\begin{enumerate}[label=\textbf{(\Roman*)},wide=0pt,itemsep=5pt]
			\item The residue at $w_k=1$:
			\[\res_{w_k=1} \mathcal{G}(w_1,\dots,w_N) = \alpha^{N-1}(q-1).\]
			To see this property let us expand \eqref{eq:G func sum 1} to separate the positive and negative sum over $i$:
			\begin{multline}
				\label{eq:G func sum 2}
				\qquad \quad\mathcal{G}(w_1,\dots,w_N) \\ = \frac{\alpha^{N-1}}{q+\alpha}\sum_{i=1}^N\varphi_1(w_i)\left(\frac{q(1-w_i)}{1-q w_i}+\sum_{j=1\atop j\neq i}^N\frac{(1-q)^2w_j}{(1-w_j)(1-qw_j)}\right)\prod_{j=1\atop j\neq i}^N\left[\frac{w_i-qw_j}{w_i-w_j}\frac{1-w_iw_j}{1- qw_iw_j}\right] \\
				+ \frac{\alpha^{N-1}}{q+\alpha} \sum_{i=1}^N\varphi_1\left(\frac{1}{qw_i}\right)\left(\frac{1-qw_i}{1- w_i}+\sum_{j=1\atop j\neq i}^N\frac{(1-q)^2w_j}{(1-w_j)(1-qw_j)}\right)\prod_{j=1\atop j\neq i}^N\left[\frac{qw_i-w_j}{w_i-w_j}\frac{1-q^2w_iw_j}{q(1- qw_iw_j)}\right].
			\end{multline}
			Now, recall that $\varphi_1(w_i)$ has a simple pole at $w_i=1$ while $\varphi_1\left(\frac{1}{qw_i}\right)$ is analytic at $w_i=1$. Hence, we may recognize that each term in both sums in \eqref{eq:G func sum 2} has a simple pole at the point $w_k=1$ for each integer $1\leq k\leq N$. So evaluating the residue of $\mathcal{G}(w_1,\dots,w_N)$ at $w_k=1$ amounts to calculating this residue on each term in both sums. This yields:
			\begin{multline}
				\label{eq:G func sum 3}
				\res_{w_k=1} \mathcal{G}(w_1,\dots,w_N) 
				= \frac{\alpha^{N-1}(q-1)}{q+\alpha} \sum_{j=1\atop j\neq k}^N \frac{(1-q)^2w_j}{(1-w_j)(1-qw_j)} + \alpha^{N-1}(q-1) \prod_{j=1\atop j\neq k}^N\left[\frac{q-w_j}{1-w_j}\frac{1-q^2w_j}{q(1-qw_j)}\right]\\
				+ \frac{\alpha^{N-1}(q-1)}{q+\alpha} \sum_{i=1\atop i\neq k}^N \varphi_1(w_i)\frac{q-w_i}{1-qw_i} \prod_{j=1\atop j\neq i,j\neq k}^N\left[\frac{w_i-qw_j}{w_i-w_j}\frac{1-w_iw_j}{1- qw_iw_j}\right] \\
				+ \frac{\alpha^{N-1}(q-1)}{q+\alpha} \sum_{i=1\atop i\neq k}^N\varphi_1\left(\frac{1}{qw_i}\right)\frac{1-q^2w_i}{q(1-w_i)}\prod_{j=1\atop j\neq i,j\neq k}^N\left[\frac{qw_i-w_j}{w_i-w_j}\frac{1-q^2w_iw_j}{q(1- qw_iw_j)}\right].
			\end{multline}            
			The first two terms in \eqref{eq:G func sum 3} are due to the residues of the terms in each sum of \eqref{eq:G func sum 2} corresponding to $i=k$, while the remaining sums correspond to the terms in \eqref{eq:G func sum 2} which are due to the terms corresponding with $i\neq k$. 
			
			For some arbitrary integer $1\leq\ell\leq N$ satisfying $\ell\neq k$, let us denote the right hand side of \eqref{eq:G func sum 3} as a function of $w_\ell$ by $g(w_\ell)$.
			From here, we claim that $g(w_\ell)$ is in fact a constant. In order to demonstrate this, we first show that $g(w_\ell)$ is an entire function of $w_\ell$. The possible singularities are at the points $w_\ell=w_j,1/(qw_j)$ for $j\neq\ell,j\neq k$ and also $w_\ell=1,q^{-1},\pm q^{-1/2}$. However, by virtue of Proposition \ref{prop:F function meromorphic}, $\mathcal{G}(w_1,\dots,w_N)$ is a meromorphic function in $w_\ell$ whose only singularities occur at $w_\ell=1,q^{-1}$. It then follows that the singularities at $w_\ell= w_j,1/(qw_j)$ present in $g(w_\ell)$ are removable.
			
			It remains to show that the singularities at $w_\ell=1,q^{-1}$ are removable. Let us first consider the apparent singularities of \eqref{eq:G func sum 3} at $w_\ell=1$, where each term has a simple pole. And so, neglecting the overall constants, we compute the residue of the first two terms on the right hand side as
			\begin{multline}
				\label{eq:G func sum 4}
				\res_{w_\ell=1} \left\{\sum_{j=1\atop j\neq k}^N \frac{(1-q)^2w_j}{(1-w_j)(1-qw_j)} + (q+\alpha) \prod_{j=1\atop j\neq k}^N\left[\frac{q-w_j}{1-w_j}\frac{1-q^2w_j}{q(1-qw_j)}\right]\right\} \\ = q-1+\frac{(q+\alpha)\left(1-q^2\right)}{q}\prod_{j=1\atop j\neq k,j\neq k}^N\left[\frac{q-w_j}{1-w_j}\frac{1-q^2w_j}{q(1-qw_j)}\right].
			\end{multline}
			A simple computation shows the contribution of the same residue on the two sums over $i$ in \eqref{eq:G func sum 3} exactly cancels \eqref{eq:G func sum 4}. From this we may conclude
			\[\res_{w_\ell=1} \res_{w_k=1}\mathcal{G}(w_1,\dots,w_N) = 0\]
			so that the $g(w_\ell)$ is analytic at $w_\ell=1$. Although we do not present it, an analogous computation also shows that $g(w_\ell)$ is also analytic at $w_\ell=q^{-1}$, and thus we conclude that it is an entire function of $w_\ell$. 
			
			By looking at each individual rational factor in the terms of the right hand side of \eqref{eq:G func sum 3}, it is straightforward to see that the $g(w_\ell)$ is bounded in the limit $w_\ell\to\infty$.
			
			And so, we may conclude that $g(w_\ell)$ is a bounded entire function in $w_\ell$. Since $\ell$ was arbitrary, we may conclude that it is a bounded entire function in the remaining $w_1,\dots,w_{k-1},w_{k+1},\dots,w_N$ alphabet, and by Liouville's theorem, it is a constant. As such, we may evaluate the right hand side of \eqref{eq:G func sum 3} by taking any specialization of the remaining $w$-alphabet. It is convenient to choose $w_1=\cdots=w_{k-1}=w_{k+1}=\cdots=w_N=0$ so that all terms in \eqref{eq:G func sum 3} vanish except the second one which yields the desired result:
			\begin{equation}
				\res_{w_k=1} \mathcal{G}(w_1,\dots,w_N) =\alpha^{N-1}(q-1) \prod_{j=1\atop j\neq k}^N\left[\frac{q-w_j}{1-w_j}\frac{1-q^2w_j}{q(1-qw_j)}\right]\Bigg|_{w_1=\cdots=w_N=0} = \alpha^{N-1}(q-1).
			\end{equation}
			\item The residue at $w_k=q^{-1}$:
			\[\res_{w_k=q^{-1}} \mathcal{G}(w_1,\dots,w_N) = \alpha^{N-1}\left(q^{-1}-1\right),\]
			which follows by a analogous argument to the one used to calculate the residue at $w_k=1$.
			\item The function $\mathcal{G}$ inherits the property
			\[\mathcal{G}(w_1,\dots,w_N)\Big|_{w_k=0} = \alpha \mathcal{G}(w_1,\dots,w_{k-1},w_{k+1},\dots,w_N)\]
			from Proposition \ref{prop:F-recurrence}.\qedhere
		\end{enumerate}
	\end{enumerate}
\end{proof}
\subsection{Proof of initial condition}
\label{sec:initial cond proof}
Let us recall the result of Lemma \ref{lm: orthogonality}, where the following orthogonality statement on the function $\mathcal{F}_y(w)$ from \eqref{eq:F initial coord function}
\begin{multline}
    \label{eq:orthog t=0 repeat}
    \oint_{\mathcal{C}_1} \frac{\dd w_1}{2\pi\ii} \cdots \oint_{\mathcal{C}_N} \frac{\dd w_N}{2\pi\ii} \prod_{1\leq i<j\leq N}\left[\frac{w_j-w_i}{qw_j-w_i}\frac{1-qw_iw_j}{1- w_iw_j}\right] \\
    \times \prod_{i=1}^N\left[\frac{1-q w_i^2}{w_i(q+\alpha-1-\alpha w_i)(1-qw_i)} \left(\frac{1-w_i}{1-qw_i}\right)^{x_i-1}\right] \mathcal{F}_y(w_1,\dots,w_N) = \delta_{x,y},
\end{multline}
serves as the initial condition of the half-space ASEP transition probability \eqref{eq:ASEP transition prob}. The integration contours of \eqref{eq:orthog t=0 repeat} satisfy the requirements of Definition \ref{defn:ASEP nested contours}.
\begin{proof}[Proof of Lemma \ref{lm: orthogonality}]
	The proof of this result is rather technical, so we will proceed in several steps. For some $\sigma\in\mathcal{B}_N$ and half-space configurations $x=(x_1,\dots,x_N),y=(y_1,\dots,y_M)$, let us define the following integral expression:
	\begin{multline}
		\label{eq:orthog proof I-defn}
		\mathcal{I}_{x,y}(\sigma) := \oint_{\mathcal{C}_1} \frac{\dd w_1}{2\pi\ii} \cdots \oint_{\mathcal{C}_N} \frac{\dd w_N}{2\pi\ii} \prod_{1\leq i<j\leq N}\left[\frac{w_i-w_j}{w_i-qw_j}\frac{1-qw_iw_j}{1- w_iw_j}\frac{w_{\sigma(i)}-qw_{\sigma(j)}}{w_{\sigma(i)}-w_{\sigma(j)}}\frac{1-w_{\sigma(i)}w_{\sigma(j)}}{1-qw_{\sigma(i)}w_{\sigma(j)}}\right] \\
		\times \prod_{i=1}^N\psi_{x_i}(w_i) \prod_{i=1}^M\varphi_{y_i}\left(w_{\sigma(i)}\right),
	\end{multline}
	where the function defined by
	\[\psi_x(w):=\frac{1-q w^2}{w(q+\alpha-1-\alpha w)(1-qw)} \left(\frac{1-w}{1-qw}\right)^{x-1}\]
	is analytic at $w=1$ for all integers $x>0$. It is important to emphasize that the contours in the integral \eqref{eq:orthog proof I-defn} are defined according to Definition \ref{defn:ASEP nested contours} and do not depend on which signed permutation is being considering. The orthogonality statement \eqref{eq:orthog t=0 repeat} can then be restated as:
	\begin{equation}
		V_{N-M}\alpha^{N-M} \sum_{\sigma\in\mathcal{B}_N}\mathcal{I}_{x,y}(\sigma) = \delta_{x,y},
	\end{equation}
	where in particular, the constants outside the sum reduce to 1 whenever $M=N$. Our argument is then split into the following claims which together imply the result.
	\begin{enumerate}[label=\textbf{(\roman*)},wide=0pt,itemsep=5pt]
		\item For any configurations with $0\leq M\leq N$, $\mathcal{I}_{x,y}(\sigma) = 0$ for any $\sigma\in\mathcal{B}_N\backslash S_N$.
		\item For any $\sigma\in S_N$, $\mathcal{I}_{x,y}(\sigma) = 0$ whenever $M<N$.
		\item Whenever $M=N$ the following summation vanishes:
		\[\sum_{\sigma\in S_N\backslash\{\id\}}\mathcal{I}_{x,y}(\sigma) = 0.\]
		\item When $\sigma=\id$ and $M=N$, the integral may be evaluated $\mathcal{I}_{x,y}(\id) = \delta_{x,y}$.
	\end{enumerate}
	We will proceed by proving these claims individually. Before proceeding, it is useful to remark on the strategy within the proof regarding the order of integration of the contour alignments of Definition \ref{defn:ASEP nested contours}. The multiple integral expression \eqref{eq:orthog proof I-defn} may be evaluated by successively taking residues at $w_1=1$, followed by $w_2=1$, and so on. However, we may switch the order of integration by deforming the contour of the $w_2$-integral from $\mathcal{C}_2$ to lie completely within the interior of $\mathcal{C}_1$. Due to the residue theorem, this deformation comes at the expense of the evaluation of residues due to possible poles at the points $w_2 = qw_1,w_1^{-1}$, or any other so-called \emph{dynamic poles} that are crossed by this process. Due to the scattering factor in the integrand, the dynamic poles which are crossed will depend on the signed permutation $\sigma$.
	
	In order to proceed, the following expressions will prove to be useful
	\begin{align*}
		\psi_{x_i}(w_i)\varphi_{y_{\sigma^{-1}(i)}}(w_i) & = (q-1)\frac{1}{1-w_i}\frac{1}{1-q w_i}\left(\frac{1-w_i}{1-qw_i}\right)^{x_i-y_{\sigma^{-1}(i)}},\\
		\psi_{x_i}(w_i)\varphi_{y_{\sigma^{-1}(i)}}\left(\frac{1}{qw_i}\right) & = (q-1)\frac{\alpha+q w_i(1-q-\alpha)}{1-q-\alpha+\alpha w_i}\frac{q^{y_{\sigma^{-1}(i)}-1}}{(1-w_i)^2}\left(\frac{1-w_i}{1-qw_i}\right)^{x_i+y_{\sigma^{-1}(i)}}.
	\end{align*}
	In particular, since all individual coordinates are positive, the second expression has a removable singularity at $w_i=1$. Additionally, whenever $\sigma\in S_N$ is an element of the symmetric group (that is, a signed permutation without any sign flips) then the scattering factor in the integrand of \eqref{eq:orthog proof I-defn} may be presented as product over the inversions of $\sigma$:
	\begin{equation}
		\prod_{1\leq i<j\leq N}\left[\frac{w_i-w_j}{w_i-qw_j}\frac{1-qw_iw_j}{1- w_iw_j}\frac{w_{\sigma(j)}-qw_{\sigma(i)}}{w_{\sigma(i)}-w_{\sigma(j)}}\frac{1-w_{\sigma(i)}w_{\sigma(j)}}{1-qw_{\sigma(i)}w_{\sigma(j)}}\right] = (-1)^{\abs{\sigma}} \prod_{1\leq i<j\leq N\atop \sigma(i)>\sigma(j)} \frac{w_{\sigma(i)}-qw_{\sigma(j)}}{w_{\sigma(j)}-qw_{\sigma(i)}}.
	\end{equation}
	We now proceed with the proof of \textbf{(i)--(iv)}.
	\begin{enumerate}[label=\textbf{(\roman*)},wide=0pt,itemsep=5pt]
		\item Let $\sigma$ be a fixed signed permutation which contains at least one sign-flip, so that $\sigma\notin S_N$. Let $k\in\{1,\dots,N\}$ be the fixed integer that corresponds to the left-most sign flip within $\sigma$. That is, let $\ell\in\{1\dots,N\}$ be the index satisfying $\sigma_\ell = -k$, so that we may represent the signed permutation as
		\[\sigma= (\sigma_1,\dots,\sigma_{\ell-1},-k,\sigma_{\ell+1},\dots,\sigma_N),\]
		where $\sigma_1,\dots,\sigma_{\ell-1}>0$. We now show that $\mathcal{I}_{x,y}(\sigma)=0$. First consider the integral expression \eqref{eq:orthog proof I-defn} for a signed permutation with $k=1$. The $\mathcal{C}_1$ contour is the inner-most contour according to Definition \ref{defn:ASEP nested contours}, and so, we may evaluate the $\mathcal{I}_{x,y}(\sigma)$ by taking the residue at $w_1=1$. The $w_1$-dependence of the integrand is given by
		\begin{multline*}
			\prod_{j=2}^N\left[\frac{w_1-w_j}{w_1-qw_j}\frac{1-qw_1w_j}{1- w_1w_j}\right]\prod_{i=1}^{\ell-1}\frac{w_{\sigma_i}-w_{1}^{-1}}{w_{\sigma_i}-w_{1}^{-1}/q}\prod_{j=\ell+1}^N\frac{w_{1}^{-1}/q-qw_{\sigma_j}}{w_{1}/q-w_{\sigma_j}}\\
			\times \prod_{i\neq\ell}\frac{1-w_{\sigma_i}w_{1}^{-1}/q}{1-w_{\sigma_i}w_{1}^{-1}}\begin{cases}
				\psi_{x_1}(w_1)\varphi_{y_\ell}\left(\frac{1}{qw_1}\right) & \text{ if }\ell\leq M \\
				\psi_{x_1}(w_1) & \text{ if }\ell> M
			\end{cases},
		\end{multline*}
		which is analytic at $w_1=1$. We thus conclude that when $k=1$ the integral expression \eqref{eq:orthog proof I-defn} vanishes.
		
		Now let us assume that $k>1$. We aim to evaluate the $w_k$-integral in $\mathcal{I}_{x,y}(\sigma)$ by first deforming its contour from $\mathcal{C}_k$ to be contained within the interior of $\mathcal{C}_1$, and consequently also within the interiors of $\mathcal{C}_2,\dots,\mathcal{C}_{k-1}$. The possible dynamic poles of $w_k$ in the integrand of \eqref{eq:orthog proof I-defn} are listed below. All of these poles are either simple poles or removable singularities. 
		\vspace{8pt}
		\begin{itemize}
			\item $w_k=q^{-1}w_i$ for all $i<k$
			\item $w_k=qw_j$ for all $j>k$
			\item $w_k=w_i^{-1}$ for all $i\neq k$
			\item $w_k = q^{-1}w_{\sigma(i)}^{-1}$ for all $i<\ell$ (where $\sigma(i)>0$)
			\item $w_k = w_{\abs{\sigma(j)}},q^{-1}w_{\abs{\sigma(j)}}^{-1}$ for all $j>\ell$ 
			\item $w_k = w_{{\sigma(i)}}$ for all $i\neq \ell$ whenever $\sigma(i)>0$
		\end{itemize}
		\vspace{8pt}
		By virtue of the contour alignment in Definition \ref{defn:ASEP nested contours} and the numerator of the integrand of \eqref{eq:orthog proof I-defn}, the only possible singularities crossed by deforming the contour of the $w_k$-integral are those at $w_k=w_i^{\pm1}$ for all $i<k$. Note that the factor $\prod_{i\neq k}(w_i-w_k)$ in the numerator of the integrand implies that poles of the form $w_k=w_i$ are removable, so that we only need to consider the effect of poles of the form $w_k=w_i^{-1}$. We will now show that the contribution of these residues vanishes for all $1\leq i<k$ so that the contour of the $w_k$-integral can be deformed to lie arbitrarily close to 1. 
		
		We first detail how the singularities $w_k=w_i^{-1}$ appear in the integrand of \eqref{eq:orthog proof I-defn}. For a fixed $\sigma\in\mathcal{B}_N$ whose left-most sign flip is $\sigma_{\ell}=-k$, consider all possible values $p\in\{1,\dots,k-1\}$ such that $p\neq\sigma_1,\dots,\sigma_{\ell-1}$. These possible values of $p$ precisely correspond to the simple poles of the form $w_k=w_p^{-1}$ that are present in the integrand. For each such $p$, we also define $m=\sigma^{-1}(\pm p)$, which is the index that maps to either $+p$ or $-p$ under $\sigma$. We note that due to the restrictions on the values of $p$, it follows that $m>\ell$. The residue of all simple poles $w_k=w_p^{-1}$ ultimately do not contribute to the evaluation of $\mathcal{I}_{x,y}(\sigma)$ so that we may deform the contour of the $w_k$-integral. If no such $p$ exists then the prescribed deformation can be performed freely without crossing any singularities.
		
		Denote by $A$ the set of all possible values of $p$ as described above. In order to show that $\mathcal{I}_{x,y}(\sigma)=0$, we will have to show that the any contributions due to residues at $w_k=w_{p}^{-1}$ vanish for all possible combinations of $p$. More explicitly for any non-empty $\{p_1,\dots, p_n\}\subseteq A$ with $p_1>\cdots >p_n$, we will show that the following combination of successive residues vanishes:
		\begin{equation}
			\label{eq:orthog proof sign 1}
			\oint_{\widetilde{\mathcal{C}}} \frac{\dd w_{p_n}}{2\pi\ii} \res_{w_k=w_{p_1}^{-1}}\res_{w_{p_1}=w_{p_2}^{-1}}\cdots\res_{w_{p_{n-1}}=w_{p_n}^{-1}} g_{x,y}(w_1,\dots,w_N) =0,
		\end{equation}
		where $g_{x,y}$ is the integrand of \eqref{eq:orthog proof I-defn} and $\widetilde{\mathcal{C}}$ is a small circular complex contour, enclosing the point 1 with an arbitrarily small radius. For any value $p$ as described above, consider the following residue on the $w_k,w_p$-dependent factors of $g_{x,y}$:
		\begin{equation}
			\label{eq:orthog proof sign 2}
			\res_{w_k=w_p^{-1}} \prod_{1\leq i<j\leq N}\left[\frac{w_i-w_j}{w_i-qw_j}\frac{1-qw_iw_j}{1- w_iw_j}\frac{w_{\sigma_i}-qw_{\sigma_j}}{w_{\sigma_i}-w_{\sigma_j}}\frac{1-w_{\sigma_i}w_{\sigma_j}}{1-qw_{\sigma_i}w_{\sigma_j}}\right] \psi_{x_k}(w_k)\psi_{x_p}(w_p) \widetilde{\varphi}_{\ell}\left(\frac{1}{qw_k}\right)\widetilde{\varphi}_{m}\left(w_{\pm p}\right),
		\end{equation}
		where 
		\[\widetilde{\varphi}_\ell(z) = \begin{cases}
			\varphi_{y_\ell}(z) & \text{ if }\ell\leq M \\
			1 & \text{ if }\ell>M
		\end{cases}.\]
		Should there be a simple pole at $w_k=w_p^{-1}$, then \eqref{eq:orthog proof sign 2} can be evaluated which yields a factor of the form:
		\[
		\psi_{x_k}\left(\frac{1}{w_p}\right)\psi_{x_p}(w_p) \widetilde{\varphi}_{\ell}\left(\frac{w_p}{q}\right)\widetilde{\varphi}_{m}\left(w_{\pm p}\right) \sim
		\begin{cases}
			(1-w_p)^{x_k+x_p-2} & \text{ if }\ell,m>M \\
			(1-w_p)^{x_k+x_p+y_\ell-3} & \text{ if }\ell\leq M,m>M \\
			(1-w_p)^{x_k+x_p+(y_\ell-y_m)-3} & \text{ if }\ell,m<M,\sigma(m)=+p \\
			(1-w_p)^{x_k+x_p+y_\ell+y_m-4} & \text{ if }\ell,m<M,\sigma(m)=-p
		\end{cases}.
		\]
		We note here that since $x_i,y_j\geq1$ for all $1\leq i\leq N,1\leq j\leq M$, all contributions are of the form $(1-w_p)^b$ for some non-negative integer $b\geq0$. The third item can be seen to have a non-negative exponent since $m>\ell$, it follows that $y_\ell-y_m\geq1$. For each successive residue evaluation in the chain $p_1>\cdots>p_n$ in \eqref{eq:orthog proof sign 1}, we obtain a factor of the same form. Evaluating all the residues leads to a factor $\left(1-w_{p_n}\right)^{b_1+\dots+b_n}$ in the $w_{p_n}$ integral of \eqref{eq:orthog proof sign 1} which is analytic at $w_{p_n}=1$ since all $b_i\geq0$. Therefore, the $w_{p_n}$ integral of \eqref{eq:orthog proof sign 1} vanishes. Finally, since this is true for all non-empty sets $A$, the multiple residues at the points $w_k=w_{p_i}^{-1}$ do not contribute to the evaluation of \eqref{eq:orthog proof I-defn}. Thus we can deform the contour of the $w_k$-integral to lie closely around $w_k=1$ without picking up additional residues. The integration over $w_k$ then vanishes and we thus conclude that $\mathcal{I}_{x,y}(\sigma)=0$ whenever $\sigma\in \mathcal{B}_N\backslash S_N$. 
		\item For a fixed $\sigma\in S_N$ and configurations $x,y$ satisfying $M<N$, we will prove here that the multiple integral expression \eqref{eq:orthog proof I-defn} vanishes. Let us define $k\in\{1,\dots,N\}$ such that $\sigma_N=k$. 
		
		Our proof strategy will be similar to the previous part, whereby we aim to deform the contour of the $w_k$-integral from $\mathcal{C}_k$ to be arbitrarily close to 1. Since $M<N$, the integrand of \eqref{eq:orthog proof I-defn} does not contain a $\varphi$ function with $w_k$ as an argument. Hence, due to the form of the $\psi$ function, the integrand is analytic at $w_k=1$. With this in mind, consider the $w_k$-dependent factors in the integrand:
		\[
		(-1)^{\abs{\sigma}} \prod_{1\leq i<N\atop \sigma(i)>k}\frac{w_{\sigma(i)}-qw_k}{w_k-q w_{\sigma(i)}}.
		\]
		From this expression we can conclude that the only possible singularities are of the form $w_k=qw_{\sigma(i)}$ for $\sigma(i)>k$. The alignment of the contours outlined in Definition \ref{defn:ASEP nested contours} ensures that these singularities are not enclosed by the $\mathcal{C}_k$ contour so that the integrand is an analytic function of $w_k$ both on and within the interior of $\mathcal{C}_k$. As a result, the integral \eqref{eq:orthog proof I-defn} vanishes.
		\item For any fixed permutation $\sigma\in S_N$ and configurations $x,y$ both containing $M=N$ particles, the integral expression \eqref{eq:orthog proof I-defn} can be written as
		\begin{multline}
			\label{eq:orthog proof perm 1}
			\mathcal{I}_{x,y}(\sigma) = (q-1)^N \oint_{\widetilde{\mathcal{C}}} \frac{\dd w_1}{2\pi\ii} \cdots \oint_{\widetilde{\mathcal{C}}} \frac{\dd w_N}{2\pi\ii} (-1)^{\abs{\sigma}} \prod_{1\leq i<j\leq N\atop \sigma_i>\sigma_{j}} \frac{w_{\sigma_i}-qw_{\sigma_j}}{w_{\sigma_j}-qw_{\sigma_i}} \\ \times \prod_{i=1}^N\left[\frac{1}{1-w_i}\frac{1}{1-qw_i}\left(\frac{1-w_i}{1-qw_i}\right)^{x_i-y_{\sigma^{-1}(i)}}\right].
		\end{multline}
		We emphasize here that in the absence of sign flips in the signed permutation, the contours have been deformed to coincide, $\mathcal{C}_i\mapsto\widetilde{\mathcal{C}}$, where $\widetilde{\mathcal{C}}$ is a positively-oriented contour enclosing 1 with an arbitrarily small radius. From here, we claim that the expression \eqref{eq:orthog proof perm 1} has the desired property
		\begin{equation}
			\label{eq:orthog proof perm 2}
			\sum_{\sigma\in S_N\backslash\{\id\}}\mathcal{I}_{x,y}(\sigma) = 0,
		\end{equation}
		which is equivalent to a claim within the proof of Theorem 2.1 of \cite{tracy_integral_2008}. We do not replicate the argument needed to prove the statement \eqref{eq:orthog proof perm 2} due to its technical nature. However, the precise details of the matching with \cite{tracy_integral_2008} are provided explicitly in Appendix \ref{sec:TW-match}.
		\item Consider finally the integral associated with the identity signed permutation $\id\in\mathcal{B}_N$ for half-space configurations $x,y$ with $M=N$ :
		\begin{equation}
			\mathcal{I}_{x,y}(\id) = (q-1)^N \oint_{\mathcal{C}_1} \frac{\dd w_1}{2\pi\ii} \cdots \oint_{\mathcal{C}_N} \frac{\dd w_N}{2\pi\ii}\prod_{i=1}^N \left[\frac{1}{1-w_i}\frac{1}{1-q w_i}\left(\frac{1-w_i}{1-qw_i}\right)^{x_i-y_{i}}\right].
		\end{equation}
		This integral is readily evaluated by taking successively taking residues at $w_i=1$ which yields $\mathcal{I}_{x,y}(\id) = \delta_{x,y}$ by the residue theorem.\qedhere
	\end{enumerate}
\end{proof}

\appendix

\section{Properties of Pfaffians}

\subsection{Definition and identities}

A \emph{Pfaffian} is defined for a skew-symmetric $2n\times2n$-dimensional matrix $A$ with entries $a_{i,j}$ as
\begin{equation}
	\pf A = \frac{1}{2^n n!} \sum_{\sigma\in S_n} (-1)^{\abs{\sigma}} \prod_{i=1}^n a_{\sigma(2i-1),\sigma(2i)}.
\end{equation}
The determinant of the same even dimensional skew-symmetric matrix $A$ is the square of a polynomial in its entries. It can be shown that the determinant is exactly the square of the Pfaffian:
\begin{equation}
	\det A = \left(\pf A\right)^2.
\end{equation}
The Pfaffian of a skew-symmetric $2n\times2n$ matrix $A$ changes sign under simultaneous interchange of rows and columns of the same index.
A simple consequence is the following: if $A$ is given and $\hat A$ is the reordered matrix
\[\left[\hat A \right]_{i,j}=A_{\sigma(i),\sigma(j)}\]
with $\sigma\in S_{2n}$ specified as 
\[\sigma(2i-1)=i,\qquad\sigma(2i)=n+i\qquad\text{for each }i=1,\dotsc,n,\]
then
\begin{equation}\label{eq:matrix shuffle pfaffian}
    \Pf\left(\hat A\right)=(-1)^{\binom{n}{2}}\Pf(A).
\end{equation}
Pfaffians may also be thought of as generalizations of determinants: for an arbitrary $m\times m$-dimensional matrix $B$, the following identity also holds
\begin{equation}
	\label{eq:block pfaffian is det}
	\det(B) = (-1)^{\binom{m}{2}} \pf \left[\begin{array}{cc}
		0 & B \\[5pt]
		-B^T & 0
	\end{array}\right].
\end{equation}
The following Pfaffian identity, found in \cite{okada_pfaffian_2019}, is a generalization of \eqref{eq:block pfaffian is det}.
\begin{prop}
    \label{prop:block pf eval}
    Let $A$ be a skew-symmetric $n\times n$-dimensional matrix and let $B$ be a $n\times m$-dimensional matrix. The following Pfaffian evaluation holds: whenever $n+m$ is even
    \begin{equation}
        \pf \left[\begin{array}{cc}
            A & B \\[5pt]
            -B^T & 0
        \end{array}\right] = (-1)^{\binom{n}{2}}\times\begin{dcases}
            \sum_S (-1)^{\sum_i S_i}\pf A_S \det B_{(S^\mathsf{C};[m])} & \text{ if }n>m \\
            \det B & \text{ if }n=m \\
            0 & \text{ if } n<m
        \end{dcases},
    \end{equation}
    and whenever $n+m$ is odd
    \begin{equation}
        \pf \left[\begin{array}{ccc}
            A & \vec{\bm1}_n & B \\[5pt]
            -\vec{\bm1}_n^T & 0 & 0 \\[5pt]
            -B^T & 0 & 0
        \end{array}\right] = (-1)^{\binom{n}{2}+n-m}\times\begin{dcases}
            \sum_S (-1)^{\sum_i S_i}\pf \left[\begin{array}{cc}
                A_S & \vec{\bm1}_{n-m} \\[5pt]
                -\vec{\bm1}^T_{n-m} & 0
            \end{array}\right] \det B_{(S^\mathsf{C};[m])} & \text{ if }n>m \\
            0 & \text{ if } n<m
        \end{dcases},
    \end{equation}
    where the sum is over subsets $S\subseteq\{1,\dots,n\}$ of cardinality $n-m$ and $S^\mathsf{C}$ denotes its complement. We have used the shorthand notation $[m] = \{1,\dots,m\}$.
\end{prop}
\begin{proof}
    The version of the result where $n+m$ is even is given explicitly as Corollary 2.4 of \cite{okada_pfaffian_2019}. Meanwhile, the odd version can be obtained as a special case of the even version where the last row and column of an extended skew-symmetric matrix $A$ are set to be $\pm 1$ appropriately and the last row of an extended matrix $B$ is set to zero. 
\end{proof}

The following identity, due to Stembridge \cite{stembridge_nonintersecting_1990} is crucial to our work.
\begin{lm}[Stembridge's Pfaffian]
	\label{lm: stembridge pfaffian}
	Let $x_1,\dots,x_m$ be an alphabet of fixed length $m$. The following Pfaffian identities hold: when $m=2n$ is even 
	\begin{equation}
		\label{eq: stembridge pfaffian even}
		\prod_{1\leq i<j\leq 2n} \frac{x_i-x_j}{1-x_ix_j} = \pf \left[S(x_i,x_j)\right]_{1\leq i,j\leq 2n},
	\end{equation}
	and when $m=2n-1$ is odd:
	\begin{equation}
		\label{eq: stembridge pfaffian odd}
		\prod_{1\leq i<j\leq 2n-1} \frac{x_i-x_j}{1-x_ix_j} = \pf\left[\begin{array}{cc}
			S(x_i,x_j)_{1\leq i,j\leq 2n-1} & \vec{\bm1}_{2n-1} \\[8pt]
			-\vec{\bm1}^T_{2n-1} & 0
		\end{array}\right],
	\end{equation}
    where 
    \[S(x_i,x_j) = \frac{x_i-x_j}{1-x_ix_j}.\]
\end{lm}

\subsection{Pfaffian integration formulas}

The following result, due to de Bruijn \cite{de_bruijn_multiple_1955}, is the Pfaffian analogue of the Andr\'{e}ief identity.
\begin{lm}[de Bruijn's integration formula]
	\label{lm:de bruijn formula}
	Let $m>0$ be a fixed integer and let $\mu$ be an appropriate measure on a region $X$. Let $g_1,\dots, g_m$ be functions and let $h$ be an skew-symmetric function of two variables which are all integrable over $X$ with respect to $\mu$. Then the following identities hold: when $m=2n$ is even 
	\begin{equation}
		\label{eq:de bruijn formula even}
		\frac{1}{(2n)!}\int_{X} \dd \mu(w_1) \cdots \int_{X} \dd \mu(w_{2n}) \det[g_j(w_i)]_{1\leq i,j\leq 2n}\pf[h(w_i,w_j)]_{1\leq i,j\leq 2n}
		= \pf(A),
	\end{equation}
	and when $m=2n-1$ is odd
	\begin{multline}
		\label{eq:de bruijn formula odd}
		\frac{1}{(2n-1)!}\int_{X} \dd \mu(w_1) \cdots \int_{X} \dd \mu(w_{2n-1}) \det[g_j(w_i)]_{1\leq i,j\leq 2n-1} 
		\pf\left[
		\begin{array}{cc}
			h(w_i,w_j)_{1\leq i,j\leq 2n-1} &  \vec{\bm1}_{2n-1} \\[8pt]
			-\vec{\bm1}_{2n-1}^T & 0
		\end{array}\right] \\
		= \pf\left[\begin{array}{cc}
			A & \vec{b} \\[5pt]
			-\vec{b}^T & 0
		\end{array}\right].,
	\end{multline}
    where $A$ is a skew-symmetric $m\times m$-dimensional matrix and $\vec{b}$ is a vector of length $m$.
	The matrix elements are given as:
	\begin{align*}
		\left[A\right]_{i,j} & =  \int_{X} \dd \mu(w) \int_{X} \dd \mu(u) g_i(w)g_j(u)h(w,u), \\
		\left[\vec{b}\right]_i & =\int_{X} \dd \mu(w) g_i(w),
	\end{align*}
    for each $1\leq i,j\leq m$.
\end{lm}
The following Lemma is generalization of Lemma \ref{lm:de bruijn formula} which appeared in the work \cite{kieburg_derivation_2010}.
\begin{lm}[Generalized de Bruijn's integration formula]
	\label{lm:gen de bruijn formula}	
	Let $\ell\geq m>0$ be fixed integers and let $\mu$ be an appropriate measure on a region $X$. Let $f_1,\dots,f_{\ell-m},g_1,\dots, g_m$ be functions and let $h$ be an skew-symmetric function of two variables which are all integrable over $X$ with respect to $\mu$. Additionally let $S$ be a constant skew-symmetric matrix of dimension $(\ell-m)\times(\ell-m)$. Then the following identities hold: when $\ell$ is even
	\begin{multline}
		\label{eq:gen de bruijn formula even}
		\frac{(-1)^{\binom{\ell-m}{2}}}{m!}\int_{X} \dd \mu(w_1) \cdots \int_{X} \dd\mu(w_{m}) \det \left[g_j(w_i)\right]_{1\leq i,j\leq m}\pf\left[\begin{array}{cc}
			h(w_i,w_j)_{1\leq i,j\leq m}  &  f_k(w_i)_{1\leq i\leq m\atop1\leq k\leq \ell-m} \\[8pt]
			-f_k(w_j)_{1\leq k\leq \ell-m\atop1\leq j\leq m} & S
		\end{array}\right] \\
		= \pf\left[\begin{array}{cc}
			A & C \\[5pt]
			-C^T & S
		\end{array}\right]
	\end{multline}
	and when $\ell$ is odd
	\begin{multline}
		\label{eq:gen de bruijn formula odd}
		\frac{(-1)^{\binom{\ell-m}{2}}}{m!}\int_{X} \dd \mu(w_1) \cdots \int_{X} \dd\mu(w_{m}) \det \left[g_j(w_i)\right]_{1\leq i,j\leq m}
		\pf\left[\begin{array}{ccc}
			h(w_i,w_j)_{1\leq i,j\leq m} & \vec{\bm1}_{m} & f_k(w_i)_{1\leq i\leq m\atop1\leq k\leq \ell-m} \\[8pt]
			-\vec{\bm1}^T_{m} & 0 & 0 \\[8pt]
			-f_k(w_j)_{1\leq k\leq \ell-m\atop1\leq j\leq m} & 0 & S
		\end{array}\right] \\
	=  \pf \left[\begin{array}{ccc}
		A & \vec{b} & C \\[5pt]
		-\vec{b}^T & 0 & 0 \\[5pt]
		-C^T & 0 & S
	\end{array}\right] .
	\end{multline}
    where $A$ is a skew-symmetric $m\times m$-dimensional matrix, $C$ is an $m\times(\ell-m)$-dimensional matrix and $\vec{b}$ is a vector of length $m$.
	The matrix elements are given as:
	\begin{align*}
	\left[A\right]_{i,j} & = \int_X \dd \mu(u) \int_X \dd \mu(w) g_i(w)g_j(u)h(w,u), \\
	\left[C\right]_{i,k} & = \int_X \dd \mu(w) g_i(w)f_k(w), \\
		\left[\vec{b}\right]_{i} & = \int_X \dd \mu(w) g_i(w),
	\end{align*}
    for each $1\leq i,j\leq m$ and $1\leq k\leq \ell-m$.
\end{lm} 
In fact, a more general version of Lemma \ref{lm:gen de bruijn formula}	appeared in \cite{kieburg_derivation_2010} involving a additional constant matrix augmented to the rows of the matrix in the determinant of \eqref{eq:gen de bruijn formula even}. However, this more general result is not needed for our purposes here and we will not present it.

The well known Andr\'{e}ief integration identity can be recovered from \eqref{eq:gen de bruijn formula even} by taking $\ell=2m$ with $h(w,u)=0$ and $[S]_{i,j}=0$ for all $1\leq i,j\leq m$.

\section{Skew-biorthogonal polynomials}\label{appdx:skew}
In this appendix we provide explicit formulas for the skew-biorthogonal polynomials which are characterized by the relations \eqref{eq:Big Phi skew-biorhtog rels empty}. We start with a general result for the Cholesky type skew-Borel factorization of a skew-symmetric matrix.
\begin{prop}
    \label{prop:skew-borel}
    Let $m$ be a fixed integer and let $M$ be a fixed $2m\times2m$-dimensional skew-symmetric matrix such that the minors 
    \[[M]_{k,\ell}, \quad \text{for } 2j-1\leq k,\ell\leq2m\]
    are invertible for all $1\leq j\leq m$. Then $M$ has a \emph{skew-Borel factorization} of the form
    \begin{equation}
        M = R_{2m} J_{2m} R_{2m}^T,
    \end{equation}
    where $R_{2m}$ is an upper-triangular $2m\times2m$-dimensional matrix which is uniquely determined up to setting $\left[R_{2m}\right]_{2j-1,2j}=0$ and $\left[R_{2m}\right]_{2j,2j}=1$ for each $1\leq j\leq m$. Also where
    \[J_{2m} = \diag\left\{\left[\begin{array}{cc}
			0 & 1\\
			-1 & 0
		\end{array}\right],\dots,\left[\begin{array}{cc}
			0 & 1\\
			-1 & 0
		\end{array}\right]\right\}.\]
\end{prop}

A proof of this result can be found in \cite{benner_cholesky-like_2000} with an alternative formulation for the matrix $J_{2m}$. Instead of presenting a proof of Proposition \ref{prop:skew-borel}, we demonstrate that the matrix $R_{2m}$ can be explicitly given in terms of $2\times 2$ blocks of Pfaffian minors. For a $2m\times2m$ matrix, let us denote its construction in terms of $2\times2$ blocks by: \[\langle A\rangle_{i,j} = \left[\begin{array}{cc}
    \left[A\right]_{2i-1,2j-1} & \left[A\right]_{2i-1,2j} \\
    \left[A\right]_{2i,2j-1} & \left[A\right]_{2i,2j}
\end{array}\right],\]
for $1\leq i,j\leq m$.

\begin{prop}
\label{prop:skewBorelfactor}
Let $I\subseteq\{1\dots,2m\}$ be an ordered list and let 
\[\pf M_I:=\pf [M_{k,\ell}]_{k,\ell\in I}\]
denote the Pfaffian of the sub-matrix of $M$ consisting of columns and rows labeled by $I$ with the convention that $M_\emptyset=1$. The matrix $R_{2m}$ of the skew-Borel factorization of $M$ has the explicit form given in terms of $2\times2$ blocks as:
\begin{equation}
\label{eq:skew-Borel R-inv}
\langle R_{2m}\rangle_{i,j} = \left\{
\begin{array}{cl}
\displaystyle
\begin{pmatrix}
\displaystyle \frac{\pf M_{\{2i-1,2j\}\cup [2j+1,2m]}}{\pf M_{[2j+1,2m]}} & - \displaystyle\frac{\pf M_{\{2i-1,2j-1\}\cup [2j+1,2m]}}{\pf M_{[2j-1,2m]}} \\[5mm]
\displaystyle\frac{\pf M_{ \{2i,2j\}\cup [2j+1,2m]}}{\pf M_{[2j+1,2m]}} & - \displaystyle\frac{\pf M_{\{2i,2j-1\}\cup [2j+1,2n\}}}{\pf M_{[2j-1,2m]}} 
\end{pmatrix}, \qquad & 1\le i < j\le m
\\[10mm]
\displaystyle
\begin{pmatrix}
\displaystyle\frac{\pf M_{[2i-1,2m]}}{\pf M_{[2i+1,2m]}} & 0 \\[5mm]
0 & 1
\end{pmatrix}, & i=j \\[8mm]
\begin{pmatrix}
    0 & 0 \\
    0 & 0
\end{pmatrix}, & 1\le j < i\le m
\end{array}\right.,
\end{equation}
for each block $1\leq i,j\leq m$.
The inverse $R_{2m}^{-1}$ similarly may also be written in terms of $2\times 2$ blocks and is given explicitly by
\begin{equation}
\left\langle R_{2m}^{-1}\right\rangle_{i,j} = \left\{
\begin{array}{cl}
\displaystyle
\begin{pmatrix}
\displaystyle- \frac{\pf M_{[2i-1,2m]\backslash\{2i,2j-1\}}}{\pf M_{[2i-1,2m]}} & \displaystyle \frac{\pf M_{[2i-1,2m]\backslash \{2i,2j\}}}{\pf M_{[2i-1,2m]}} \\[5mm]
\displaystyle- \frac{\pf M_{[2i,2m]\backslash \{2j-1\}}}{\pf M_{[2i+1,2m]}} & \displaystyle \frac{\pf M_{[2i,2m]\backslash \{2j\}}}{\pf M_{[2i+1,2m]}} 
\end{pmatrix}, \qquad & 1\le i < j\le m
\\[10mm]
\displaystyle
\begin{pmatrix}
\displaystyle\frac{\pf M_{[2i+1,2m]}}{\pf M_{[2i-1,2m]}} & 0 \\[5mm]
0 & 1
\end{pmatrix}, & i=j \\[8mm]
\begin{pmatrix}
    0 & 0 \\
    0 & 0
\end{pmatrix}, & 1\le j < i\le m
\end{array}\right..
\end{equation}
\end{prop}
Recall that $\Psi(x,y)=Q_{1,1}(x,y)$, where $Q_{k,\ell}(x,y)$ is defined in \eqref{eq:Q-kernel} as
\begin{equation}
	Q_{k,\ell}(x,y) = \alpha^2 \oint_\contour \frac{\dd v}{2\pi\ii} \oint_\contour \frac{\dd w}{2\pi\ii} \frac{v-w}{1-v-w} \frac{w^{k-x}\mathrm{e}^{t(w-1)}}{(w-\alpha)(w-1)^k} \frac{v^{\ell-y}\mathrm{e}^{t(v-1)}}{(v-\alpha)(v-1)^\ell},
\end{equation}
The following is a simple application of the generalized binomial identity on the integrand of $Q_{k,\ell}(x,y)$
\begin{lm}\label{lem:Theta}
Let us denote, by $\Theta_i$, the polynomials from \eqref{eq:pre phi contour virtual} so that for a fixed integer $N$
\begin{equation}
   \Theta_i(x) := \phi_{[N-i+1,N]}(\virst_i,x) = \oint_{\gamma} \frac{\dd z}{2\pi\ii} \frac{1}{z^i} \left(\frac{1}{1-z}\right)^x = \binom{x+i-2}{x-1},
\end{equation}
where the contour surrounds $z=0$ while omitting all other singularities of the integrand. Then the following holds for all $1\leq i,j\leq N$:
\begin{equation}
    \label{eq:upsilon_Q}
    \sum_{x=1}^\infty \sum_{y=1}^\infty \Theta_i(x) Q_{1,1}(x,y) \Theta_j(y) =  Q_{i+1,j+1}(1,1).
\end{equation}
\end{lm}
Recall the definition of the skew-symmetric matrix $\mathcal{N}$ from \eqref{eq:cond prob proof empty N defn}. Lemma implies that the entries of this matrix are given by
\begin{equation}
    \left[\mathcal{N}\right]_{k,\ell} = \phi_{[k,N]}\conv \Psi\conv\left(\phi_{[\ell,N}\right)^T(\vir_k,\vir_\ell) = Q_{N-k+2,N-\ell+2}(1,1),
\end{equation}
for each $1\leq k,\ell\leq N$. The following result follows by the application of the explicit skew-Borel factorization of $\mathcal{N}=\mathcal{R}_N J_N\mathcal{R}_N^T$ from \eqref{eq:cond prob proof empty N skew-Borel}, so that $\mathcal{R}_N$ is an upper-triangular $N\times N$ matrix and determined by Proposition~\ref{prop:skewBorelfactor}. Let us also denote, for $I\subseteq [1,N]$ which has an even cardinality, the following sub-Pfaffian of the matrix which appears in Corollary \ref{cor:boundary current schutz pf}:
\begin{equation}
    \pf \widetilde{Q}_I := \pf\left[Q_{i+1,j+1}(1,1)\right]_{i,j\in I}.
\end{equation}
\begin{prop}
\label{prop:Upsilon2Phi}
Let us define the family of polynomials $\widetilde{\Phi}_0,\dots,\widetilde{\Phi}_{N-1}$ which satisfy the skew-biorthogonalization relations from \eqref{eq:Big Phi skew-biorhtog rels empty}, so that $\widetilde{\Phi}_k(x)$ is a polynomial in $x$ of degree $k$:
\begin{equation}
\label{eq:skew-biorhtog polyn appendix}
\sum_{x=1}^\infty \sum_{y=1}^\infty \widetilde{\Phi}_{N-k}(x) \Psi(x,y) \widetilde{\Phi}_{N-\ell}(y) = \left[J_N\right]_{k,\ell}, \qquad 1\le k,\ell \le N.
\end{equation}
These polynomials are constructed explicitly as
\begin{equation}
    \widetilde{\Phi}_{N-k}(x)=\sum_{j=1}^N \left[\mathcal{R}_N^{-1}\right]_{k,j}\Theta_{N-j+1}(x),
\end{equation}
where $\mathcal{R}_N^{-1}$ is the inverse of the upper-triangular skew-Borel matrix of $\mathcal{N}$. Moreover, it is constructed explicitly out of $2\times2$-blocks, where\footnote{The signs have been modified from the conventions of Proposition \ref{prop:skewBorelfactor} since the underlying matrices satisfy $\pf \mathcal{N}=(-1)^{N/2}\pf[Q_{k+1,\ell+1}(1,1)]_{1\leq,k,\ell\leq N}$.}
\begin{equation}
\left\langle\mathcal{R}_{N}^{-1}\right\rangle_{i,j} = \left\{
\begin{array}{cl}
\displaystyle
\begin{pmatrix}
\displaystyle \frac{\pf \widetilde{Q}_{[1,N-2i+2]\backslash\{N-2i+1,N-2j+2\}}}{\pf \widetilde{Q}_{[1,N-2i+2]}} & \displaystyle -\frac{\pf \widetilde{Q}_{[1,N-2i+2]\backslash\{N-2i+1,N-2j+1\}}}{\pf \widetilde{Q}_{[1,N-2i+2]}} \\[5mm]
\displaystyle- \frac{\pf \widetilde{Q}_{[1,N-2i+1]\backslash \{N-2j+2\}}}{\pf \widetilde{Q}_{[1,N-2i]}} & \displaystyle \frac{\pf \widetilde{Q}_{[1,N-2i+1]\backslash \{N-2j+1\}}}{\pf \widetilde{Q}_{[1,N-2i]}} 
\end{pmatrix}, \qquad & 1\le i < j\le \frac{N}2
\\[10mm]
\displaystyle
\begin{pmatrix}
\displaystyle-\frac{\pf \widetilde{Q}_{[1,N-2i]}}{\pf \widetilde{Q}_{[1,N-2i+2]}} & 0 \\[5mm]
0 & 1
\end{pmatrix}, & i=j \\[8mm]
\begin{pmatrix}
    0 & 0 \\
    0 & 0
\end{pmatrix}, & 1\le j < i\le \frac{N}2
\end{array}\right.,
\end{equation}
for each $1\leq i,j\leq N/2$.
\end{prop}
Recall Remark \ref{rem:skew nonuniq initial}, whereby the skew-biorthogonal polynomials which satisfy \eqref{eq:skew-biorhtog polyn appendix} are defined up to $N$ free parameters. Proposition \ref{prop:Upsilon2Phi} constructs an explicit family of these polynomials, $\widetilde{\Phi}_k$, which corresponds to the choices $\left[\mathcal{R}_{N}\right]_{2j-1,2j}=0$ and $\left[\mathcal{R}_{N}\right]_{2j,2j}=1$ for each $1\leq j\leq N/2$; as is made in Proposition \ref{prop:skew-borel}.

\section{Matching with Tracy--Widom transition probability}
\label{sec:TW-match}
Let us consider the following expression from the proof of the initial condition of Theorem 2.1 of \cite{tracy_integral_2008}. In the aforementioned work, the authors consider ASEP on $\mathbb{Z}$ is considered with a left-hoping rate of $q$ and a right-hoping right of $p$, subject to the normalization $p+q=1$. We will consider a modification of their integral expression where the normalization is relaxed and we set $p=1$. This is given by:
\begin{equation}
    \label{eq:orthog I-TW}
    \mathcal{I}^{\mathsf{TW}}_{x,y}(\sigma): = \oint_{\delta} \frac{\dd \xi_1}{2\pi\ii} \cdots \oint_{\delta} \frac{\dd \xi_N}{2\pi\ii} \prod_{1\leq i<j\leq N\atop \sigma_i>\sigma_{j}} f(\xi_{\sigma(i)},\xi_{\sigma(j)})\prod_{i=1}^N\xi_{\sigma(i)}^{x_{N-i+1}}\xi_i^{-y_{N-i+1}-1},
\end{equation}
where
\[f(\xi_i,\xi_j) = - \frac{1+q\xi_{i}\xi_{j}-(1+q)\xi_{i}}{1+q\xi_{i}\xi_{j}-(1+q)\xi_{j}},\]
and where $\delta$ is a contour surrounding the origin, $\xi_i=0$, with a sufficiently small radius to avoid all other singularities. Let us also emphasize that the convention of labeling individual particle coordinates in \eqref{eq:orthog I-TW} are consistent with the main body of this text. We note that the particle coordinates of \cite{tracy_integral_2008} are the reverse to our conventions. That is, we consider $x_1>\dots>x_N$ as opposed to $x_1<\dots<x_N$, which is the reason for $x_{N-i+1}$ appearing in \eqref{eq:orthog I-TW} rather than $x_i$.

For any configurations $x,y$, it is proven in \cite{tracy_integral_2008} that the following sum vanishes:
\begin{equation}
    \label{eq:orthog I-TW sum}
    \sum_{\sigma\in S_N\backslash\{\id\}} \mathcal{I}^{\mathsf{TW}}_{x,y}(\sigma) = 0.
\end{equation}
The proof of this statement is involved. It is made necessarily more complicated by the fact that individual terms in the sum \eqref{eq:orthog I-TW sum} are in general non-zero after the integrals are evaluated.

Now recall the integral expression from the proof of Lemma \ref{lm: orthogonality} given by \eqref{eq:orthog proof perm 1}. We will show here that this integral expression obeys the vanishing summation property \eqref{eq:orthog proof perm 2} by explicitly matching our expression with \eqref{eq:orthog I-TW}. Let us consider the integral expression \eqref{eq:orthog proof perm 1} under the change of integration coordinates $w_i\mapsto \frac{1-\xi_i}{1-q\xi_i}$. After straightforward algebraic manipulation, for any $\tau\in S_N$, \eqref{eq:orthog proof perm 1} is given by:
\begin{equation}
    \label{eq:orthog I-TW match 1}
    \mathcal{I}_{x,y}(\tau): = \oint_{\delta} \frac{\dd \xi_1}{2\pi\ii} \cdots \oint_{\delta} \frac{\dd \xi_N}{2\pi\ii} \prod_{1\leq i<j\leq N\atop \tau(i)>\tau(j)} f(\xi_{\tau(i)},\xi_{\tau(j)})\prod_{i=1}^N\xi_i^{x_{i}}\xi_{\tau(i)}^{-y_{i}-1}.
\end{equation}
From here, for each individual fixed $\tau$, we may simultaneously relabel the integration coordinates of \eqref{eq:orthog I-TW match 1} by $\xi_i\mapsto \xi_{\rho(i)}$ for the permutation defined by $\rho=\left(r(\tau)\right)^{-1}$, where $r(\tau)=(\tau(N),\dots,\tau(1))$ is permutation which is the reverse ordering of $\tau$. We emphasize here that we are allowed to freely change the order of integration after relabeling since all contours may be deformed to be the same. Under this procedure, the quantity \eqref{eq:orthog I-TW match 1} associated with $\tau$ is 
\begin{equation}
    \label{eq:orthog I-TW match 2}
    \mathcal{I}_{x,y}\left(\tau\right) = \oint_{\delta} \frac{\dd \xi_1}{2\pi\ii} \cdots \oint_{\delta} \frac{\dd \xi_N}{2\pi\ii} \prod_{1\leq i<j\leq N\atop \tau(i)>\tau(j)} f(\xi_{N-i+1},\xi_{N-j+1})\prod_{i=1}^N\xi_{r(\rho)(i)}^{x_{N-i+1}}\xi_{i}^{-y_{N-i+1}-1}.
\end{equation}
We now claim that the integrands of \eqref{eq:orthog I-TW} and \eqref{eq:orthog I-TW match 2} agree whenever their respective permutations are related by the following
\[\tau=r\left((r(\sigma))^{-1}\right).\]
This relationship is precisely equivalent to saying that $\tau$ is the \emph{mirror image} of $\sigma$. That is, if, as expressed a reduced word in simple transpositions $s_i$ which generate $S_N$, we have $\tau=s_{k_1}\cdots s_{k_n}$ for some sequence $k_1,\dots,k_n\in\{1\dots,N-1\}$. Then we have the corresponding expression $\sigma=s_{N-k_1+1}\cdots s_{N-k_n+1}$. Additionally, we note that this correspondence preserves the identity, i.e. $\tau=\id$ corresponds with $\sigma=\id$. Using this correspondence, we may conclude that
\begin{equation}
    \label{eq:orthog I-TW match 3}
     \sum_{\tau\in S_N\backslash\{\id\}} \mathcal{I}_{x,y}\left(\tau\right)=\sum_{\sigma\in S_N\backslash\{\id\}} \mathcal{I}^{\mathsf{TW}}_{x,y}(\sigma) = 0, 
\end{equation}
which is the desired result.

\bibliography{Library}{}

\newcommand{\etalchar}[1]{$^{#1}$}
\begin{thebibliography}{GdGMW25}

\bibitem[ABW23]{aggarwal_colored_2023}
A.~Aggarwal, A.~Borodin, and M.~Wheeler.
\newblock Colored fermionic vertex models and symmetric functions.
\newblock {\em Comm. Amer. Math. Soc.}, 3(08):400--630, 2023.
\newblock arXiv:2101.01605.

\bibitem[BBC20]{barraquand_half-space_2020}
G.~Barraquand, A.~Borodin, and I.~Corwin.
\newblock Half-space {Macdonald} processes.
\newblock {\em Forum Math. Pi}, 8:e11, 2020.
\newblock arXiv:1802.08210.

\bibitem[BBCS18]{baik_pfaffian_2018}
J.~Baik, G.~Barraquand, I.~Corwin, and T.~Suidan.
\newblock Pfaffian {Schur} processes and last passage percolation in a
  half-quadrant.
\newblock {\em Ann. Probab.}, 46(6), 2018.
\newblock arXiv:1606.00525.

\bibitem[BBCW18]{barraquand_stochastic_2018}
G.~Barraquand, A.~Borodin, I.~Corwin, and M.~Wheeler.
\newblock Stochastic six-vertex model in a half-quadrant and half-line open
  {ASEP}.
\newblock {\em Duke Math. J.}, 167(13), 2018.
\newblock arXiv:1704.04309.

\bibitem[BBF{\etalchar{+}}00]{benner_cholesky-like_2000}
P.~Benner, R.~Byers, H.~Fassbender, V.~Mehrmann, and D.~Watkins.
\newblock Cholesky-like factorizations of skew-symmetric matrices.
\newblock {\em Electron. Trans. Numer. Anal.}, 11:85--93, 2000.

\bibitem[BBNV18]{betea_free_2018}
D.~Betea, J.~Bouttier, P.~Nejjar, and M.~Vuletić.
\newblock The free boundary {Schur} process and applications {I}.
\newblock {\em Ann. Henri Poincaré}, 19(12):3663--3742, 2018.
\newblock arXiv:1704.05809.

\bibitem[BC24]{barraquand_markov_2024}
G.~Barraquand and I.~Corwin.
\newblock Markov duality and {Bethe} ansatz formula for half-line open {ASEP}.
\newblock {\em Prob. Math. Phys.}, 5(1):89--129, 2024.
\newblock arXiv:2212.07349.

\bibitem[BF08]{borodin_large_2008}
A.~Borodin and P.~L. Ferrari.
\newblock Large time asymptotics of growth models on space-like paths {I}:
  {PushASEP}.
\newblock {\em Electron. J. Probab.}, 13:1380--1418, 2008.
\newblock arXiv:0707.2813.

\bibitem[BFPS07]{borodin_fluctuation_2007}
A.~Borodin, P.~L. Ferrari, M.~Prähofer, and T.~Sasamoto.
\newblock Fluctuation {Properties} of the {TASEP}
  with {Periodic} {Initial} {Configuration}.
\newblock {\em J. Stat. Phys.}, 129(5):1055--1080, 2007.
\newblock arXiv:math-ph/0608056.

\bibitem[BL18]{baik_fluctuations_2018}
J.~Baik and Z.~Liu.
\newblock Fluctuations of {TASEP} on a ring in relaxation time scale.
\newblock {\em Commun. Pure Appl. Math.}, 71(4):747--813, 2018.

\bibitem[BP18]{borodin_higher_2018}
A.~Borodin and L.~Petrov.
\newblock Higher spin six vertex model and symmetric rational functions.
\newblock {\em Sel. Math. New Ser.}, 24(2):751--874, 2018.
\newblock arXiv:1601.05770.

\bibitem[BR01a]{baik_algebraic_2001}
J.~Baik and E.~M. Rains.
\newblock Algebraic aspects of increasing subsequences.
\newblock {\em Duke Math. J.}, 109(1), 2001.
\newblock arXiv:math/9905083.

\bibitem[BR01b]{baik_asymptotics_2001}
J.~Baik and E.~M. Rains.
\newblock The asymptotics of monotone subsequences of involutions.
\newblock {\em Duke Math. J.}, 109(2):205--281, 2001.
\newblock arXiv:math/9905084.

\bibitem[BR05]{borodin_eynardmehta_2005}
A.~Borodin and E.~M. Rains.
\newblock Eynard–{Mehta} {Theorem}, {Schur} {Process}, and their {Pfaffian}
  {Analogs}.
\newblock {\em J. Stat. Phys.}, 121(3):291--317, 2005.
\newblock arXiv:math-ph/0409059.

\bibitem[BW22]{borodin_coloured_2022}
A.~Borodin and M.~Wheeler.
\newblock Coloured stochastic vertex models and their spectral theory.
\newblock {\em Astérisque}, 2022.
\newblock arXiv:1808.01866.

\bibitem[dB55]{de_bruijn_multiple_1955}
N.~G. de~Bruijn.
\newblock On some multiple integrals involving determinants.
\newblock {\em J. Indian Math. Soc.}, 19:133--151, 1955.

\bibitem[dGE05]{GierEssler2005}
J.~de~Gier and F.~H.~L. Essler.
\newblock Bethe ansatz solution of the asymmetric exclusion process with open
  boundaries.
\newblock {\em Phys. Rev. Lett.}, 95:240601, 2005.

\bibitem[dGE06]{Gier_2006}
J.~de~Gier and F.~H.~L. Essler.
\newblock Exact spectral gaps of the asymmetric exclusion process with open
  boundaries.
\newblock {\em J. Stat. Mech.}, 2006(12):P12011, 2006.

\bibitem[dGE08]{Gier_2008}
J.~de~Gier and F.~H.~L. Essler.
\newblock Slowest relaxation mode of the partially asymmetric exclusion process
  with open boundaries.
\newblock {\em J. Phys. A Math. Theor.}, 41(48):485002, 2008.

\bibitem[dGE11]{PhysRevLett.107.010602}
J.~de~Gier and F.~H.~L. Essler.
\newblock Large deviation function for the current in the open asymmetric
  simple exclusion process.
\newblock {\em Phys. Rev. Lett.}, 107:010602, 2011.

\bibitem[DY25]{dimitrov_half-space_2025}
E.~Dimitrov and Z.~Yang.
\newblock Half-space {Airy} line ensembles, 2025.
\newblock arXiv:2505.01798.

\bibitem[GdGMW25]{garbali_symmetric_2025}
A.~Garbali, J.~de~Gier, W.~Mead, and M.~Wheeler.
\newblock Symmetric {Functions} from the {Six}-{Vertex} {Model} in
  {Half}-{Space}.
\newblock {\em Ann. Henri Poincaré}, 26(7):2557--2624, 2025.
\newblock arXiv:2312.14348.

\bibitem[GLMV12]{PhysRevLett.109.170601}
M.~Gorissen, A.~Lazarescu, K.~Mallick, and C.~Vanderzande.
\newblock Exact current statistics of the asymmetric simple exclusion process
  with open boundaries.
\newblock {\em Phys. Rev. Lett.}, 109:170601, 2012.

\bibitem[GM04]{Golinelli_2004}
O.~Golinelli and K.~Mallick.
\newblock Bethe ansatz calculation of the spectral gap of the asymmetric
  exclusion process.
\newblock {\em J. Phys. A Math. Gen.}, 37(10):3321, 2004.

\bibitem[GS92]{GwaSpohn1992}
L.-H. Gwa and H.~Spohn.
\newblock Bethe solution for the dynamical-scaling exponent of the noisy
  burgers equation.
\newblock {\em Phys. Rev. A}, 46:844--854, 1992.

\bibitem[He24]{he_boundary_2024}
J.~He.
\newblock Boundary current fluctuations for the half-space {ASEP} and
  six-vertex model.
\newblock {\em Proc. Lond. Math. Soc.}, 128(2):e12585, 2024.
\newblock arXiv:2303.16335.

\bibitem[He25]{he_shift_2025}
J.~He.
\newblock Shift invariance of half space integrable models.
\newblock {\em Probab. Theory Relat. Fields}, 2025.
\newblock arXiv:2205.13029.

\bibitem[Joh00]{johansson_shape_2000}
K.~Johansson.
\newblock Shape {Fluctuations} and {Random} {Matrices}.
\newblock {\em Commun. Math. Phys.}, 209(2):437--476, 2000.
\newblock arXiv:math/9903134.

\bibitem[JR21]{JohanssonRahman}
K.~Johansson and M.~Rahman.
\newblock Multitime distribution in discrete polynuclear growth.
\newblock {\em Comm. Pure Appl. Math.}, 74(12):2561--2627, 2021.

\bibitem[KG10]{kieburg_derivation_2010}
M.~Kieburg and T.~Guhr.
\newblock Derivation of determinantal structures for random matrix ensembles in
  a new way.
\newblock {\em J. Phys. A: Math. Theor.}, 43(7):075201, 2010.
\newblock arXiv:0912.0654.

\bibitem[Kim95]{Kim1995}
D.~Kim.
\newblock Bethe ansatz solution for crossover scaling functions of the
  asymmetric xxz chain and the kardar-parisi-zhang-type growth model.
\newblock {\em Phys. Rev. E}, 52:3512--3524, 1995.

\bibitem[Lig10]{liggett_continuous_2010}
T.~M. Liggett.
\newblock {\em Continuous time {Markov} processes : an introduction}.
\newblock American Mathematical Society, 2010.

\bibitem[Liu22]{Liu-multi-point}
Z.~Liu.
\newblock {Multipoint distribution of TASEP}.
\newblock {\em Ann. Probab.}, 50(4):1255 -- 1321, 2022.

\bibitem[LK06]{Lee_2006}
D.-S. Lee and D.~Kim.
\newblock Universal fluctuation of the average height in the early-time regime
  of one-dimensional kardar–parisi–zhang-type growth.
\newblock {\em J. Stat. Mech.}, 2006(08):P08014, 2006.

\bibitem[LP14]{Lazarescu_2014}
A.~Lazarescu and V.~Pasquier.
\newblock Bethe ansatz and q-operator for the open asep.
\newblock {\em J. Phys. A Math. Theor.}, 47(29):295202, 2014.

\bibitem[MQR21]{matetski_kpz_2021}
K.~Matetski, J.~Quastel, and D.~Remenik.
\newblock The {KPZ} fixed point.
\newblock {\em Acta Math.}, 227(1):115--203, 2021.
\newblock arXiv:1701.00018.

\bibitem[MR23]{matetski_tasep_2023}
K.~Matetski and D.~Remenik.
\newblock {TASEP} and generalizations: {Method} for exact solution.
\newblock {\em Probab. Theory Relat. Fields}, 185(1-2):615--698, 2023.
\newblock arXiv:2107.07984.

\bibitem[Oka19]{okada_pfaffian_2019}
S.~Okada.
\newblock Pfaffian {Formulas} and {Schur} {Q}-{Function} {Identities}.
\newblock {\em Adv. Math.}, 353:446--470, 2019.
\newblock arXiv:1706.01029.

\bibitem[OQR17]{ortmann_pfaffian_2017}
J.~Ortmann, J.~Quastel, and D.~Remenik.
\newblock A {Pfaffian} representation for flat {ASEP}.
\newblock {\em Comm. Pure Appl. Math.}, 70(1):3--89, 2017.
\newblock arXiv:1501.05626.

\bibitem[Pro16]{Prolhac2016}
S.~Prolhac.
\newblock Finite-time fluctuations for the totally asymmetric exclusion
  process.
\newblock {\em Phys. Rev. Lett.}, 116:090601, 2016.

\bibitem[Pro24]{Prolhac_2024}
S.~Prolhac.
\newblock Approach to stationarity for the kpz fixed point with boundaries.
\newblock {\em EPL}, 148(1):11002, 2024.

\bibitem[Rai00]{rains_correlation_2000}
E.~M. Rains.
\newblock Correlation functions for symmetrized increasing subsequences, 2000.
\newblock arXiv:math/0006097.

\bibitem[Sas05]{sasamoto_spatial_2005}
T.~Sasamoto.
\newblock Spatial correlations of the {1D} {KPZ} surface on a flat substrate.
\newblock {\em J. Phys. A: Math. Gen.}, 38(33):L549--L556, 2005.
\newblock arXiv:cond-mat/0504417.

\bibitem[Sch97]{schutz_exact_1997}
G.~M. Schütz.
\newblock Exact solution of the master equation for the asymmetric exclusion
  process.
\newblock {\em J. Stat. Phys.}, 88(1):427--445, 1997.
\newblock arXiv:cond-mat/9701019.

\bibitem[SI04]{sasamoto_fluctuations_2004}
T.~Sasamoto and T.~Imamura.
\newblock Fluctuations of a one-dimensional polynuclear growth model in a half
  space.
\newblock {\em J. Stat. Phys.}, 115(3-4):749--803, 2004.
\newblock arXiv:cond-mat/0307011.

\bibitem[Sim09]{Simon_2009}
D.~Simon.
\newblock Construction of a coordinate bethe ansatz for the asymmetric simple
  exclusion process with open boundaries.
\newblock {\em J. Stat. Mech.}, 2009(07):P07017, 2009.

\bibitem[Ste90]{stembridge_nonintersecting_1990}
J.~R. Stembridge.
\newblock Nonintersecting paths, pfaffians, and plane partitions.
\newblock {\em Adv. Math.}, 83(1):96--131, 1990.

\bibitem[TW08]{tracy_integral_2008}
C.~A. Tracy and H.~Widom.
\newblock Integral {Formulas} for the {Asymmetric} {Simple} {Exclusion}
  {Process}.
\newblock {\em Commun. Math. Phys.}, 279(3):815--844, 2008.
\newblock arXiv:0704.2633.

\bibitem[TW13]{tracy_bose_2013}
C.~A. Tracy and H.~Widom.
\newblock The {Bose} {Gas} and {Asymmetric} {Simple} {Exclusion} {Process} on
  the {Half}-{Line}.
\newblock {\em J. Stat. Phys.}, 150(1):1--12, 2013.
\newblock arXiv:1205.4054.

\bibitem[Ven15]{venkateswaran_symmetric_2015}
V.~Venkateswaran.
\newblock Symmetric and nonsymmetric {Koornwinder} polynomials in the \$q
  {\textbackslash}rightarrow 0\$ limit.
\newblock {\em J. Algebr. Comb.}, 42(2):331--364, 2015.
\newblock arXiv:1209.2933.

\bibitem[Zha24]{zhang_tasep_2024}
X.~Zhang.
\newblock {TASEP} in half-space, 2024.
\newblock arXiv:2409.09974.

\end{thebibliography}
\bibliographystyle{alphamod}

\end{document}